\documentclass{article}

\usepackage[utf8]{inputenc}

\usepackage{subfigure}
\usepackage[T1]{fontenc}

\usepackage{graphicx}
\usepackage{subfigure}
\usepackage{amsfonts}
\usepackage{amssymb} 
\usepackage{enumitem} 

\usepackage{amsmath, xypic}
\usepackage{amsthm}
\usepackage[normalem]{ulem}

\newtheorem{thm}{Theorem}

\newtheorem{defn}{Definition}
\newtheorem{lem}{Lemma}
\newtheorem{prop}{Proposition}
\newtheorem{cor}{Corollary}
\newtheorem{exm}{Example}

\newtheorem{rem}{Remark}

\newtheorem{conj}{Conjecture}

\newcommand{\ga}{\gamma}
\newcommand{\al}{\alpha}
\newcommand{\be}{\beta}
\newcommand{\de}{\delta}

\newcommand{\la}{\lambda}
\newcommand{\eps}{\varepsilon}

\newcommand{\CC}{\mathbb{C}}
\newcommand{\wit}{\widetilde}

\renewcommand{\leq}{\leqslant}
\renewcommand{\geq}{\geqslant}
\renewcommand{\phi}{\varphi}
\newcommand{\sm}{\setminus}
\newcommand{\da}{\!\downarrow\!} \newcommand{\ua}{\!\uparrow\!}
 
\newcommand{\GG}{\mathcal{G}}
\newcommand{\hex}{\includegraphics[scale=1]{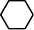}}

\newcommand{\wh}{\widehat}

\DeclareMathOperator{\wt}{wt}

\newlength{\cellsz}
\newcounter{cellsize}
\newcommand{\setcellsize}[1]{%
  \setcounter{cellsize}{#1}%
  \setlength{\cellsz}{\value{cellsize}\unitlength}}%
\setcellsize{12}%
\newcommand\cellify[1]{\def\thearg{#1}\def\nothing{}%
\hbox to 0pt{{\begin{picture}(\value{cellsize},\value{cellsize})
  \put(0,0){\line(1,0){\value{cellsize}}}
  \put(0,0){\line(0,1){\value{cellsize}}}
  \put(\value{cellsize},0){\line(0,1){\value{cellsize}}}
  \put(0,\value{cellsize}){\line(1,0){\value{cellsize}}} \end{picture}
\HanStaSte92}}
\vbox to \cellsz{ \vss \hbox to \cellsz{\HanStaSte92$#1$\HanStaSte92} \vss}}
\newcommand\tableau[1]{\vcenter{\vbox{\let\\\cr
\baselineskip -16000pt \lineskiplimit 16000pt \lineskip 0pt
\ialign{&\cellify{##}\cr#1\crcr}}}}
\newcommand\tabl[1]{\vtop{\let\\\cr
\baselineskip -16000pt \lineskiplimit 16000pt \lineskip 0pt
\ialign{&\cellify{##}\cr#1\crcr}}}

\title{On the matchings-Jack and hypermap-Jack conjectures for labelled matchings and star hypermaps}

\author{Andrei L. Kanunnikov\\
\small Department of Higher Algebra\\[-0.8ex]
\small Moscow State University\\[-0.8ex] 
\small Moscow, Russia.\\
\small\tt andrew.kanunnikov@gmail.com\\
\and
Valentin V. Promyslov\\
\small Department of Higher Algebra\\[-0.8ex]
\small Moscow State University\\[-0.8ex] 
\small Moscow, Russia.\\
\small\tt andrew.kanunnikov@gmail.com\\
\and
Ekaterina A. Vassilieva\\
\small LIX\\[-0.8ex]
\small Ecole Polytechnique\\[-0.8ex]
\small Palaiseau, France.\\
\small\tt ekaterina.vassilieva@lix.polytechnique.fr
}

\usepackage[backend=bibtex]{biblatex}
\begin{document}
\maketitle
\begin{abstract}
Introduced by Goulden and Jackson in their 1996 paper, the matchings-Jack conjecture and the hypermap-Jack conjecture (also known as the $b$-conjecture) are two major open questions relating Jack symmetric functions, the representation theory of the symmetric groups and combinatorial maps.  
They show that the coefficients in the power sum expansion of some Cauchy sum for Jack symmetric functions and in the logarithm of the same sum interpolate respectively between the structure constants of the class algebra and the double coset algebra of the symmetric group and between the numbers of orientable and locally orientable hypermaps. They further provide some evidence that these two families of coefficients indexed by three partitions of a given integer $n$ and the Jack parameter $\alpha$ are polynomials in $\beta = \alpha-1$ with non negative integer coefficients of combinatorial significance. This paper is devoted to the case when one of the three partitions is equal to $(n)$. We exhibit some polynomial properties of both families of coefficients and prove a variation of the hypermap-Jack conjecture and the matchings-Jack conjecture involving labelled hypermaps and matchings in some important cases. 
\end{abstract}

\section{Introduction}\label{sec:in}

\subsection{Cauchy sums for Jack symmetric functions} 
For any integer $n$ denote $\la=(\la_1,\la_2,\cdots,\la_p) \vdash n$ an integer partition of $|\la| = n $ with $\ell(\la)=p$ parts sorted in decreasing order. The set of all integer partitions (including the empty one) is denoted $\mathcal{P}$. If $m_i(\la)$ is the number of parts of $\la$ that are equal to $i$, then we may write $\la$ as $[1^{m_1(\la)}\,2^{m_2(\la)}\cdots]$ and define $z_\lambda =\prod_i i^{m_i(\lambda)}m_i(\lambda)!$, $Aut_\la = \prod_{i}m_i(\lambda)!$. When there is no ambiguity, the one part partition of integer $n$, $(n) = [n^1]$ is simply denoted $n$. Given a parameter $\alpha$, denote $p_\lambda(x)$ and  $J^\alpha_\la(x)$  the power sum and {\bf Jack symmetric function} indexed by $\la$ on $x=(x_1,x_2,\cdots)$. Jack symmetric functions are  orthogonal for the scalar product $\langle \cdot\, ,\cdot\rangle_\alpha$ defined by $\langle p_\la,p_\mu \rangle_\alpha = \alpha^{\ell(\la)}z_\la\delta_{\la,\mu}$. Denote $j_\la(\al)$ the value of the scalar product $\langle J^\al_\la,J^\al_\mu \rangle_\alpha = j_\la(\al)\delta_{\la,\mu}.$
This paper is devoted to the study of the following series for Jack symmetric functions introduced by Goulden and Jackson in \cite{GouJac96}.
\begin{align*}
\Phi(x,y,z,t,\al) &=  \sum_{\gamma \in \mathcal{P}}\frac{J^{\al}_\gamma(x)J^{\al}_\gamma(y)J^{\al}_\gamma(z)t^{|\ga|}}{\langle J^{\al}_\gamma,J^{\al}_\gamma\rangle_{\al}},\\
\Psi(x,y,z,t,\al) &= \al t\frac{\partial}{\partial t} \log\Phi(x,y,z,t,\al).
\end{align*}
More specifically, we focus on the coefficients $a_{\mu,\nu}^\la(\al)$ and $h_{\mu,\nu}^\la(\al)$ in their power sum expansions defined by:
\begin{align*}
 \Phi(x,y,z,t,\al) &=\sum_{n \geq 0}t^n\sum_{\la,\mu,\nu \vdash n}\al^{-\ell(\la)}z_\la^{-1}a_{\mu,\nu}^\la(\al)p_\la(x)p_\mu(y)p_\nu(z),\\
 \Psi(x,y,z,t,\al) &= \sum_{n \geq 0}t^n\sum_{\la,\mu,\nu \vdash  n}h_{\mu,\nu}^\la(\al)p_\la(x)p_\mu(y)p_\nu(z).
\end{align*}
Goulden and Jackson conjecture that both the $a_{\mu,\nu}^\la(\al)$ and $h_{\mu,\nu}^\la(\al)$ may have a strong combinatorial interpretation. In particular thanks to exhaustive computations of the coefficients they show that the $a^{\la}_{\mu,\nu}(\alpha)$ and $h_{\mu,\nu}^\la(\al)$ are polynomials in $\beta = \al -1$ with non negative integer coefficients and of degree at most $n-\min\{\ell(\mu),\ell(\nu)\}$  for all $\la,\mu,\nu\vdash n \leq 8$ . They conjecture this property for arbitrary $\la,\mu,\nu$ and prove it  in the limit cases $\la =[1^n]$ and $\la = [1^{n-2}2^1]$. Moreover, for $\la,\mu,\nu$ partitions of a given integer $n$, they make the stronger suggestion that the coefficients in the powers of $\beta$ in $a^{\la}_{\mu,\nu}(\alpha)$ count certain sets of {\em matchings} i.e. fixpoint-free involutions of the symmetric group on $2n$ elements (the {\em matchings-Jack conjecture}) and that the coefficients in the powers of $\beta$ in $h_{\mu,\nu}^\la(\al)$ count certain sets of {\em locally orientable hypermaps} i.e.\ connected bipartite graphs embedded in a locally orientable surface (the {\em hypermap-Jack conjecture} or {\em $b$-conjecture}). We look at the case $\mu = (n)$ and study a variant of the conjectures involving labelled objects defined in the following section.

In this specific case the two conjectures are related. Indeed, one has:

\begin{align*}
\Psi(x,y,z,t,\al) = \al t\frac{\partial}{\partial t}&\sum_{k\geq 0}(-1)^{k+1} \sum_{\mu^1\cdots\mu^k \in \mathcal{P\setminus{\emptyset}}}\\
&\prod_{i}{p_{\mu^i}(y)}t^{|\mu_i|}\sum_{\la,\nu \vdash |\mu^i|}z_\la^{-1}\al^{-\ell(\la)}a^{\la}_{\mu^i\nu}(\al)p_\la(x)p_\nu(z),
\end{align*}

\noindent which implies that

\begin{align*}
[p_{n}(y)]\Psi(x,y,z,t,\al) = \al t\frac{\partial}{\partial t}t^{n}\sum_{\la,\nu \vdash n}z_\la^{-1}\al^{-\ell(\la)}a^{\la}_{n\nu}(\al)p_\la(x)p_\nu(z).
\end{align*}

\noindent As a result, the following formula holds:

\begin{align}
\label{eq : ha}
h^{\la}_{n \nu}(\al) = \al n z_\la^{-1}\al^{-\ell(\la)}a^{\la}_{n\nu}(\al).
\end{align}

However, because of the difference in the combinatorial objects involved in the two conjectures, they are not equivalent.

\begin{rem}
According to the definition of the coefficients $a^{\la}_{\mu,\nu}(\al)$, Equation~(\ref{eq : ha}) can be rewritten as $h^{\la}_{n \nu}(\al) = a^{n}_{\la\nu}(\al)$.
\end{rem}

\subsection{Combinatorial background}
\subsubsection{Matchings}

Given a non-negative integer $n$ and a set of $2n$ vertices $V_n= \{1,\wh{1},\cdots, n,\wh{n}\}$ we call a {\bf matching} on $V_n$ a set of $n$ non-adjacent edges such that all the vertices are the endpoint of one edge. Given two matchings $\de_1$ and $\de_2$, the graph induced by the vertices in $V_n$ and the $2n$ edges of $\de_1 \cup \de_2$ is composed of cycles of even length $2\eps_1, 2\eps_2,\cdots,2\eps_p$ for some $\eps \vdash n$ and we denote $\Lambda(\de_1,\de_2) = \eps$. For a partition $\la=(\la_1,\cdots,\la_p)$ of $n$, define two canonical matchings ${\bf g}_n$ and ${\bf b}_\la$. The matching ${\bf g}_n$ is obtained by drawing a gray colored edge between vertices $i$ and $\wh{i} = {\bf g}_n(i)$ for $i=1,\cdots,n$. The matching ${\bf b}_\la$ is obtained by drawing a black colored edge between vertices $\wh{i}$ and ${\bf b}_\la(\wh{i})$ for $i=1,\cdots,n$ where ${\bf b}_\la(\wh{i}) = 1+\sum_{k=1}^{l-1} \la_k$ if $i = \sum_{k=1}^l \la_k$ for some $1\leq l \leq p$ and ${\bf b}_\la(\wh{i}) = i+1$ otherwise. Obviously $\Lambda({\bf g}_n,{\bf b}_\la) = \la$. Denote by $\mathcal{G}^\la_{\mu,\nu}$ the set of all the matchings $\delta$ on $V_n$ such that $\Lambda({\bf g}_n,\de) = \mu$ and $\Lambda({\bf b}_\la,\de) = \nu$. A matching $\delta$ in which all edges are of kind $\{i,\widehat{j}\}$ is called {\bf bipartite}.
This graph model is closely linked to the connection coefficients of two classical algebras.
\begin{itemize}
\item The {\bf class algebra} is the center of the group algebra $\CC S_n$. For $\la\vdash n$, denote by $C_\la$ the formal sum 
of all permutations with cycle type $\la$. The set $\{C_\la \mid \la\vdash n\}$ is a basis of the class algebra.

\item The {\bf double coset algebra} is the Hecke algebra of the Gelfand pair $(S_{2n},B_n)$ where $B_n$ is the centralizer 
of  $f_\star = (1\widehat{1})(2\widehat{2})\ldots(n\widehat{n})$ in $S_{2n}$. For $\la\vdash n$, denote by $K_\lambda$ the double coset consisting of all the permutations $\omega\in S_{2n}$ such that 
$f_\star \circ\omega\circ f_\star\circ\omega^{-1}$ has a cycle type $\la\la$.
The set $\{K_\la \mid \la\vdash n\}$ is a basis of the double coset algebra.
\end{itemize}
We define the  connection coefficients of these algebras by
\begin{equation}
\label{eq : defconcoef} c^{\la}_{\mu\nu} = [C_{\la}]C_{\mu}C_{\nu} \;\; \text{ and } \;\; b^{\la}_{\mu\nu} = [K_{\la}]K_{\mu}K_{\nu}, \qquad \la,\mu,\nu\vdash n.
\end{equation}

\begin{prop}[\cite{GouJac96}, proposition 4.1; \,\cite{HanStaSte92}, Lemma 3.2] \label{combint} 
Using the notation above:
$$b^\la_{\mu\nu}/|B_n|=|\GG^\la_{\mu\nu}| \text{ and } 
c^\la_{\mu\nu}=|\{\de\in \GG^\la_{\mu\nu} \mid \de \text{ is bipartite}\}|.$$
\end{prop}
This paper is focused on the case $\mu = (n)$. For $\la, \nu \vdash n$ we consider the set of {\bf labelled matchings} $\wit{\GG}^{\la}_{\nu}$, i.e. the tuples $\de = (\bar{\de}, \sigma_2, \cdots)$ composed of a matching $\bar{\delta} \in \GG^{\la}_{n,\nu}$ and a permutation $\sigma_i$ on the $m_i(\nu)$ cycles of length $2i$ in ${\bf b}_\la\cup\bar{\delta}$ for all $i>1$ (the cycles of length $2$ are not labelled). Clearly, 
$|\wit{\GG}^\la_{\nu}|=\frac{Aut_\nu}{m_1(\nu)!}|\GG^\la_{n,\nu}|$. 
\begin{exm} Figure \ref{la_g} depicts a labelled matching from $\wit{\GG}^{(4,2)}_{[2^3]}$ with three labelled squares:  $\wh{1}2\wh{3}4$, $\wh{2}3\wh{6}5$ and $\wh{4}1\wh{5}6$.

\end{exm}
\begin{figure}[htbp]
\begin{center}
 \includegraphics[scale=0.33]{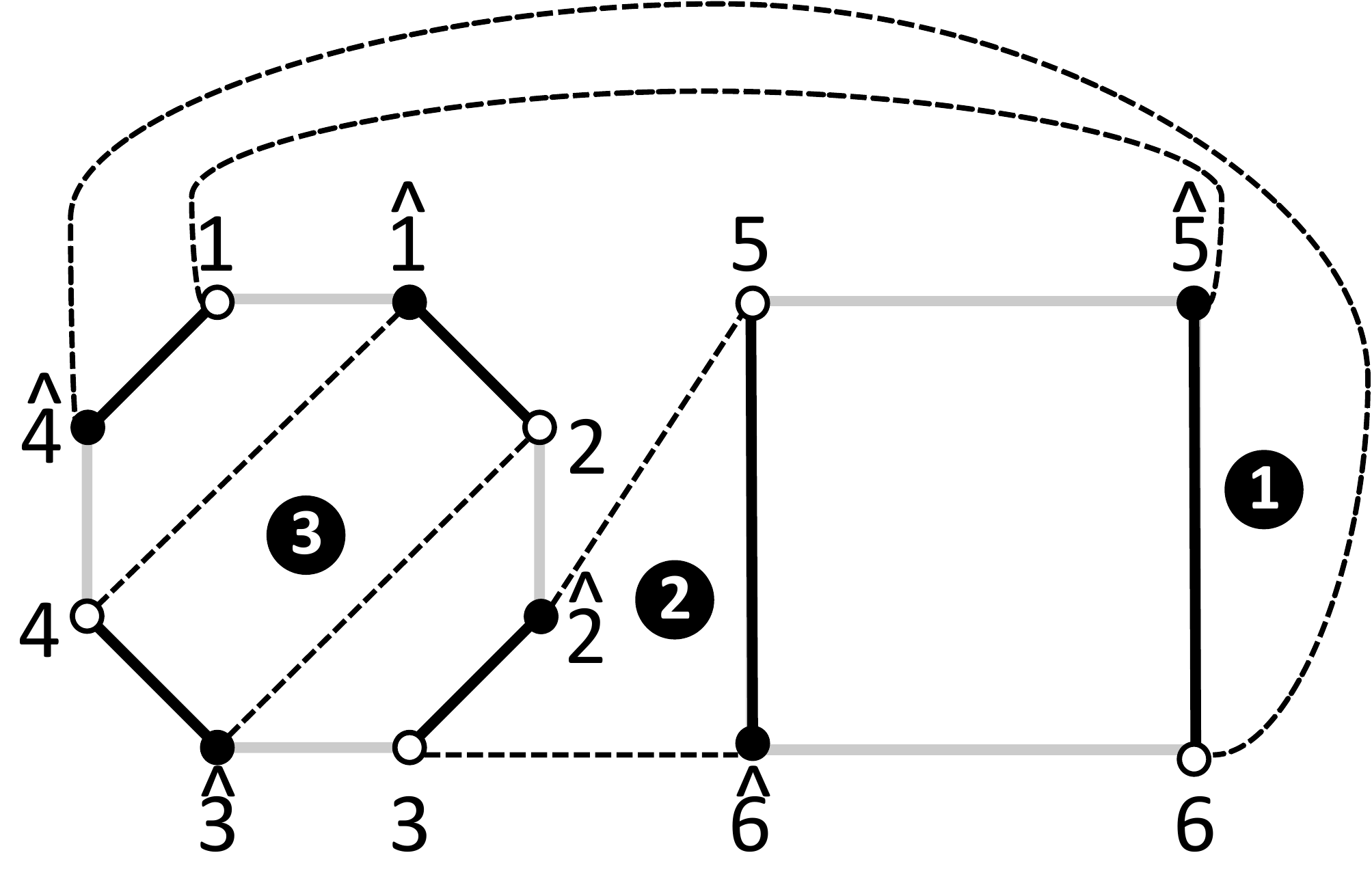}\caption{A labelled matching from $\wit{\GG}^{(4,2)}_{[2^3]}$ with three labelled squares.}
 \label{la_g}
 \end{center}
 \end{figure}


\subsubsection{Locally orientable hypermaps}
One can define {\bf locally orientable hypermaps} as connected bipartite graphs with black and white vertices. Each edge is composed of two half-edges both connecting the two incident vertices. This graph is embedded in a locally orientable surface such that if we cut the graph from the surface, the remaining part consists of connected components called  faces or cells, each homeomorphic to an open disk. The map can also be represented (not in a unique way) as a {\bf ribbon graph} on the plane keeping the incidence order of the edges around each vertex. In such a representation,  two half-edges can be parallel or cross in the middle. We say that the hypermap is {\bf orientable} if it is embedded in an orientable surface (sphere, torus, bretzel, \ldots). Otherwise the hypermap is embedded in a non orientable surface (projective plane, Klein bottle, \dots) and is said to be {\bf non-orientable}. In this paper we consider only {\bf rooted hypermaps}, i.e. hypermaps with a distinguished half-edge. More details about hypermaps can be found in \cite{lanZvo04}.\\
The {\bf degree} of a face, a white vertex or a black vertex is the number of edges incident to it. Hypermaps are also classified according to a triple of integer partitions that give respectively the degree distribution of the faces, the degree distribution of the white vertices, and the degree distribution of the black vertices. For any integer $n$ and partitions $\la, \mu$ and $\nu$ of $n$,  denote $\mathcal{L}^\la_{\mu,\nu}$ and $l^\la_{\mu,\nu}$ (resp. $\mathcal{M}^\la_{\mu,\nu}$ and $m^\la_{\mu,\nu}$) the set and the number of locally orientable hypermaps (resp. orientable) of face degree distribution $\la$, white vertices degree distribution $\mu$ and black vertices degree distribution $\nu$. When $\mu = (n)$, the hypermap has only one white vertex. We call it a {\bf star hypermap}. 

\begin{rem}
Star hypermaps are in natural bijection with {\bf unicellular hypermaps}, i.e. hypermaps with only one face but an arbitrary number of white vertices. While unicellular hypermaps received a more significant attention in previous papers (see e.g. \cite{GouSch98,MorVas13,Vas13,Vas17}), it is much more convenient to work with multicellular star hypermaps for our purpose. 
\end{rem}

\begin{exm}
Two star hypermaps are depicted on Figure~\ref{MapEx}. The leftmost (resp. rightmost) one is orientable (resp. non-orientable)  and has a face degree distribution $\la = (4,1,1)$ (resp. $\la = (4)$).
 \end{exm}
\begin{figure}[htbp]
\begin{center}
 \includegraphics[scale=0.3]{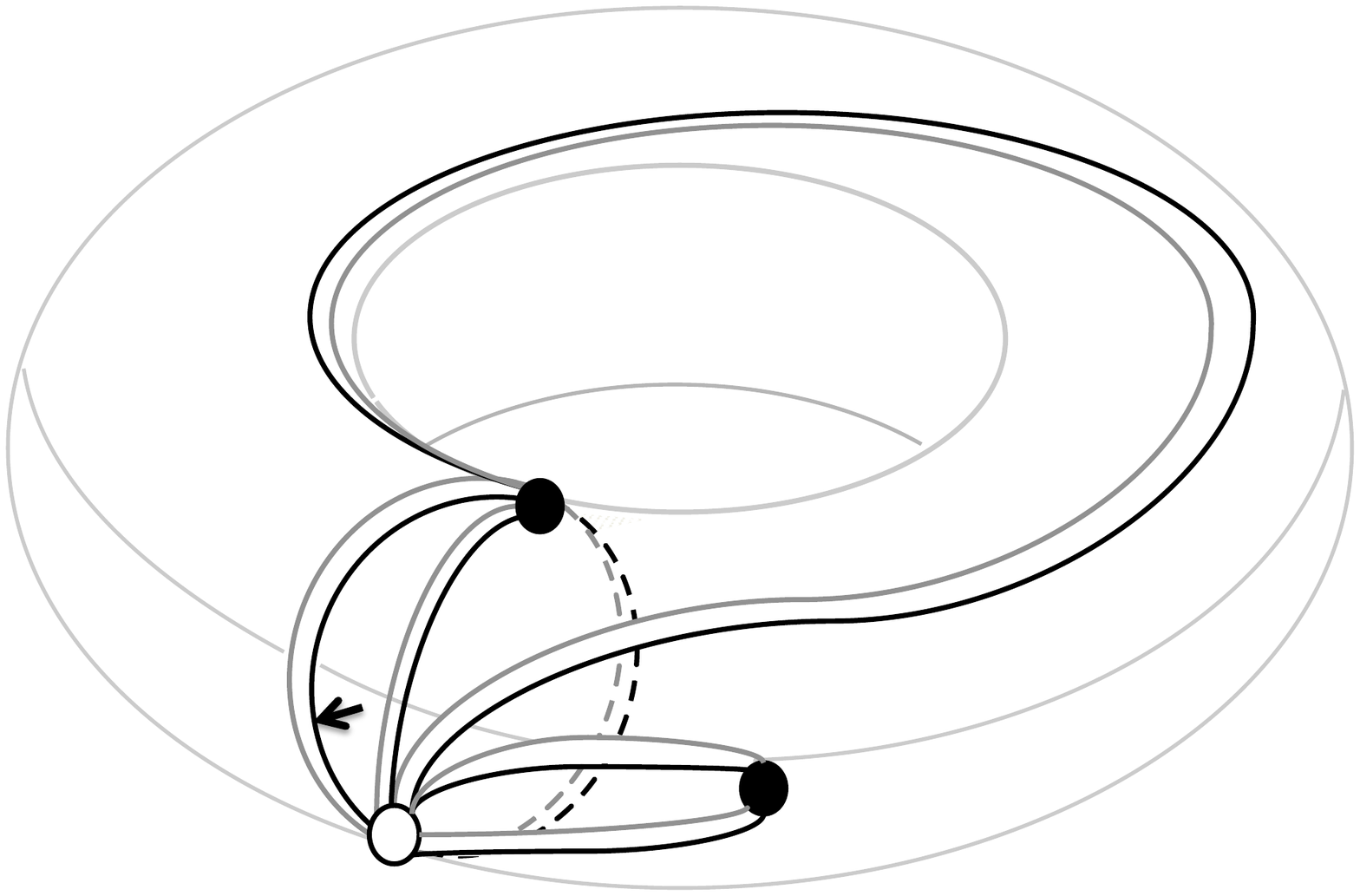}\hfill \includegraphics[scale=0.38]{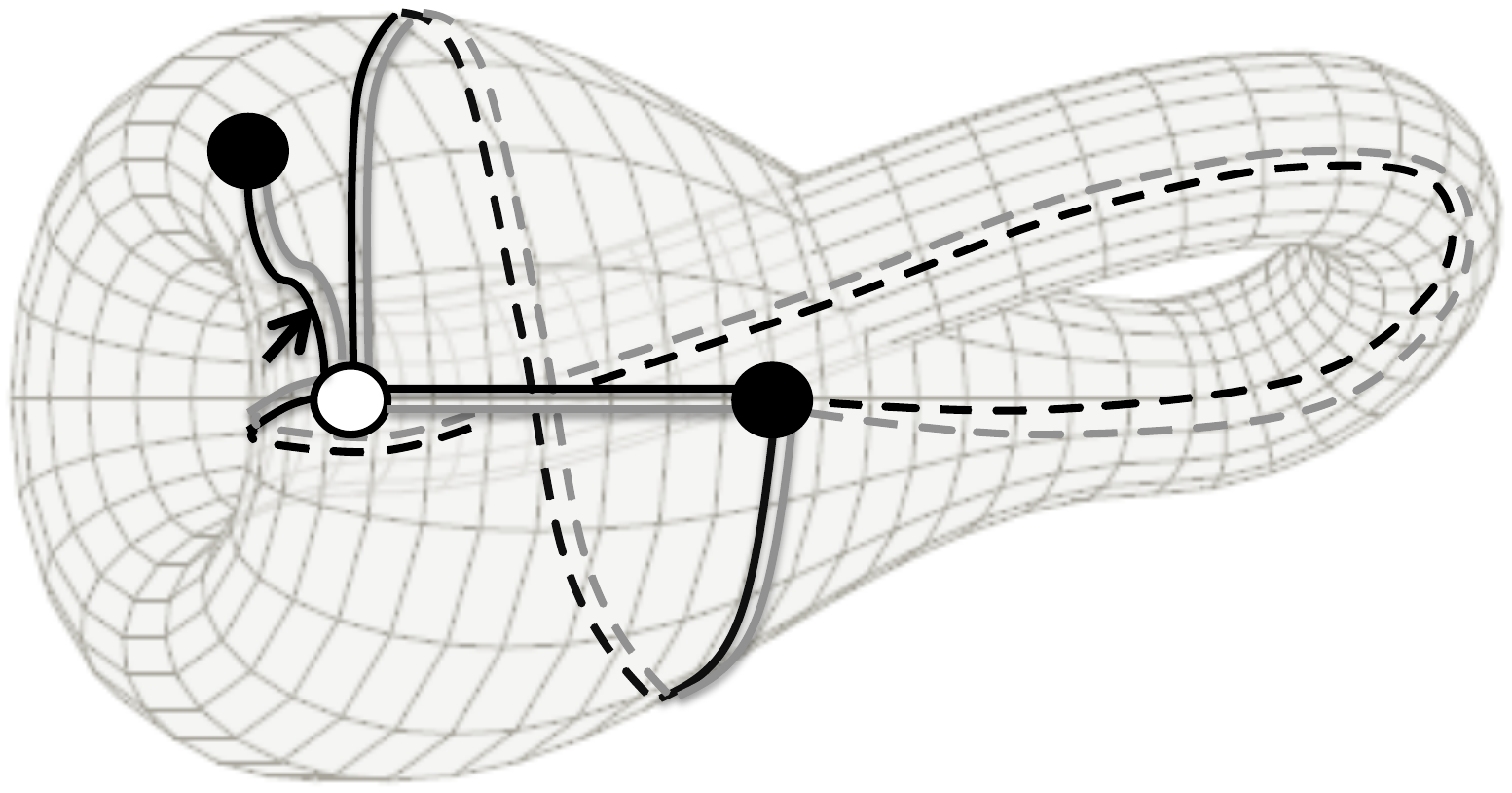}\caption{Examples of star hypermaps embedded in the torus (left) and the Klein bottle (right).}
\label{MapEx} \end{center}
 \end{figure}
 \begin{rem}\label{rem : 2m}
When $\nu = [2^m]$ for some integer $m$, hypermaps reduce to classical (non-bipartite maps). The reduction is obtained by connecting the two edges incident to each black vertex and removing all the black vertices.   Figure~\ref{fig : nbmap} gives an example of a non-bipartite star maps represented as a ribbon graph. 
\begin{figure}[htbp]
\begin{center}
 \includegraphics[scale=0.38]{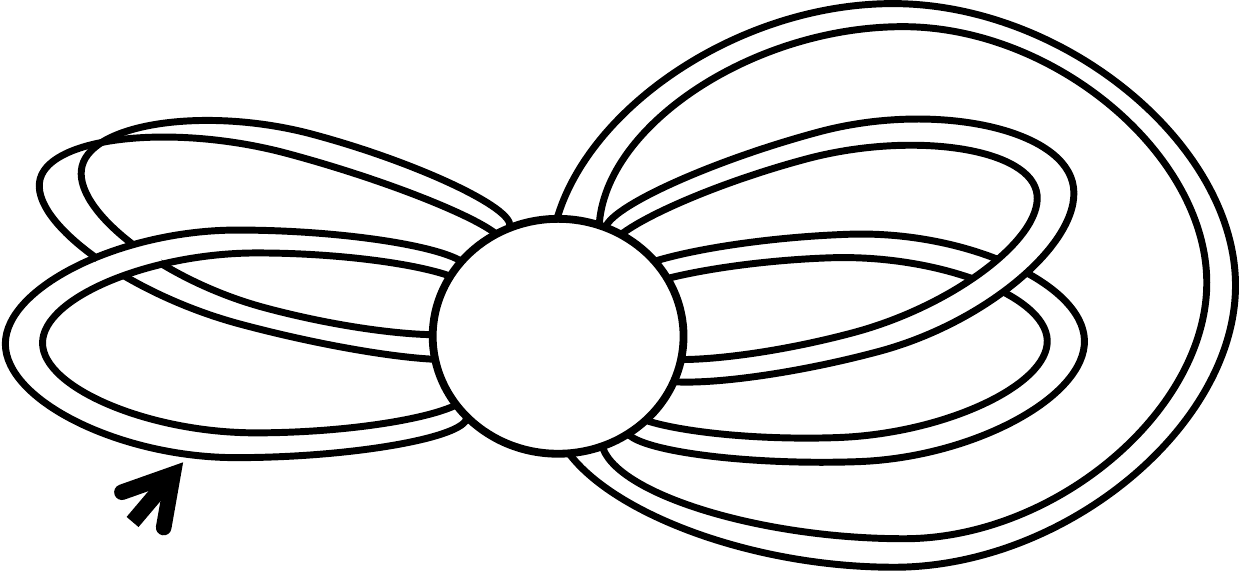}
 \caption{Example a non-bipartite star map.}
\label{fig : nbmap} \end{center}
 \end{figure}
 \end{rem}

In this paper, we consider {\bf labelled star hypermaps}, i.e. hypermaps where the $\ell(\nu)$ black vertices are labelled by integers $1,\cdots,\ell(\nu)$ such that the vertex incident to the root is labelled $1$. Denote $d_i$ the degree of the black vertex indexed $i$, we further assume that the edges incident to the black vertex indexed $i$ are labelled with $$\sum_{1\leq j < i}d_j+1,\sum_{1\leq j < i}d_j+2,\cdots,\sum_{1\leq j \leq i}d_j$$ with the additional condition that the root edge (incident to the black vertex indexed $1$) is labelled $1$. For $\la, \nu \vdash n$, denote $\widetilde{\mathcal{L}}^\la_{\nu}$ the set of labelled star hypermaps with face degree distribution $\la$ and black vertices degree distribution $\nu$. We focus on the special case $\nu =[k^m]$. Clearly, 
$$|\widetilde{\mathcal{L}}^\la_{[k^m]}| = (m-1)!k!^{m-1}(k-1)!l^\la_{n,[k^m]} = \frac{m!k!^m}{n}l^\la_{n,[k^m]}.$$
\begin{exm}
The two ribbon graphs depicted on Figure~\ref{MapExLab} are labelled star hypermaps with $\nu = [3^m]$ (left-hand side) and $\nu = [2^m]$ (right-hand side). \end{exm}
\begin{figure}[htbp]
\begin{center}
 \includegraphics[scale=0.42]{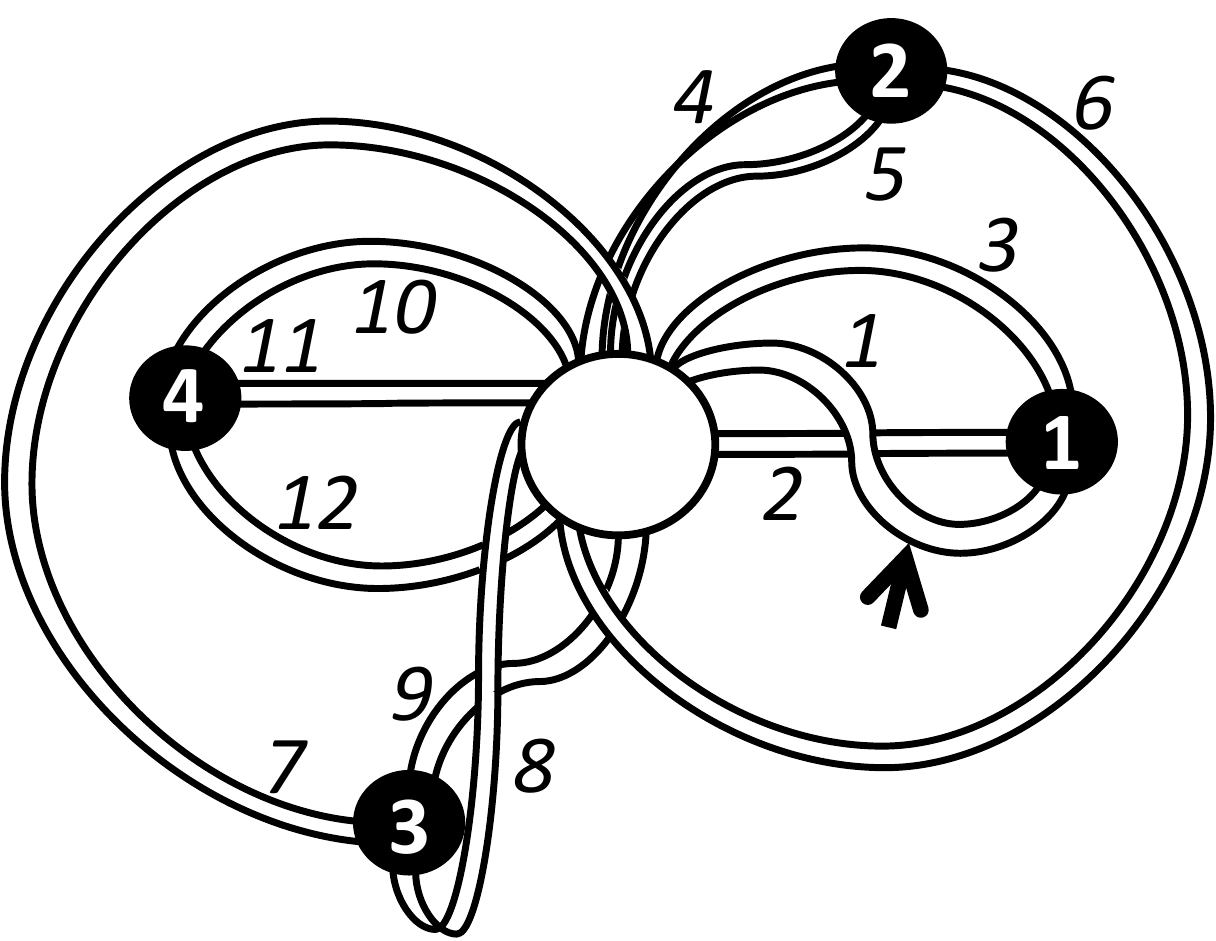}\qquad \includegraphics[scale=0.38]{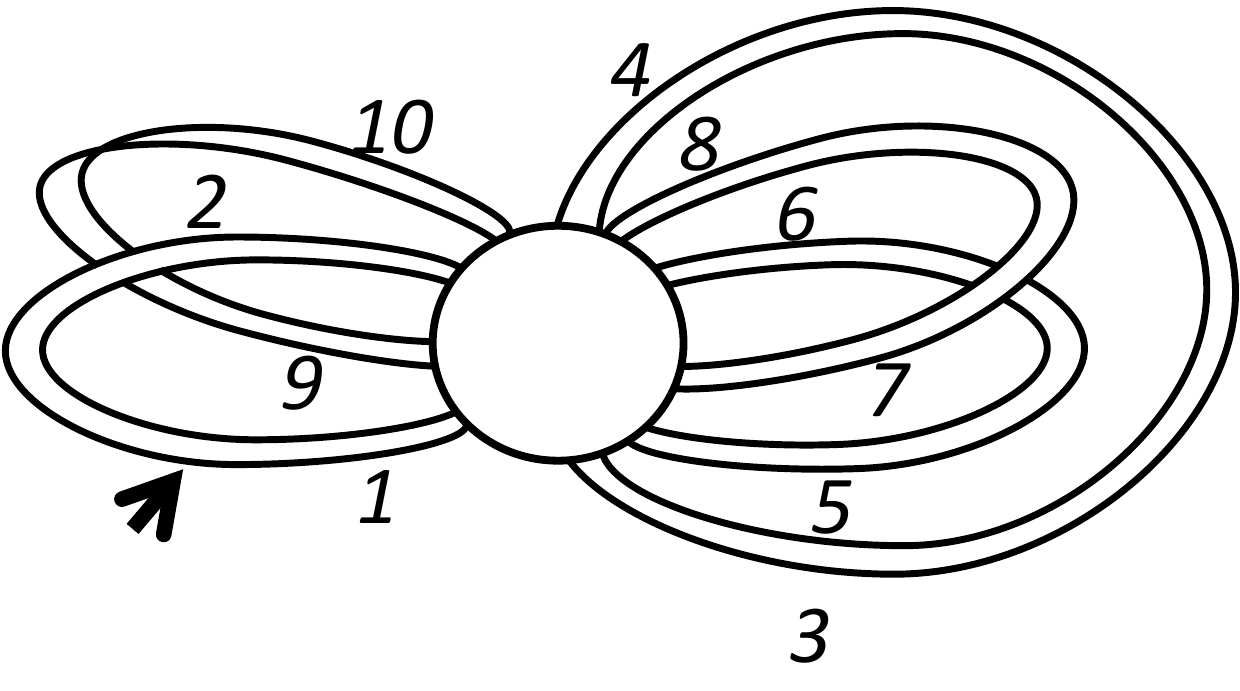}\caption{Examples of labelled star hypermaps.}
\label{MapExLab} \end{center}
 \end{figure}

\subsection{Relation to the matchings-Jack and hypermap-Jack conjectures}

One can show (see e.g. \cite{GouJac96}, \cite{HanStaSte92}, \cite{Vas15a}) that the numbers $c^\la_{\mu,\nu}$, $b^\la_{\mu,\nu}$ and numbers of hypermaps are linked to the coefficients $a^{\la}_{\mu,\nu}(\alpha)$ and $h^{\la}_{\mu,\nu}(\alpha)$:
\begin{align*} a_{\mu,\nu}^\la(1) &= c_{\mu,\nu}^\la \quad \text{ and } 
 \quad a_{\mu,\nu}^\la(2) = \frac{1}{|B_n|}b^{\la}_{\mu,\nu},\\
 h_{\mu,\nu}^\la(1) &= m_{\mu,\nu}^\la \;\; \text{ and } 
 \quad h_{\mu,\nu}^\la(2) = l^\la_{\mu,\nu}.\end{align*}
 
For general values of $\al$ Goulden and Jackson conjecture the following relations between $a^{\la}_{\mu,\nu}(\alpha)$ (resp. $h^{\la}_{\mu,\nu}(\alpha)$) and sets of matchings (resp. hypermaps).

\begin{conj}[Matchings-Jack conjecture, \cite{GouJac96}, conjecture 4.2] \label{conjM} For $\la,\mu,\nu\vdash n$ there exists a function 
$\wt_\la\colon \mathcal{G}^\la_{\mu,\nu}\to \{0,1,\cdots, n-\min\{\ell(\mu),\ell(\nu)\}\}$ such that 
$$a_{\mu,\nu}^\la(\be+1) = \sum_{\de\in \mathcal{G}^\la_{\mu,\nu}}\be^{\wt_{\la}(\de)}$$
and $\wt_\la(\de)=0 \iff \de$ is bipartite.
\end{conj} 
\begin{conj}[Hypermap-Jack conjecture, \cite{GouJac96}, conjecture 6.3] \label{conjH} For $\la,\mu,\nu\vdash n$ there exists a function 
$\vartheta \colon \mathcal{L}^\la_{\mu,\nu}\to\{0,1,\cdots, n-\min\{\ell(\mu),\ell(\nu)\}\}$ such that
$$h_{\mu,\nu}^\la(\be+1) = \sum_{M \in \mathcal{L}^\la_{\mu,\nu}}\be^{\vartheta(M)}$$
and $\vartheta(M)=0 \iff M$ is orientable.
\end{conj}



\section{Main results}

We use linear operators for Jack symmetric functions to derive a new formula for the coefficients $a^{\la}_{n,\nu}(\alpha)$ for general $\la$ and $\nu$ which shows their polynomial properties and, as a consequence to Equation~(\ref{eq : ha}),  the polynomial properties of the coefficients $h^{\la}_{n,\nu}(\alpha)$. Making this formula explicit and using some bijective constructions for labelled star hypermaps and matchings, we show a variant of the matchings-Jack and the hypermap-Jack conjectures for labelled objects in some important cases.\\ 
Denote $D_\al$, the Laplace-Beltrami operator. Namely,
$$D_\al = \frac{\al}{2}\sum_{i}x_i^2\frac{\partial^2}{\partial x_i^2}+\sum_{i\neq j}\frac{x_ix_j}{x_i-x_j}\frac{\partial}{\partial x_i}$$
and let $\Delta$ and $\{\Omega_{k}\}_{k \geq 1}$ be the operators on symmetric functions defined by  $$\Delta = [D_\al,[D_\al,p_1/\al]],\;\;\;\; \Omega_{1} = [D_\al,p_1/\al],\;\;\;\; \Omega_{k+1} = [\Delta,\Omega_{k}].$$

\noindent Our main result can be stated as follows
\begin{thm}
\label{thm : main}
For any integer $n$ and  $\la, \nu \vdash n$, the coefficients  $a^{\la}_{n,\nu}(\al)$ verify:
\begin{align}
\label{eq : main}
Aut_\nu \sum_{\la \vdash n}z_\la^{-1}\al^{-\ell(\la)}a^{\la}_{n,\nu}(\al)p_\la = \frac{1}{\prod_{i \geq 1} \nu_i!}\left(\prod_{i\geq 2}\Omega_{\nu_i}\right)\Delta^{\nu_1-1}(p_1/\al).
\end{align}
\end{thm}
As a consequence to Theorem \ref{thm : main}, we have the following polynomial properties.
\begin{cor}
\label{thm : cor1}
For $\la,\nu \vdash n$, $Aut_\nu|C_\la|a^{\la}_{n,\nu}(\al)$ and $Aut_\nu{\prod_{i \geq 1} \nu_i!}h^{\la}_{n,\nu}(\al)$ are polynomials in $\al$ with integer coefficients of respective degrees at most $n-\ell(\nu)$ and $n+1-\ell(\la)-\ell(\nu)$.
\end{cor}

Explicit computation of operators $\Omega_{k}$ for $k=1,2,3$ and $\Delta$, allows us to show:

\begin{thm} \label{thm : A} For $\la, \nu \vdash n$, define $\wit{a}^\la_{n,\nu}(\al)= ({Aut_\nu}/{m_1(\nu)!})a^\la_{n,\nu}(\al).$ If all but one part of $\nu$ are less or equal to $3$, there exists a function $\wt \colon \wit{\GG}^\la_{\nu}\to \{0,1,2,\cdots,n-\ell(\nu)\}$ such that
$$\wit{a}^\la_{n,\nu}(\be+1)=\sum_{\de \in \wit{\GG}^\la_{\nu}} \be^{\wt(\de)}$$ and $\wt(\delta) = 0 \iff \de$ is bipartite. 
\end{thm}

\begin{thm}
For $\la \vdash n$ and integers $k$ and $m$ with $n=km$, define $\widetilde{h}_{n,[k^m]}^\la(\alpha) = \frac {m!k!^m} {n}  {h}_{n,[k^m]}^\la(\alpha).$ For all $k \in \{1,2,3,n\}$ there exists a function
$\vartheta \colon \widetilde{\mathcal{L}}^\la_{[k^m]} \to \{0,1,2,\cdots, n+1-\ell(\la)-m\}$ such that
$$\widetilde{h}_{n, [k^m]}^\la(\be+1) = \sum_{M \in \widetilde{\mathcal{L}}^\la_{[k^m]}}\be^{\vartheta(M)}$$
and $\vartheta(M) = 0 \iff M$ is orientable.
\label{thm : H}
\end{thm}

\begin{rem}
The focus on labelled objects and this variant of the matchings-Jack and the hypermap-Jack conjectures is motivated by the coefficients $Aut_\nu$ and ${\prod_{i \geq 1} \nu_i!}$ that appear in Equation~(\ref{eq : main}).
\end{rem}
\begin{rem}
One can notice that we consider all the partitions $\nu$ with any number of parts $1,2$ and $3$ in Theorem \ref{thm : A} but only the partitions of the type $\nu = [k^m]$ in Theorem \ref{thm : H}. This is due to the existence of a distinguished (root) edge in the star hypermaps of the later theorem that prevents the extension of our methods to less symmetric cases.   
\end{rem}
\section{Background and prior works}
The following sections provide some relevant background regarding the computation of $c^\la_{\mu,\nu}$, $b^\la_{\mu,\nu}$, $m^\la_{\mu,\nu}$ and $l^\la_{\mu,\nu}$, i.e. the computation of the coefficients $a^\la_{\mu,\nu}(\al)$ and $h^\la_{\mu,\nu}(\al)$ in the classical cases $\al \in \{1,2\}$ and known results for these coefficients with general $\al$.

\subsection{Classical enumeration results for matchings and hypermaps}
Except for special cases no closed formulas are known for the coefficients $c^\la_{\mu,\nu}$, $b^\la_{\mu,\nu}$, $m^\la_{\mu,\nu}$ and $l^\la_{\mu,\nu}$. Prior works on the subjects are usually focused on the case $\la = (n)$. With this particular parameter, one has $c^n_{\mu,\nu} = m^n_{\mu,\nu}$ and $l^n_{\mu,\nu} = \widetilde{b}^n_{\mu,\nu}$. Using an inductive argument B\'{e}dard and Goupil  \cite{BedGou92} first found a formula for $c^n_{\lambda,\mu}$ in the case $\ell(\la)+\ell(\mu)=n+1$, which was later reproved by Goulden and Jackson \cite{GouJac92} via a bijection with a set of ordered rooted bicolored trees. Later, using characters of the symmetric group and a combinatorial development, Goupil and Schaeffer \cite{GouSch98} derived an expression for the connection coefficients $c^n_{\mu,\nu}$ in the general case as a sum of positive terms (see Biane \cite{Bia04} for a succinct algebraic derivation; and Poulalhon and Schaeffer \cite{PouSch00}, and Irving \cite{Irv06} for further generalisations). 
Closed form formulas of the expansion of the generating series for the $c^{n}_{\lambda,\mu}$ and $b^{n}_{\lambda,\mu}$ and their generalisations in the monomial basis were provided by Morales and Vassilieva and Vassilieva using bijective constructions for hypermaps in \cite{MorVas13}, \cite{Vas17} and \cite{Vas13}. Equivalent results using purely algebraic methods are provided in \cite{Vas15b}.

\subsection{Prior results on the Matchings-Jack conjecture and the Hypermap-Jack conjecture}
\label{sec : MHJC}
While the matchings-Jack and the hypermap-Jack conjecture are still open in the general case, some special cases and weakened forms have been solved over the past decade. In particular, Brown and Jackson in \cite{BroJac07} prove that for any partition $\mu \vdash 2m$, $\sum_{\la}h^\la_{\mu, [2^m]}(\beta +1)$ verifies a weaker form of the hypermaps-Jack conjecture. Later on, in his PhD thesis (\cite{Lac09}), Lacroix defines a {\it measure of non-orientability} $\vartheta$ for hypermaps and focuses on a stronger form of the result of Brown and Jackson. He shows that $$\sum_{\ell(\la)=r}h^\la_{\mu, [2^m]}(\beta +1) = \sum_{M \in \bigcup_{\ell(\la)=r} \mathcal{L}^\la_{\mu, [2^m]}}\beta^{\vartheta(M)}.$$ In particular he proves the hypermap-Jack conjecture for $h^n_{\mu, [2^m]}(\beta +1) = h^\mu_{n, [2^m]}(\beta +1)$. Finally, Dolega in \cite{Dol17} shows that $$h^n_{\mu, \nu}(\beta +1) = \sum_{M \in  \mathcal{L}^n_{\mu, \nu}}\beta^{\vartheta(M)}$$ holds true when either $\beta$ is restricted to the values $\beta \in \{-1,0,1\}$ or $\beta$ is general but $\ell(\mu) + \ell(\nu) \geq n-3$. 

Except the limit cases $\la=[1^n],[2,1^{n-2}]$ already covered by Goulden and Jackson \cite{GouJac96}, the matchings-Jack conjecture has be proved by  Kanunnikov and Vassiliveva  \cite{KanVas16} in the case $\mu = \nu = (n)$. More precisely the authors introduce a weight function $\wt_\la$ for matchings in $\mathcal{G}^\la_{n,n}$ $$a^\la_{n,n}(\beta+1) = \sum_{\delta \in \mathcal{G}^\la_{n,n}}\beta^{\wt_\la(\delta)},$$ besides, $\wt(\de)=0$ iff $\de$ is bipartite.

In \cite{DolFer16} and \cite{DolFer17} Dolega and Feray focus only on the polynomiality part of the conjectures and show that the $a^\la_{\mu, \nu}(\al)$ and $h^\la_{\mu, \nu}(\al)$ are polynomials in $\al$ with {\bf rational} coefficients for arbitrary partitions $\la, \mu, \nu$. See also \cite{Vas15a} for a proof of the polynomiality with non-negative integer coefficients of a multi-indexed variation of $a^\la_{\mu, \nu}(\al)$ in some important special cases.

\section{Proof of Theorem \ref{thm : main} and Corollary \ref{thm : cor1}}

\subsection{Properties of Jack symmetric functions}
In order to prove Theorem \ref{thm : main}, we need to recall some known properties of Jack symmetric functions.\\

\noindent {\bf Pieri formulas}\\
Given two partitions, $\mu \subseteq \la$ the generalised binomial coefficients $\binom{\la}{\mu}$ are defined through the relation
\begin{equation*}
\frac{J^\al_\la(1+x_1, 1+x_2,\ldots)}{J^\al_\la(1,1,\ldots)} = \sum_{\mu \subseteq \la}\binom{\la}{\mu}\frac{J^\al_\mu(x_1,x_2,\ldots)}{J^\al_\mu(1,1,\ldots)}.
\end{equation*}
Details about the existence and properties of these binomial coefficients can be found in \cite{Las89, Las90}. As shown in \cite{OkoOls97} these coefficients are equal to some properly normalised shifted Jack polynomials.\\
For $\rho \vdash n+1$ and integer $1\leq i \leq \ell(\rho)$ define the partition $\rho_{(i)}$ of $n$ (if it exists) obtained by replacing $\rho_i$ in $\rho$ by $\rho_i-1$ and keeping all the other parts as in $\rho$. Similarly for $\gamma \vdash n$ and integer $1\leq i \leq \ell(\gamma)+1$ we define the partition $\gamma^{(i)}$ of $n+1$ (if it exists) obtained by replacing $\gamma_i$ in $\gamma$ by $\gamma_i+1$ and keeping all the other parts as in $\gamma$. Define also the numbers $c_i(\gamma)$ as
\begin{equation*}
\label{eq : cdef} c_i(\gamma) = \alpha\binom{\gamma^{(i)}}{\gamma}\frac{j_\gamma(\al)}{j_{\gamma^{(i)}}(\al)}.
\end{equation*}
In \cite{Las89} Lassalle showed the following Pieri formulas 
\begin{align}
\label{eq : p1ci} p_1J_\gamma^\al &= \sum_{i=1}^{\ell(\gamma)+1}c_i(\gamma)J_{\gamma^{(i)}}^\al,\\
\label{eq : p1p}p_1^\perp J_\rho^\al &= \al\frac{\partial}{\partial p_1}J_\rho^\al = \sum_{i=1}^{\ell(\rho)}\frac{j_\rho(Dol17\al)}{j_{\rho_{(i)}}(\al)}c_i(\rho_{(i)})J_{\rho_{(i)}}^\al.
\end{align}
\noindent {\bf Power sum expansion and Laplace Beltrami operator}\\
For $\la, \mu \vdash n$, denote $\theta^\la_\mu(\al)$ the coefficient of $p_\mu$ in the power sum expansion of $J_\la^\al$. Namely, $$J_\la^\al = \sum_{\mu \vdash n}\theta^\la_\mu(\al)p_\mu.$$ As shown in \cite{Mac87} Jack symmetric functions are eigenfunctions of $D_\al$ and verify
\begin{equation}
\label{eq : eigen}
D_\al J_\la^\al =\theta_{[1^{|\la|-2}2^1]}^\la(\al)J_\la^\al.
\end{equation}
Furthermore, according to \cite[Lemma 2]{KanVas16}, for any partition $\gamma$ of some integer $n$
\begin{align} \label{the}
\theta_{n+1}^{\gamma^{(i)}}(\al) = \theta_{n}^\ga(\al)\left(\theta^{\ga^{(i)}}_{[1^{n-1}2]}(\al)-\theta^{\ga}_{[1^{n-2}2]}(\al)\right).
\end{align}
Finally, for integers $a,b \geq 0$, the following relation holds (\cite[Equation (30)]{KanVas16}):
\begin{align} \label{DapDb}
\sum_{i=1}^{\ell(\ga)+1}c_i(\ga)\left(\theta^{\ga^{(i)}}_{[1^{n-1}2]}\right)^a&\left(\theta^{\ga}_{[1^{n-2}2]}\right)^b
J_{\ga^{(i)}}^\al=D_\al^ap_1D_\al^bJ_\gamma^\al
\end{align}

\noindent {\bf Operators}\\
Following \cite{Las08}, denote also the two conjugate operators $E_2$ and  $E_2^{\perp}$ defined by
\begin{align}
\nonumber \label{eq : e} &E_2 = [D_\al,p_1/\al] = \sum_{i\geq 1}ip_{i+1}\frac{\partial}{\partial p_i},\\
\nonumber &E_2^{\perp} = [p_1^\perp/\al,D_\al] = \sum_{i\geq 1}(i+1)p_i\frac{\partial}{\partial p_{i+1}}.
\end{align}
We show in \cite[Theorem 5]{KanVas16} that for $x$ and $y$ indeterminates  the following relation for Jack symmetric functions holds
\begin{equation}
\label{eq : EDE}
\sum_{\rho \vdash n+1}\frac{\theta^\rho_{n+1}(\al)J_\rho^\al(x)E_2^\perp J_\rho^\al(y)}{j_\rho(\al)} = \sum_{\gamma \vdash n}\frac{\theta^\gamma_{n}(\al)J_\gamma^\al(y)\Delta J_\gamma^\al(x)}{j_\gamma(\al)}.
\end{equation}

\subsection{Proof of Theorem \ref{thm : main}}

The first step is to show the following lemma.  
\begin{lem} Let $x$ and $y$ be two indeterminates. Jack symmetric functions verify
 \begin{equation} \label{lem} 
\sum_{\rho\vdash n+1} \dfrac{\theta^\rho_{n+1}(\al)J^\al_\rho(x)p_1^\perp J^\al_\rho(y)}{j_\rho(\al)}=
\al\sum_{\ga\vdash n} \dfrac{\theta^\ga_n(\al)J_\ga^\al(y)E_2J^\al_\ga(x)}{j_\ga(\al)}.   
 \end{equation}
\end{lem}
\begin{proof} Start with the second Pieri formula and then apply the known identities above. For brevity, we omit parameter $\al$ in Jack symmetric functions and their coefficients in the power sum basis.
 \begin{align*} \sum_{\rho\vdash n+1} \dfrac{\theta^\rho_{n+1}J_\rho(x)p_1^\perp J_\rho(y)}{j_\rho}&\stackrel{(\ref{eq : p1p})}{=}
\sum_{\rho\vdash n+1} \sum_{i=1}^{\ell(\rho)} \dfrac{\theta^\rho_{n+1}J_\rho(x)c_i(\rho_{(i)}) J_{\rho_{(i)}}(y)}{j_{\rho_{(i)}}}, \\
&\stackrel{\phantom{(\ref{the})}}{=}\sum_{\ga\vdash n} \sum_{i=1}^{\ell(\ga)+1} \dfrac{\theta^{\ga^{(i)}}_{n+1}J_{\ga^{(i)}}(x)c_i(\ga) J_\ga(y)}{j_{\ga}}, \\
&\stackrel{(\ref{the})}{=} \sum_{\ga\vdash n} \sum_{i=1}^{\ell(\ga)+1} \dfrac{c_i(\ga)\left(\theta^{\ga^{(i)}}_{[1^{n-1}2]}-\theta^\ga_{[1^{n-2}2]}\right)\theta^\ga_nJ_{\ga^{(i)}}(x)J_\ga(y)}{j_\ga}, \\
&\stackrel{(\ref{DapDb})}{=}\sum_{\ga\vdash n} \dfrac{ \theta^\ga_n J_\ga(y)(D_\al p_1-p_1D_\al)J_\ga(x) }{j_\ga},\\
&\stackrel{\phantom{(\ref{the})}}=\al\sum_{\ga\vdash n} \dfrac{\theta^\ga_nJ_\ga(y)E_2J_\ga(x)}{j_\ga}.\qedhere
 \end{align*}
\end{proof}

The key element of the proof of Theorem \ref{thm : main} is the following result.
\begin{thm}\label{thm : DPi}
\noindent For any integer $k \geq 1$ denote $\Pi_k$ the operator defined by:
\begin{align*}
\Pi_1 = \frac{1}{\al}p_1^{\perp}=\frac{\partial}{\partial p_1}, \;\;\; \Pi_{k+1} = [\Pi_k,E_2^{\perp}].
\end{align*}
Given two indeterminates $x$ and $y$, the following identity holds:
\begin{equation} \label{eq : DPi}
\sum_{\rho\vdash n+k} \dfrac{\theta^\rho_{n+k}(\al)J^\al_\rho(x)\Pi_k J^\al_\rho(y)}{j_\rho(\al)}=
\sum_{\ga\vdash n} \dfrac{\theta^\ga_n(\al)J_\ga^\al(y)\Omega_kJ^\al_\ga(x)}{j_\ga(\al)}.   
 \end{equation}
\end{thm}
\begin{proof}
In the case $k=1$ Theorem \ref{thm : DPi} reduces to Equation (\ref{lem}). Assume the property is true for some $k\geq 1$. We have (reference to parameter $\al$ is also removed)
\begin{align*} 
&\sum_{\rho\vdash n+k+1} \dfrac{\theta^\rho_{n+k+1}J_\rho(x)\Pi_{k+1} J_\rho(y)}{j_\rho}\\ 
&= \sum_{\rho\vdash n+k+1} \dfrac{\theta^\rho_{n+k+1}J_\rho(x)[\Pi_k,E_2^{\perp}]J_\rho(y)}{j_\rho}\\
&= \Pi_k \sum_{\rho\vdash n+k+1} \dfrac{\theta^\rho_{n+k+1}J_\rho(x)E_2^{\perp}J_\rho(y)}{j_\rho}
-E_2^{\perp}\sum_{\rho\vdash n+k+1} \dfrac{\theta^\rho_{n+k+1}J_\rho(x)\Pi_kJ_\rho(y)}{j_\rho}\\
&= \Pi_k \sum_{\rho\vdash n+k} \dfrac{\theta^\rho_{n+k}\Delta J_\rho(x)J_\rho(y)}{j_\rho}
-E_2^{\perp}\sum_{\ga\vdash n+1} \dfrac{\theta^\ga_{n+1}J_\ga(y)\Omega_kJ_\ga(x)}{j_\ga}\\
&= \Delta \sum_{\rho\vdash n+k} \dfrac{\theta^\rho_{n+k}J_\rho(x)\Pi_k J_\rho(y)}{j_\rho}
-\Omega_k\sum_{\ga\vdash n+1} \dfrac{\theta^\ga_{n+1}E_2^{\perp}J_\ga(y)J_\ga(x)}{j_\ga}\\
&= \Delta \sum_{\ga\vdash n} \dfrac{\theta^\ga_{n}\Omega_kJ_\ga(x)J_\ga(y)}{j_\ga}
-\Omega_k\sum_{\ga\vdash n} \dfrac{\theta^\ga_{n}J_\ga(y)\Delta J_\ga(x)}{j_\ga}\\
&= \sum_{\ga\vdash n} \dfrac{\theta^\ga_{n}\Delta \Omega_kJ_\ga(x)J_\ga(y)}{j_\ga}
-\sum_{\ga\vdash n} \dfrac{\theta^\ga_{n}J_\ga(y)\Omega_k\Delta J_\ga(x)}{j_\ga}\\
&= \sum_{\ga\vdash n} \dfrac{\theta^\ga_{n}[\Delta ,\Omega_k]J_\ga(x)J_\ga(y)}{j_\ga}.
 \end{align*}
 Where the fourth and the sixth line are both obtained by applying Equation (\ref{eq : EDE}) and the recurrence hypothesis. As a result the property is true for $k+1$.
\end{proof}
\noindent We end the proof of Theorem \ref{thm : main} by noticing that
\begin{align*}
\Pi_{k} = k!\frac{\partial}{\partial p_{k}}.
\end{align*}
For an arbitrary integer partition $\nu  = (\nu_1, \ldots, \nu_p)$ of $n$, rewrite Equation (\ref{eq : DPi}) with $n$ instead of $n+k$ and $\nu_p$ instead of $k$ and extract the coefficient in $p_{\nu\sm \nu_p}(y)$: 
\begin{align} \label{eq}
m_{\nu_p}(\nu) \sum_{\la \vdash n}z_\la^{-1}\al^{-\ell(\la)}a^{\la}_{n,\nu}p_\la(x) = \Omega_{\nu_p}\sum_{\rho \vdash n-\nu_p}\frac{z_\rho^{-1}\al^{-\ell(\rho)}}{\nu_p!}a^{\rho}_{n-\nu_p,\nu\setminus\nu_p}p_\rho(x).
\end{align}
Iterating the equation above for $\nu_2, \ldots, \nu_{p-1}$ and, then, applying $\nu_1-1$ times Equation (\ref{eq : EDE}) yields the desired formula. 
\subsection{Proof of Corollary \ref{thm : cor1}}
As shown e.g. by Stanley in \cite{Sta89}, the Laplace Beltrami operator can be expressed in terms of the power sum symmetric functions as:
\begin{equation*}
D_\al = \frac{(\al-1)}{2}\sum_{i}i(i-1)p_i\frac{\partial}{\partial p_i} +  \frac{\al}{2}\sum_{i,j}ijp_{i+j}\frac{\partial}{\partial p_i}\frac{\partial}{\partial p_j} + \frac{1}{2}\sum_{i,j}(i+j)p_ip_j\frac{\partial}{\partial p_{i+j}}.
\end{equation*}
Besides, $$[D_\al,p_1/\al] = E_2 = \sum_{i\geq 1} ip_{i+1}\frac{\partial}{\partial p_i}$$
As a result it is clear from the definition of operators $\{\Omega_k\}_k$ and $\Delta$ that, for any integer partition $\nu$,  the coefficients in the power sum expansion of $\left(\prod_{i\geq 2}\Omega_{\nu_i}\right)\Delta^{\nu_1-1}(p_1)$ are polynomial in $\al$ with (possibly negative) integer coefficients. Denote for any $\la, \nu \vdash n$ the integers $\{g^i_{\la,\nu}\}_{i\geq 0}$ such that
\begin{align*}
Aut_\nu |C_\la|\al^{-\ell(\la)}a^{\la}_{n,\nu}(\al) = \frac{1}{\al}\binom{n}{\nu}[p_\la]\left(\prod_{i\geq 2}\Omega_{\nu_i}\right)\Delta^{\nu_1-1}(p_1) = \frac{1}{\al}\sum_{i\geq 0}g^i_{\la,\nu}\al^i
\end{align*}
But, according to \cite[Thm. 5]{Vas15a},
\begin{align*}
\al^{\ell(\la)}a^\la_{n,\nu}(\al^{-1}) = (-\al)^{-n+1+\ell(\la)+\ell(\nu)}\al^{-\ell(\la)}a^\la_{n,\nu}(\al).
\end{align*}
Replacing the coefficients $a^\la_{n,\nu}(\al)$ and $a^\la_{n,\nu}(\al^{-1})$ by their expressions in terms of $\{g^i_{\la,\nu}\}_{i\geq 0}$ yields
\begin{align*}
\al\sum_{i\geq 0}g^i_{\la,\nu}\al^{-i} &= (-\al)^{-n+1+\ell(\la)+\ell(\nu)}\frac{1}{\al}\sum_{i\geq 0}g^i_{\la,\nu}\al^i\\
\sum_{i\geq 0}g^i_{\la,\nu}\al^{n+1-\ell(\la)-\ell(\nu)-i} &=(-1)^{-n+1+\ell(\la)+\ell(\nu)}\sum_{i\geq 0}g^i_{\la,\nu}\al^i
\end{align*}
Equating the coefficients in $\al^i$ in the equation above shows that $i> n+1 -\ell(\la)-\ell(\nu)$ implies that $g^i_{\la, \nu} = 0$. As a consequence, $Aut_\nu |C_\la|\al^{1-\ell(\la)}a^{\la}_{n,\nu}(\al)$ is a polynomial in $\al$ with integer coefficients of degree at most $n+1 -\ell(\la)-\ell(\nu)$ and, finally, $Aut_\nu |C_\la|a^{\la}_{n,\nu}(\al)$ is a polynomial in $\al$ with integer coefficients of degree at most $n-\ell(\nu)$.\\
Using this result together with  Equation (\ref{eq : ha}) shows that $Aut_\nu(n-1)!h^{\la}_{n,\nu}(\al)$ is a polynomial in $\al$ with integer coefficients of degree at most $n+1-\ell(\nu)-\ell(\la)$.

\section{Proof of Theorems \ref{thm : A} and  \ref{thm : H}}
While Theorem \ref{thm : main} allows us to demonstrate most of the polynomial properties of  $a^\la_{n,\nu}(\al)$ and $h^\la_{n,\nu}(\al)$, it is not enough to prove the matchings-Jack and the hypermap-Jack conjectures. In particular, it is not clear from the definition of operators $\{\Omega_k\}_k$ that the coefficients of the expansion of  $a^\la_{n,\nu}(\al)$ and $h^\la_{n,\nu}(\al)$ in $\beta = \al -1$ are non-negative. We overcome this issue by computing a more explicit form of operator $\Omega_k$ in the cases $k=1,2$ and $3$ to show some recurrence formulas for $\wit{a}^\la_{n,\nu}(\al)$ (for $\nu$ with at most one part strictly greater than $3$) and $\widetilde{h}_{n,[k^m]}^\la(1+\beta)$ ($k\in\{1,2,3,n\}$). 
Then we use bijective constructions for labelled matchings and hypermaps to show that the right-hand sides of the main equations of theorems \ref{thm : A} and \ref{thm : H} fulfil the same recurrence relation. This method can be extended to higher values of $k$ but computations and bijections become cumbersome and do not shed much more light on the problem.

\subsection{Recurrence relations for the coefficients $\wit{a}_{n,\nu}^\la$ and $\wit{h}^\la_{n,[k^m]}$}

 \subsubsection{Explicit computation of operator $\Omega_k$}
We begin with the following lemma.

\begin{lem}\label{thm : lemOm} For $k\geq 1$, let $\Omega_k$ be the operator defined in Theorem \ref{thm : main}. For small values of $k$, the explicit form of $\Omega_k$ is given by

\begin{align}
\label{eq : om1} \Omega_1=&\sum_{i\geq 1} ip_{i+1}\frac{\partial}{\partial p_i},\\
\nonumber\Omega_2 =& (\al-1)\sum_{i}(i-1)(i-2)p_i\frac{\partial}{\partial p_{i-2}}\\
\label{eq : om2}&+\sum_{i,j}(i+j-2)p_ip_j\frac{\partial}{\partial p_{i+j-2}}+\al\sum_{i,j}ijp_{i+j+2}\frac{\partial}{\partial p_{i}}\frac{\partial}{\partial p_{j}},
\end{align}
\begin{align}
\nonumber\Omega_3 =& (2(\al-1)^2+\al)\sum_{i}(i-1)(i-2)(i-3)p_i\frac{\partial}{\partial p_{i-3}}\\
\nonumber&+3(\al-1)\sum_{i,j}(i+j-2)(i+j-3)p_ip_j\frac{\partial}{\partial p_{i+j-3}}\\
\nonumber&+3\alpha(\al-1)\sum_{i,j}ij(i+j+2)p_{i+j+3}\frac{\partial}{\partial p_{i}}\frac{\partial}{\partial p_{j}}\\
\nonumber&+3\alpha\sum_{i,j,k}ijp_{i+j-k+3}p_k\frac{\partial}{\partial p_{i}}\frac{\partial}{\partial p_{j}}\\
\nonumber&+2\alpha^2\sum_{i,j,k}ijkp_{i+j+k+3}\frac{\partial}{\partial p_{i}}\frac{\partial}{\partial p_{j}}\frac{\partial}{\partial p_{k}}\\
\label{eq : om3}&+2\sum_{i,j,k}(i+j+k-3)p_ip_jp_k\frac{\partial}{\partial p_{i+j+k-3}}.
\end{align}
\end{lem}
\begin{proof} The proof of Lemma \ref{thm : lemOm} uses elementary computations on operators. The details are postponed to Appendix \ref{sec : appendix}. 
\end{proof}

We proceed with the recurrence relations that we use throughout this section. To state them some additional notation is required. Namely, for integers $i_1,\cdots,i_k$, denote $m_{i_1,\cdots,i_k}(\la)$ the number of ways to chose in $\la$ first a part equal to $i_1$, then a part equal to $i_2$, etc. We have  $$m_{i_1,\cdots,i_k}(\la) = m_{i_1}(\la)(m_{i_2}(\la) - \delta_{i_1,i_2})\cdots(m_{i_k}(\la)-\sum_{j=1}^{k-1}\delta_{i_k,i_j})$$ where $\delta_{a,b}=1$ if $a=b$, $0$ otherwise.

\subsubsection{Reduction of  $(n)$-parts} \label{sub:n}

Denote for integers $i$, $j$ and partition $\la$
\begin{align*} 
 &\la_{\downarrow(i)\phantom{,j}}= \la\setminus \{i\} \cup \{i-1\},\\
&\la_{\downarrow(i,j)}= \la\setminus \{i,j\}\cup\{i+j-1\},\\
 &\la^{\uparrow(i,j)}= \la \setminus \{i+j+1\} \cup \{i,j\}.
\end{align*}

Using Theorem 1 in \cite{KanVas16} in the case $\nu = (n)$ one gets the following formula for the coefficients $a_{n, n}^\la(\al)$. 
\begin{align} 
\nonumber na_{n,n}^\la(\al) = \sum_i \la_i &\left[ (\al-1)(\la_i-1)a_{n-1,n-1}^{\la_{\downarrow(\la_i)}}(\al)\right. \\
&+\sum_{d=1}^{\la_i-2}a_{n-1,n-1}^{\la^{\uparrow(\la_i-1-d,d)}}(\al)+\al\sum_{j\neq i}\la_j\left.a_{n-1,n-1}^{\la_{\downarrow(\la_i,\la_j)}}(\al)\right].
\label{eq : ann}
\end{align}

As a consequence, we have the following recurrence for the coefficients $\widetilde{h}_{n, n}^\la(\al)$. 

\begin{lem}\label{lem : hnn} For $\la$ integer partition of $n \geq 2$, the numbers $\widetilde{h}_{n,n}^\la(\al)$ verify
\begin{align*} 
\nonumber \widetilde{h}_{n,n}^\la(\al) = \sum_{i\geq 1} &\left[ (\al-1)(i-1)^2m_{i-1}(\la_{\downarrow(i)})\widetilde{h}_{n-1,n-1}^{\la_{\downarrow(i)}}(\al)\right. \\
\nonumber&+\al\sum_{i,d\geq 1}(i-1-d)dm_{i-1-d,d}(\la^{\uparrow(i-1-d,d)})\widetilde{h}_{n-1,n-1}^{\la^{\uparrow(i-1-d,d)}}(\al)\\
&+\sum_{i,j\geq 1}(i+j-1)m_{i+j-1}(\la_{\downarrow(i,j)})\left.\widetilde{h}_{n-1,n-1}^{\la_{\downarrow(i,j)}}(\al)\right].
\end{align*} 
\end{lem}
\begin{proof}

Using Equation~(\ref{eq : ha}) one can rewrite Equation~(\ref{eq : ann}) in terms of the $h^{\la}_{n,n}$ as
\begin{align} 
\nonumber (n-1)h_{n,n}^\la(\al) = \sum_{i\geq 1} &\left[ (\al-1)(i-1)^2m_{i-1}(\la_{\downarrow(i)})h_{n-1,n-1}^{\la_{\downarrow(i)}}(\al)\right. \\
\nonumber&+\al\sum_{i,d\geq 1}(i-1-d)dm_{i-1-d,d}(\la^{\uparrow(i-1-d,d)})h_{n-1,n-1}^{\la^{\uparrow(i-1-d,d)}}(\al)\\
\nonumber&+\sum_{i,j\geq 1}(i+j-1)m_{i+j-1}(\la_{\downarrow(i,j)})\left.h_{n-1,n-1}^{\la_{\downarrow(i,j)}}(\al)\right].
\end{align} 
\noindent Multiplying both sides by $(n-2)!$ yields the desired result.
\end{proof}

\subsubsection{Reduction of  $1$-parts} \label{sub:one}

Given two integers $n$ and $k$, a partition $\la=(\la_1,\ldots, \la_p)\vdash n+k$ and a tuple of integers $\kappa=(k_1,\ldots, k_p)$ such that $k_1+\ldots+k_p=k$, $k_1,\ldots, k_p\geq 0$ and $k_i<\la_i$ for all $i$ we denote $\la - \kappa $ the reordering in decreasing order of the non-negative integers $(\la_1-k_1,\ldots, \la_p-k_p)$. Clearly  $\la - \kappa \vdash n$. Furthermore, for two partitions $\rho_1$ and $\rho_2$ we denote $\rho_1\cup\rho_2$ the partitions obtained by reordering in decreasing order the union of the parts of $\rho_1$ and the parts of $\rho_2$.\\ 
Given an integer partition $\rho\vdash n$ such that $m_1(\rho) = 0$ we write the coefficient $a^\la_{n+k,\rho\cup [1^k]}$ as a sum of the $a^{\la - \kappa}_{n,\rho}$'s. 
\begin{lem} 
 For any integers $n$ and $k$ and partitions $\la \vdash n+k$ and $\rho\vdash n$ such that $\ell(\la) = p$ and $m_1(\rho) = 0$ we have
\begin{align} 
 a^{\la}_{n+k,\rho\cup [1^k]}=\sum_{\substack{\kappa=(k_1,\ldots, k_p)\\0\leq k_i<\la_i,\\ \sum_i k_i = k }}
\binom{\la_1}{k_1}\ldots \binom{\la_p}{k_p}a^{\la- \kappa}_{n,\rho}.
\label{rhon}
\end{align}
If $p>n$ then this sum is empty and $a^{\la}_{n+k,\rho\cup [1^k]}=0$.
\end{lem}
\begin{proof}
Using formula (\ref{eq : om1}) rewrite Equation (\ref{eq}) with $\nu_p=1$:
$$k\sum_{\la\vdash n+k}z_\la^{-1}\al^{-\ell(\la)}p_\la(x)=\sum_{i\geq 1} ip_{i+1}\dfrac{\partial}{\partial p_i}
\sum_{\tau \vdash n+k-1} z_\tau^{-1}\al^{-\ell(\tau)}a^\tau_{n+k-1,\rho\cup [1^k]}p_\tau(x).$$
Denote $ \la_{\downarrow(i)}= \la\setminus \{i\} \cup \{i-1\}$. Extracting the coefficient in $p_\la(x)$ yields
\begin{equation} \label{1one}
ka^\la_{n+k,\rho\cup [1^k]}=\sum_{i\colon \la_i>1} \la_i a^{\la_{\da(\la_i)}}_{n+k-1,\rho\cup [1^{k-1}]}.
\end{equation}
In the case $p\leqslant n$, there exist $\dbinom{k}{k_1,\ldots, k_p}=\dfrac{k!}{k_1!\ldots k_p!}$ ways to turn $\la$ into $\la- \kappa$ in $k$ steps for fixed $\kappa$ (each step is a decrease of some $\la_i$ by one). Therefore the reduction of the quantity $k!a^\la_{n+k,\rho\cup [1^k]}$ step by step yields:
$$k!a^\la_{n+k,\rho\cup [1^k]}=\sum_{\substack{\kappa=(k_1,\ldots, k_p)\\0\leq k_i<\la_i,\\ \sum_i k_i = k }}
{(\la_1)}_{k_1}\ldots {(\la_p)}_{k_p}\dfrac{k!}{k_1!\ldots k_p!}a^{\la - \kappa}_{n,\rho}$$
where  ${(N)}_{s}:=N(N-1)\ldots (N-s+1)$ if $s\geq 1$ and $(N)_{0} =1$.
Dividing both sides by $k!$ we get Equation~(\ref{rhon}).
If $p> n$ then the reduction process terminates after $n+k-p<k$ steps and we get that 
$a^\la_{n+k,\rho\cup [1^k]}$ is proportional to $a^{1^{p}}_{p, [\rho,1^{p-n}]}=0$ as per
\cite[lemma 3.3]{GouJac96}.
\end{proof}

\subsubsection{Reduction of  $2$-parts}

For any integers $i$, $j$ and an integer partition $\la$, define the following operations 
\begin{align*} 
&\la_{\da\da(i)\phantom{,j}}= \la\setminus \{i\} \cup \{i-2\},\\
&\la_{\da\da(i,j)}= \la\setminus \{i,j\}\cup\{i+j-2\}, \\
&\la^{\ua\ua(i,j)}= \la \setminus \{i+j+2\} \cup \{i,j\}.
\end{align*}
We have the following lemma. 
\begin{lem} \label{thm : 2} For any integer $n\geq 0, l\geq 1$ and any partitions $\rho\vdash n$, $\la\vdash n+2l$ such that $1,2 \notin \rho$,
the following formula is true:
\begin{align} \label{n2l}
 \nonumber&\wit{a}^\la_{n+2l,\rho\cup[2^l]}=
(\al-1)\sum_{i\colon \la_i>2} \frac{\la_i(\la_i-1)}{2}\wit{a}^{\la_{\da\da\la_i}}_{n+2l-2,\rho\cup[2^{l-1}]}\\
&+\al\sum_{i<j,\la_i+\la_j>2}\la_i\la_j\wit{a}^{\la_{\da\da(\la_i,\la_j)}}_{n+2l-2,\rho\cup[2^{l-1}]}+
\frac 12 \sum_i \la_i\sum_{d=1}^{\la_i-3}\wit{a}^{\la^{\ua\ua(\la_i-2-d,d)}}_{n+2l-2,\rho\cup[2^{l-1}]}.
\end{align}
\end{lem}

\begin{proof} We use formula (\ref{eq : om2}) and Equation (\ref{eq}).
It follows from (\ref{eq}) that
\begin{align}
\dfrac{2l a^\la_{n+2l,\rho\cup[2^l]}}{z_\la\al^{\ell(\la)}}=[p_\la(x)]\Omega_2\left(\sum_{\eps\vdash n+2l-2}
\dfrac{a^\eps_{n+2l-2,\rho\cup[2^{l-1}]}}{z_\eps\al^{\ell(\eps)}}p_\eps(x)\right).
\end{align}
For each summand $S$ of $\Omega_2$, one can find a partition $\eps\vdash n+2l-2$ such that $S p_\eps(x)$ contributes in $p_\la(x)$. All cases are in the following table where $m_k=m_k(\la)$.

\vskip 5 pt

\noindent \begin{tabular}{|c|c|c|c|c|}
\hline
Summand of $\Omega_2$ & $\eps$ & ${z_\la}/{z_\eps}$ & ${\al^{\ell(\la)}}/{\al^{\ell(\eps)}}$ \\
\hline
$(\al-1)(i-1)(i-2)p_i\dfrac{\partial}{\partial p_{i-2}}$ & $\la_{\da\da (i)}$, $i>2$ & 
$\dfrac{im_i}{(i-2)(m_{i-2}+1)}$ & $1$ \\
\hline
$(i+j-2)p_ip_j\dfrac{\partial}{\partial p_{i+j-2}}$  & $\la_{\da\da(i,j)}$, $i\ne j$ & $\dfrac{ijm_im_j}{(i+j-2)(m_{i+j-2}+1)}$ & $\al^{-1}$ \\
\hline
$(2i-2)p_i^2\dfrac{\partial}{\partial p_{2i-2}}$ & $\la_{\da\da(i,i)}$, $m_i>1$ & 
$\dfrac{i^2m_i(m_i-1)}{(2i-2)(m_{2i-2}+1)}$  & $\al^{-1}$ \\
\hline
$\al ijp_{i+j+2}\dfrac{\partial}{\partial p_{i}}\dfrac{\partial}{\partial p_j}$ & $\la^{\ua\ua(i,j)}$, $i\ne j$ & $\dfrac{(i+j+2)m_{i+j+2}}{ij(m_i+1)(m_j+1)}$  & $\al$  \\
\hline
$\al i^2p_{2i+2}\dfrac{\partial^2}{\partial p_i^2}$ & $\la^{\ua\ua(i,i)}$ & $\dfrac{(2i+2)m_{2i+2}}{i^2(m_i+2)(m_i+1)}$ & $\al$ \\
\hline
\end{tabular}

\vskip 5 pt

Putting everything together yields the desired formula.
\end{proof}

\begin{rem}
Iterating Equation (\ref{n2l}) one can remove all the $2$ but the corresponding explicit formula is hard even in the case $l=2$. 
\end{rem}

Furthermore, the coefficients $\widetilde{h}_{2m,[2^m]}^\la(\al)$ verify the following recurrence relation.
\begin{lem} For $\la$ integer partition of $2m$ with $m \geq 2$, the numbers $\widetilde{h}_{2m,[2^m]}^\la(\al)$ verify
\begin{align*}
\widetilde{h}_{2m,[2^m]}^\la(\al) &= (\al-1)\sum_{i\geq1}(i-1)(i-2)m_{i-2}(\la_{\da\da(i)})\widetilde{h}_{2m-2,[2^{m-1}]}^{\la_{\da\da(i)}}(\al)\\
&+\sum_{i,j\geq 1}(i+j-2)m_{i+j-2}(\la_{\da\da(i,j)})\widetilde{h}_{2m-2,[2^{m-1}]}^{\la_{\da\da(i,j)}}(\al)\\
&+\al\sum_{i,d\geq 1}(i-2-d)dm_{i-2-d,d}(\la^{\ua\ua(i-2-d,d)})\widetilde{h}_{2m-2,[2^{m-1}]}^{\la^{\ua\ua(i-2-d,d)}}(\al).
\end{align*}
\end{lem}
\begin{proof}
Using Equation (\ref{eq : om2}) gives
\begin{align*}
2(m-1)h_{2m,[2^m]}^\la(\al) &= (\al-1)\sum_{i\geq1}(i-1)(i-2)m_{i-2}(\la_{\da\da(i)})h_{2m-2,[2^{m-1}]}^{\la_{\da\da(i)}}(\al)\\
&+\sum_{i,j\geq 1}(i+j-2)m_{i+j-2}(\la_{\da\da(i,j)})h_{2m-2,[2^{m-1}]}^{\la_{\da\da(i,j)}}(\al)\\
&+\al\sum_{i,d\geq 1}(i-2-d)dm_{i-2-d,d}(\la^{\ua\ua(i-2-d,d)})h_{2m-2,[2^{m-1}]}^{\la^{\ua\ua(i-2-d,d)}}(\al).
\end{align*}
Multiplying both sides by $(m-2)!2^{m-2}$ yields the desired result.
\end{proof}

\subsubsection{Reduction of $3$-parts}

For any integers $i$, $j$, $k$ and an integer partition $\la$, define the following operations 
\begin{align*}
&\la_{\da\da\da (i)\phantom{,j,k}} = \la\setminus \{i\}\cup\{i-3\}\\
&\la_{\da\da\da( i, j)\phantom{,k}} = \la\setminus \{i,j\}\cup\{i+j-3\}\\
&\la^{\ua\ua\ua(i,j)\phantom{,k}}= \la\setminus \{i+j+3\}\cup\{i,j\}\\
&\la^{\ua\ua\ua(i,j,k)}=\la\setminus \{i+j+k+3\}\cup\{i,j,k\}\\
&\la_{\da\da\da(i,j,k)} =\la\setminus \{i,j,k\}\cup\{i+j+k-3\}
\end{align*}
We also use the notation $\la_{\downarrow(i,j)}= \la\setminus \{i,j\} \cup \{i+j-1\}$.The following recurrence formula holds. 

\begin{lem} \label{rec3a}
 Let $n\geq 0, m\geq 1$ be integer such that $n+3m>3$ and $\rho\vdash n$, $\la\vdash n+3m$ be partitions such that
  $1,2,3\notin \rho$. Denote $\mu=(n+3m-3, \rho\cup[3^{m-1}])$. The following recurrence formula holds. 
\begin{align} 
 \label{s1} \wit{a}^\la_{n+3m,\rho\cup[3^m]} (1+\be)&=
 (2\be^2+\be+1)\sum_{\la_i>3}\binom{\la_i}{3} \wit{a}^{\la_{\da\da\da\la_i}}_{\mu}\\
 \label{s2} &+(2\be^2+4\be+2)\sum_{\substack{i<j<k \\ \la_i + \la_j+\la_k>3}}\la_i\la_j\la_k \wit{a}^{\la\da\da\da(\la_i,\la_j,\la_k)}_\mu\\
 \label{s3} &+\dfrac 13\sum_i \la_i\sum_{d,f\geq 1} \wit{a}^{\la\sm\la_i\cup(f,d,\la_i-3-d-f)}_\mu\\
  \label{s4} &+(\be^2+\be)\sum_{\substack{i<j \\ \la_i + \la_j>3}}\la_i\la_j(\la_i+\la_j-2)\wit{a}^{\la\da\da\da(\la_i,\la_j)}_\mu\\
 \label{s5} &+(\be+1)\sum_{i< j}\la_i\la_j\sum_{d=1}^{\la_i+\la_j-4}\wit{a}^{\la\sm(\la_i,\la_j)\cup(d,\la_i+\la_j-3-d)}_\mu\\
 \label{s6} &+ \be\sum_i\binom{\la_i}{2}\sum_{d=1}^{\la_i-4}  \wit{a}^{\la\ua\ua\ua(d,\la_i-3-d)}_\mu.
\end{align}
\end{lem}
The proof is similar to the proof of Lemma \ref{thm : 2}. The corresponding recurrence relation for the coefficients $\widetilde{h}^\la_{3m,[3^m]}(\al)$ is given by the following lemma. 
\begin{lem} For integer $m>1$ and $\la \vdash 3m$, the coefficients $\widetilde{h}^\la_{3m,[3^m]}(\al)$ verify
\begin{align*}
\nonumber&\widetilde{h}^\la_{3m,[3^m]} =\\
& 3\widetilde{h}^3_{3,3}\sum_{i}\binom{i-1}{3}m_{i-3}(\la_{\da\da\da (i)})\widetilde{h}^{\la_{\da\da\da (i)}}_{3m-3,[3^{m-1}]}\\
\nonumber&+\widetilde{h}^{(2,1)}_{3,3}\frac{1}{2}\sum_{i, j}(i+j-2)(i+j-3)m_{i+j-3}({\la_{\da\da\da( i, j)}})\widetilde{h}^{\la_{\da\da\da( i, j)}}_{3m-3,[3^{m-1}]}\\
\nonumber&+6\al\widetilde{h}^{2}_{2,2}\frac{1}{2}\sum_{i,d}(i-1)(i-d-3)dm_{i-3-d,d}({\la^{\ua\ua\ua(i-3-d,d)}})\widetilde{h}^{{\la^{\ua\ua\ua(i-3-d,d)}}}_{3m-3,[3^{m-1}]}\\
\nonumber&+6\al\widetilde{h}^{(1,1)}_{2,2}\frac{1}{2}\sum_{i,j,d}(i+j-3-d)dm_{i+j-d-3,d}(\la^{\da(i,j)\ua\ua(i+j-3-d,d)})\widetilde{h}^{\la^{\da(i,j)\ua\ua(i+j-3-d,d)}}_{3m-3,[3^{m-1}]}\\
\nonumber&+2\al^2\widetilde{h}^{1}_{1,1}\sum_{i,d,f}(i-3-d-f)dfm_{i-3-d-f,d,f}(\la^{\ua\ua\ua(i-3-d-f,d,f)})\widetilde{h}^{\la^{\ua\ua\ua(i-3-d-f,d,f)}}_{3m-3,[3^{m-1}]}\\
&+\widetilde{h}^{(1,1,1)}_{3,3}\sum_{i,j,k}(i+j+k-3)m_{i+j+k-3}(\la_{\da\da\da(i,j,k)})\widetilde{h}^{{\la_{\da\da\da(i,j,k)}}}_{3m-3,[3^{m-1}]}.
\end{align*}
\end{lem} 
\subsection{Proof of Theorem \ref{thm : A}}
\label{sec : corA}
We proceed with the proof of Theorem \ref{thm : A}. To this end we define recursively a weight function $\wt$ for labelled matchings such that for any partitions $\la$ and $\nu$ such that at most one part of $\nu$ is strictly greater than $3$, the polynomial $\mathfrak{S}_{\nu}^\la (\beta)$ defined as 
$$\mathfrak{S}_{\nu}^\la (\beta) = \sum_{\de \in \wit{\GG}^\la_{\nu}} \beta^{\wt(\de)}$$
verifies the same recurrence relation as the coefficient $\wit{a}^\la_{n, \nu}(1+\beta)$.
To construct this weight function we define a (multi-) {\bf edge deletion procedure} for labelled matchings. For any integer partitions $\la,\nu \vdash n$ such that the smallest part of $\nu$ is equal to $i\geq 1$ (i.e. $m_j(\nu) = 0$ for $j<i$) we associate to a labelled matching $\de \in \wit{\GG}^\la_{\nu}$, a labelled matching $\de' \in  \wit{\GG}^{\la'}_{\nu'}$ for some $\la'$ and $\nu'$ obtained by deleting successively all the edges $(u,\delta(u))$ in the labelled graph $\Gamma_{\la,\nu}(\de)$ induced by ${\bf g}_{n}$, ${\bf b}_\la$ and $\de$ where $u$ run over all the vertices belonging to the cycle of length $2i$ of ${\bf b}_\la\cup\bar{\delta}$ with the greatest label $m_i(\nu)$ (in the case $i=1$ the cycles are not labelled and any canonical choice for the order of edge deletion is acceptable).
New (labelled) matchings ${\bf g}_{n-1}$, ${\bf b}_{\la'}$ and $\de'$ are obtained by
\begin{enumerate}[label=(\roman*)]
 \item deleting the vertices $u$, $\delta(u)$ and all their incident edges, 
 \item drawing a new gray edge between vertices $\wh{u}$ and ${\bf g}_{n}\circ\de(u)$ 
 \item  if $\de(u)\neq {\bf b_\la}(u)$ (i.e. $i\neq 1$) drawing a new black edge between ${\bf b_\la}(u)$ and ${\bf b}_{\la}\circ\de(u)$.  
 \end{enumerate}
It is easy to show that the resulting graph does not depend on the order of deletion of the edges within a given cycle of length $2i$ of ${\bf b}_\la\cup\bar{\delta}$. Once all the edges of this given cycle are deleted, relabel the remaining vertices in some canonical way (note that the cycles' labels are not changed).\\
One can show that thanks to iteration of item (ii) $\Lambda({\bf g}_{n-1},\de') = (n-i)$. Note that the procedure may imply that some vertices with a non-hat index in $\Gamma_{\la,\nu}(\de)$ are eventually relabelled with a hat index and vice-versa. 
\begin{exm}
Figure \ref{fig : EdRe} illustrates the application of the edge removal procedure to a labelled matchings in $\wit{\GG}^{(4,2)}_{[2^3]}$. Note that among the remaining vertices, the one indexed $\wh{2}$ (resp. the one indexed $3$) in the original labelled matching on the left-hand side is relabelled with a non-hat (resp. hat) index.
\begin{figure}[htbp]
\begin{center}
 \includegraphics[scale=0.25]{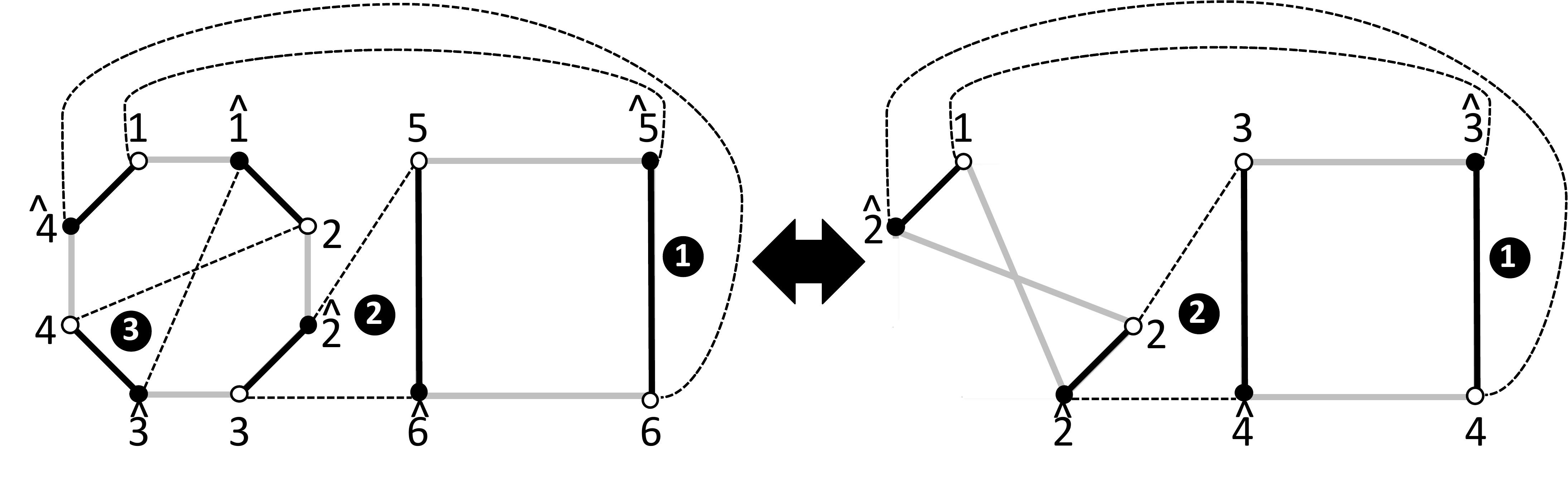}
 \caption{Application of the multi-edge deletion procedure to a labelled matching $\de \in \wit{\GG}^{(4,2)}_{[2^3]}$. The resulting labelled matchings $\de'$ belongs to $\wit{\GG}^{(2,2)}_{[2^2]}$}
\label{fig : EdRe} \end{center}
 \end{figure}
\end{exm}
Suppose now the partition $\nu$ is of the form $\nu = [1^k2^l3^mn^1]$ where $n>3$. When $k=l=m=0$, the definition of the weight function $\wt$ and the proof of equality  $$\wit{a}^\la_{n,n}(1+\be) = \mathfrak{S}_{n}^\la (\beta)$$
is provided in \cite{KanVas16}. We proceed by defining $\wt$ in the cases $i=1,2,3$ and by proving that if the equality
\begin{equation}
\wit{a}^\la_{n,\rho}(1+\be) = \mathfrak{S}_{\rho}^\la (\beta)
\label{eq : arho}
\end{equation}
is true for any partition $\la \vdash n$ and for some partition $\rho \vdash n$ with all its parts strictly greater than $i$ then the same equality is true if we replace $\rho$ by $\rho\cup[i^j]$ where $j\geq 0$ and $n$ by $n+ij$. It will be more convenient to give the proof first for $i=1$ and then increase the value of $i$.

\subsubsection{Case $i=1$}

Assume Equation (\ref{eq : arho}) is true for some partition $\rho$ such that $1\notin \rho$. We show that this equality remains true if we replace $\rho$ by $\rho \cup [1^k]$.\\
Indeed, let $\delta$ be a labelled matching of $\wit{\GG}^{\la}_{\rho \cup [1^k]}$. The edges of $\bar{\delta}$ belonging to a cycle of length $2$ of ${\bf b}_\la \cup \bar{\delta}$ are always bipartite (they link vertices $u$ and ${\bf b}_\la(u)$ for some index $u$). If we further suppose that the two vertices to such an edge belongs to a cycle of length $2\la_i$ in ${\bf b}_\la\cup{\bf g}_n$, one obtains a labelled matching $\delta' \in \wit{\GG}^{\la_{\da(\la_i)}}_{\rho \cup [1^{k-1}]}$ by removing it. The edge is bipartite and we define
$$\wt (\de) = \wt(\de').$$
The successive removal of all the edges belonging to the cycles of length $2$ of ${\bf b}_\la \cup \bar{\delta}$ becomes bijective if one keeps track of the non-hat index among the two incident vertices. One has
\begin{align*}
\mathfrak{S}_{\rho\cup[1^k]}^\la (\beta)=\sum_{\de\in \wit{\GG}^\la_{\rho\cup [1^k]}} \be^{\wt(\de)}&=\sum_{\substack{\kappa=(k_1,\ldots, k_p)\\0\leq k_i<\la_i,\\ \sum_i k_i = k }}
\binom{\la_1}{k_1}\ldots \binom{\la_p}{k_p} \sum_{\de' \in \wit{\GG}^{\la - \kappa}_{\rho} }\be^{\wt(\de')}\\
&=\sum_{\substack{\kappa=(k_1,\ldots, k_p)\\0\leq k_i<\la_i,\\ \sum_i k_i = k }}
\binom{\la_1}{k_1}\ldots \binom{\la_p}{k_p} \wit{a}^{\la-\kappa}_{n,\rho}(1+\beta)\\
&= \wit{a}^{\la}_{n,\rho\cup[1^k]}(1+\be),
\end{align*}
where the last equality is given by Equation (\ref{rhon}) and the previous one is the recurrence hypothesis.

\subsubsection{Case $i=2$} 

Given an integer partition $\la$, iterating twice the edge removal procedure on a labelled matching $\de$ in $\wit{\GG}^\la_{\rho\cup[2^{l}]}$ where $l\geq 1$ and $\rho$ is such that $1,2\notin \rho$ yields a new labelled matching $\de'$ in $\wit{\GG}^{\la'}_{\rho\cup[2^{l-1}]}$ for some partition $\la'$ (we assume that either $\rho \neq \emptyset$ or $\rho = \emptyset$ and $l>1$ otherwise the problem reduces to a known case of \cite{KanVas16}). We define three bijective constructions for labelled matchings $\de$ in $\wit{\GG}^\la_{\rho\cup[2^{l}]}$ based on this procedure that depend on the properties of the cycle of length $4$ in ${\bf b}_\la \cup \de$ with label $l$ that we denote $\Box_l$. In the sequel, we say that $\Box_l$ is bipartite if its four incident vertices $u,v,\delta(u),\delta(v)$ are such that both $(u,\de(u))$ and $(v,\de(v))$ are bipartite. 
Define:
$$\wt(\de):=
\begin{cases} {\wt}({\de'}) & \text  { if } \Box_l \text{ is bipartite},\\
{\wt}({\de'})+1 & \text { otherwise. }
\end{cases}$$

We have three cases.

\begin{itemize}
\item $\Box_l$ is not bipartite and its black edges lie inside the $i$-th cycle of length $2\la_i$ in ${\bf g}_{2m}\cup{\bf b}_\la$.  In this case, one can show that $\la_i>2$, $\la'=\la_{\da\da\la_i}$ and there is a bijection between such labelled matchings in $\wit{\GG}^\la_{\rho\cup[2^{l}]}$ and triples $(\de', a,b)$ where $\de' \in \wit{\GG}^{\la_{\da\da\la_i}}_{\rho\cup[2^{l-1}]}$ and $a$ and $b$ are the non-hat labels in $\Box_l$ (we have $\sum_j^{i-1}\la_j\leq a<b \leq \sum_j^{i}\la_j$). In this case $\wt(\de) = 1+\wt(\de')$. 

\item $\Box_l$ is bipartite and its black edges lie inside the $i$-th cycle of length $2\la_i$ in ${\bf g}_{2m}\cup{\bf b}_\la$. We have $\la_i>3$, $\la'=\la^{\ua\ua(\la_i-d-2,d)}$ for some $d$ and there is a bijection between such labelled matchings in $\wit{\GG}^\la_{\rho\cup[2^{l}]}$ and pairs $(\de', c)$ where $\de' \in \bigcup_d\wit{\GG}^{\la^{\ua\ua(\la_i-d-2,d)}}_{\rho\cup[2^{l-1}]}$ and, if $a<b$ are the non-hat labels in $\Box_l$, $c = a$ if $b-a\leq \la_i/2$ and $c = b$ otherwise. In this case $\wt(\de) = \wt(\de')$. 

\item  The black edges of $\Box_l$ lie in two different cycles, namely the $i$-th and the $j$-th one. In this case, $\la_i+\la_j>2$, $\la'=\la_{\da\da(\la_i,\la_j)}$ and there is a bijection between such labelled matchings in $\wit{\GG}^\la_{\rho\cup[2^{l}]}$ and quadruples $(\de', a,b,i)$ where $\de' \in \wit{\GG}^{\la_{\da\da(\la_i,\la_j)}}_{\rho\cup[2^{l-1}]}$, $a$ and $b$ are the non-hat labels in $\Box_m$ and $i\in\{b,nb\}$ indicates whether $\Box_l$ is bipartite or not (the two cases are always possible). In this case $\wt(\de) = \wt(\de')$ if $i=b$ and $\wt(\de) = 1+\wt(\de')$ otherwise. 

\end{itemize}

Using these three bijections (see Figure \ref{fig : bijg2m} for an illustration), one gets:
\begin{align*}
\mathfrak{S}^\la_{\rho\cup[2^{l}]}(\beta) &= \sum_{\substack{i\geq 1\\ \la_i>2}}\hspace{-0mm}\binom{\la_i}{2}\hspace{-2mm}\sum_{\de' \in \widetilde{\GG}^{{\la}_{\da\da(\la_i)}}_{\rho\cup[2^{l-1}]}}\hspace{-2mm}\beta^{1+\wt(\de')}\\
&+\frac 1 2 \sum_{i,d\geq 1}\la_i \hspace{-2mm}\sum_{\de' \in \widetilde{\GG}^{\la^{\ua\ua(\la_i-d-2,d)}}_{\rho\cup[2^{l-1}]}}\hspace{-6mm}\beta^{\wt(\de')}+\hspace{-2mm}\sum_{\substack{i<j\\ \la_i+\la_j>2}}\hspace{-0mm}\la_i\la_j\hspace{-4mm}\sum_{\de' \in \widetilde{\GG}^{\la_{\da\da(\la_i,\la_j)}}_{\rho\cup[2^{l-1}]}}\hspace{-4mm}\left(\beta^{\wt(\de')}+\beta^{1+\wt(\de')}\right)
\end{align*}
As a result, $\mathfrak{S}^\la_{\rho\cup[2^l]}(\beta) = \wit{a}^\la_{n+2l,\rho\cup[2^l]}(1+\beta)$ for all $l\geq 1$ and any $\la \vdash n+2l$. 
\begin{figure}[htbp]
\begin{center}
\subfigure[First bijection.]
{\includegraphics[scale = 0.17]{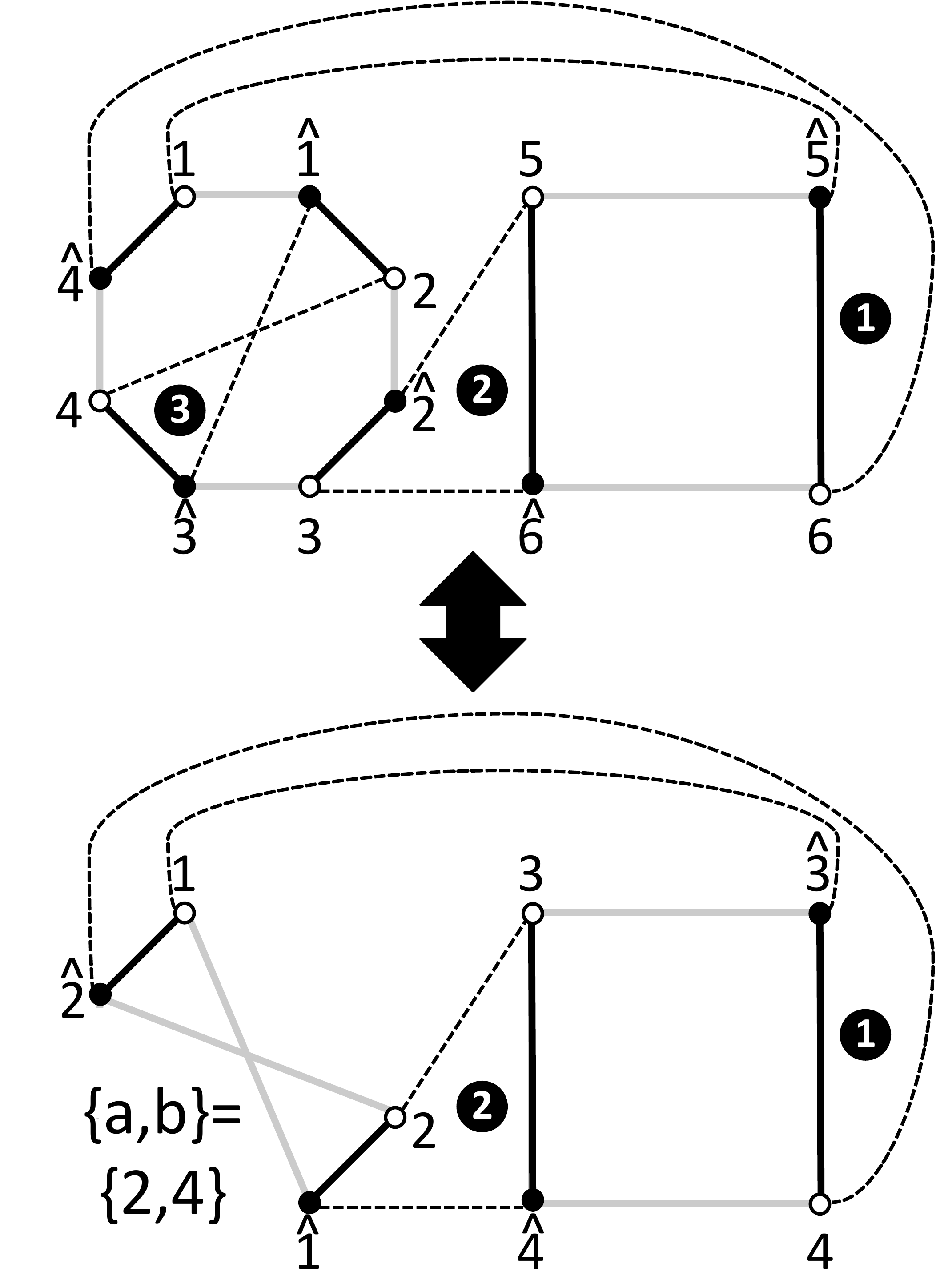}}\qquad
\subfigure[Second bijection.]
{\includegraphics[scale = 0.17]{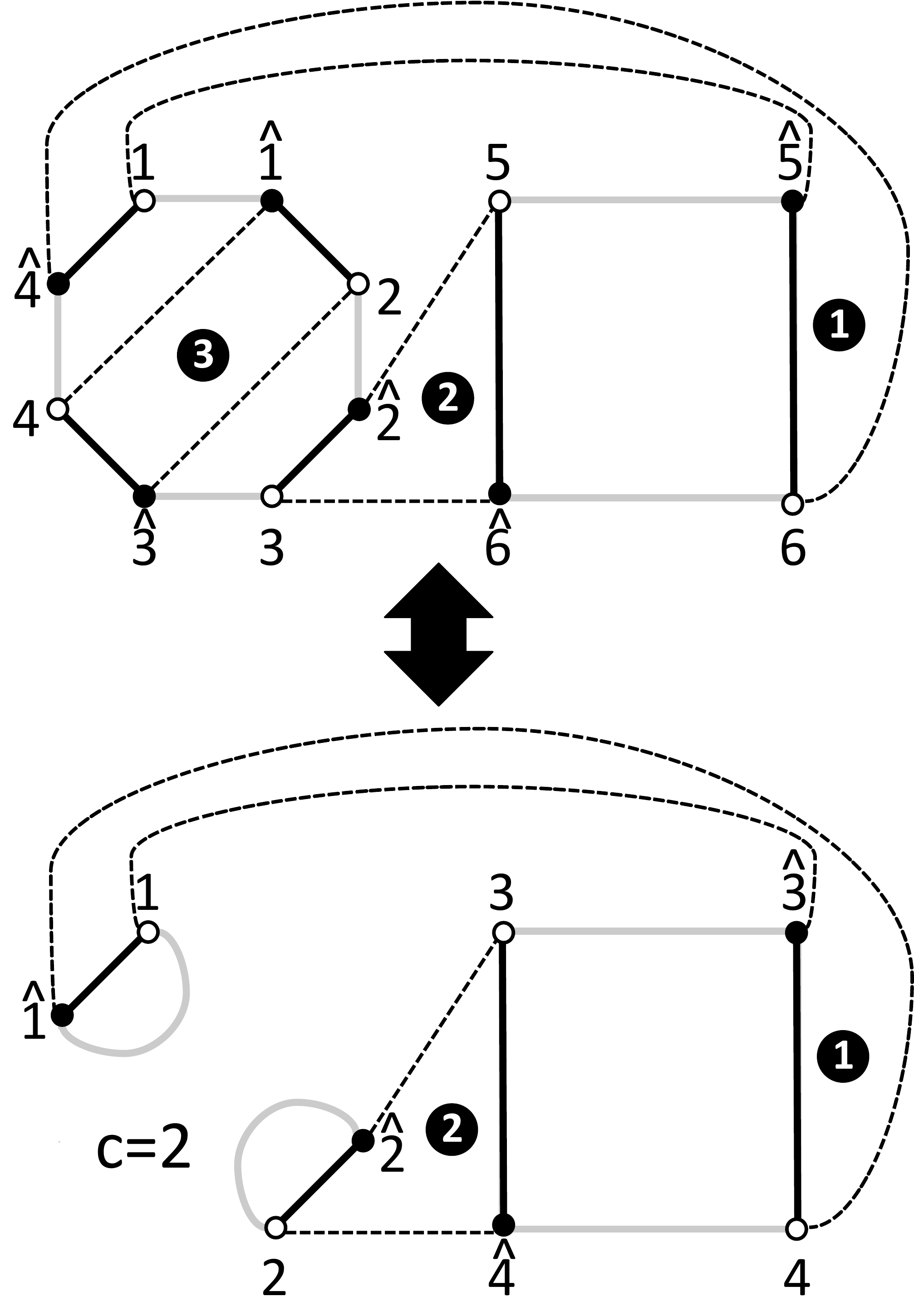}}\qquad
\subfigure[Third bijection.]
{\includegraphics[scale = 0.17]{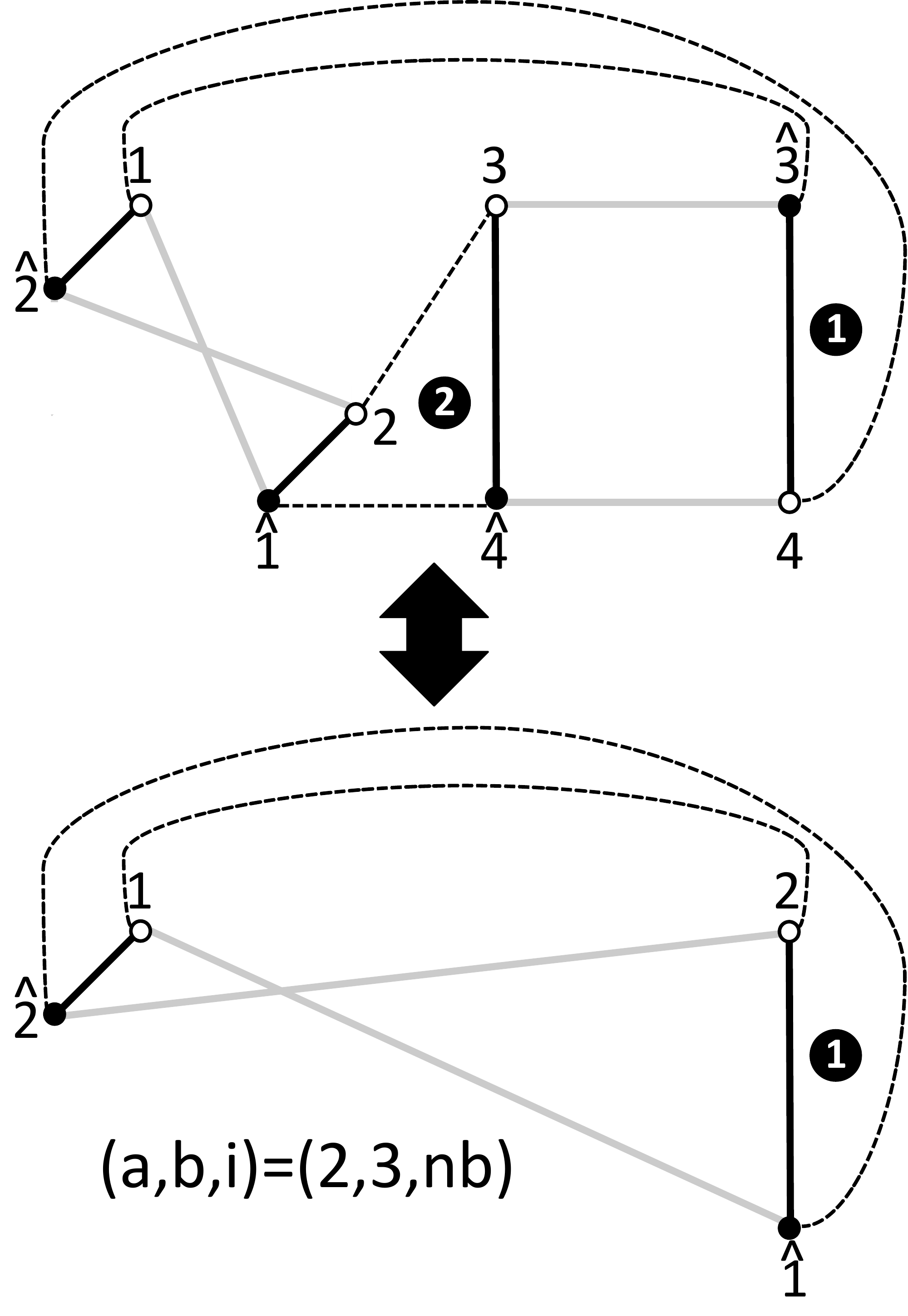}}\caption{Examples of application of the three bijections for labelled matchings.}
\label{fig : bijg2m}
\end{center}
\end{figure}

\subsubsection{Case $i=3$}

\begin{lem}
If Theorem \ref{thm : A} holds for $k=l=m=0$ then it holds for $k=l=0$ and any $m$. 
\end{lem}

\begin{proof} The statement follows from the recurrence formula in Lemma \ref{rec3a} by induction on $m$ like in the case of twos. For each hexagon 
$\hex$ in a $\la$-graph $G$, we fix a  bijection between labelled matchings from 
$\wit{\GG}^\la_{N,\rho\cup[3^m]}$ with $m$-th hex $\hex_m=\hex$ and labelled matchings from 
$\wit{\GG}^\la_{N-3,\rho\cup[3^{m-1}]}$ on the graph obtained from $G$ by removing of $\hex$.
 Let $\de\in \wit{\GG}^\la_{N,\nu}$ be a given labelled matching   and $\de'$ is obtained from $\de$ by deleting the last $m$-th hexagon $\hex=\hex_m$.
In each case, we define 
 $$\wt(\de):=\wt(\de')+w$$ where {\bf adding weight} $w\in \{0,1,2\}$ and $w=0$ iff $\hex$ is bipartite else $w=1$ or $2$ in such way that the common contribution of all hexagons of considering type be equal to the coefficient at the corresponding summand (\ref{s1})--(\ref{s6}).
 We explain the combinatorial sence of each summand and its contribution to the sum $\sum_{\de} \be^{\wt(\de)}$ depending on the position of $\hex$. 
 
\begin{itemize}
\item {\bf Summand (\ref{s1})} (Figure \ref{fig : s1}): black edges of $\hex$ lie in the same $\la_i$ component 
of $G$ for some $\la_i>3$, and these edges are connected  in such way that $\la'=\la_{\da\da\da (\la_i)}$.  There exist $\binom{\la_i}{3}$ triples of edges of the $\la_i$-component and, for each of them, there exist four such hexagons. One of them is bipartite so we define $\wt(\de)=\wt(\de')$. To {\it any one} of another hexagons, we assign a weight  $\wt(\de')+1$. The matchings with two last hexagons get weights  $\wt(\de')+2$. So the common contribution is $a_{33}^3=2\be^2+\be+1$.

\begin{figure}[htbp]
\begin{center}
\includegraphics[scale=0.8]{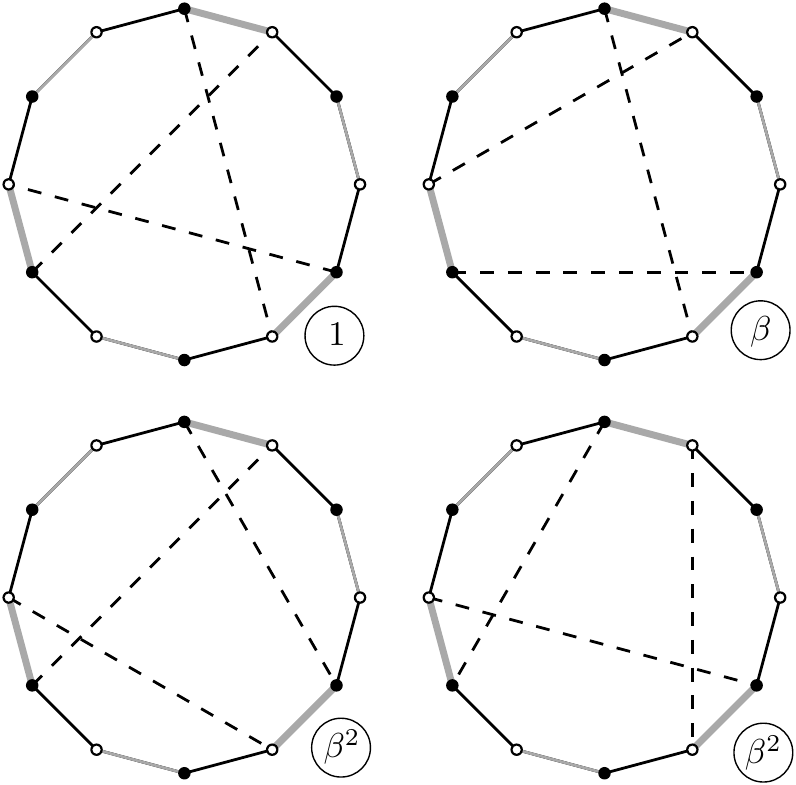} 
\caption{Hexagons contributing to summand (\ref{s1}).} \label{fig : s1}
\end{center}
\end{figure}

\item {\bf Summand (\ref{s2})} (Figure \ref{fig : s2}): black edges of $\hex$ lie in three different components, say in $\la_i,\la_j,\la_k$, $i<j<k$. (Of course $\la_i+\la_j+\la_k>3$ else $\de\cup${\bf b} cannot be a cycle.) Clearly there are $(2\la_i)(2\la_j)(2\la_k)$ such hexagons and exactly $(2\la_i)\la_j\la_k$ among them are bipartite. For fixed three edges  
there exist eight hexagons on them, besides exactly two of these hexagons are bipartite. So we assign adding weights $0,0,1,1,1,1,2,2$ to matchings with these hexagons according to the coefficient 
$2\be^2+4\be+2$ in (\ref{s2}).

\begin{figure}[htbp]
\begin{center}
\includegraphics[scale=0.8]{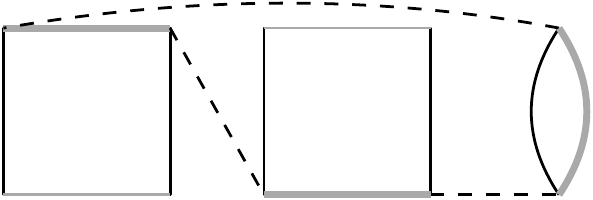}
\caption{Example of hexagon contributing to summand (\ref{s2}).} \label{fig : s2}
\end{center}
\end{figure}

\item {\bf Summand (\ref{s3})} (Figure \ref{fig : s3}): black edges of $\hex$ are bipartite and lie in the same component of $\de$ (say, $\la_i$) besides they are not neighbours and their nearest vertices are connected. This case is similar to the second subcase from the case $i=2$.

\begin{figure}[htbp]
\begin{center}
\includegraphics[scale=0.7]{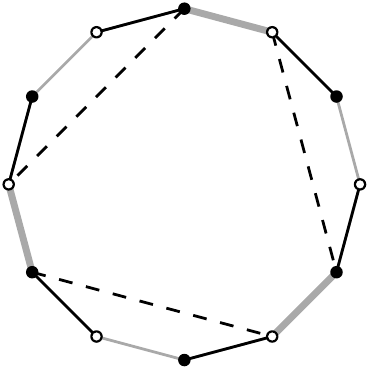} 
\caption{Example of hexagon contributing to summand (\ref{s3}).}\label{fig : s3}
\end{center}
\end{figure}

\item {\bf Summand (\ref{s4})} (Figure \ref{fig : s4}): two black edges (call $ab$ and $cd$) of $\hex$ are from the one $\la$-component and the third black edge (call $ef$) is from another one, besides exactly two one-colored vertices from the first component are connected in $\de$, i.~e. these six vertices are connected in $\de$ in one of four ways:
$\{ad, be,cf\}$, $\{ad,bf,ce\}$, $\{bc, ae,df\}$, $\{bc,af,de\}$. One can choose these three $\la$-edges in
$\binom{\la_i}{2} \la_j+\binom{\la_j}{2} \la_i=\frac{\la_i\la_j(\la_i+\la_j-2)}{2}$ ways for fixed $i<j$ (besides $\la_i+\la_j>2$). We distribute their adding weights as $1,1,2,2$ in order to the common contribution be equal $2\be^2+2\be$.

\begin{figure}[htbp]
\begin{center}
\includegraphics[scale=0.7]{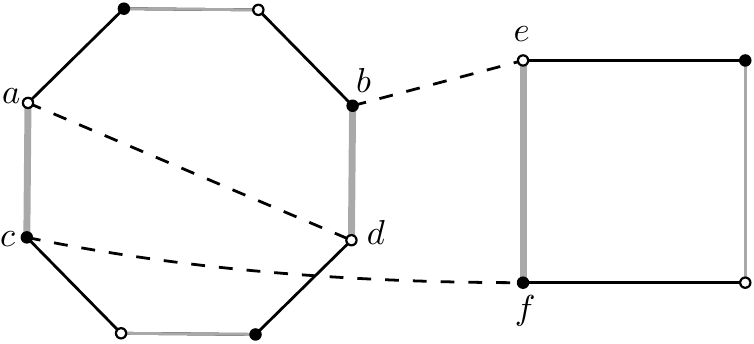} 
\caption{Example of hexagon contributing to summand (\ref{s4}).}\label{fig : s4}
\end{center}
\end{figure}

\item {\bf Summand (\ref{s5})} (Figure \ref{fig : s5}). We fix $i<j$ and $d\in \{1,\ldots, \la_i+\la_j-4\}$. Firstly, we choose a vertex $a$ from $\la_i$-component and a vertex $b$ from $\la_j$-component (we can do it in $(2\la_i)(2\la_j)$ ways) and remove the edge $\{a,b\}$ getting $(\la_i+\la_j-1)$-component. Let $\{a,a'\}, \{b,b'\}\in {\bf b}$.
Secondly, we skip $d$ gray edges in this connected component from the vertex $b'$ in the direction $a'\to b'$ (on Figure \ref{fig : s5} $d=2$). So we get the edge $\{c,c'\}\in {\bf b}$ such that the edge $\{b', c\}$ is bipartite and we have the hexagon $abb'cc'a'$. Note that the same hexagon can be obtained by choosing $a'$ and $c'$ (instead of $a$ and $b$) 
therefore we divide by 2 the number $(2\la_i)(2\la_j)$ of ways to choose two starting vertices. Clearly, the half of all hexagons of this type are non-bipartite so we define their adding weights as 1.

\begin{figure}[htbp]
\begin{center}
\includegraphics[scale=0.7]{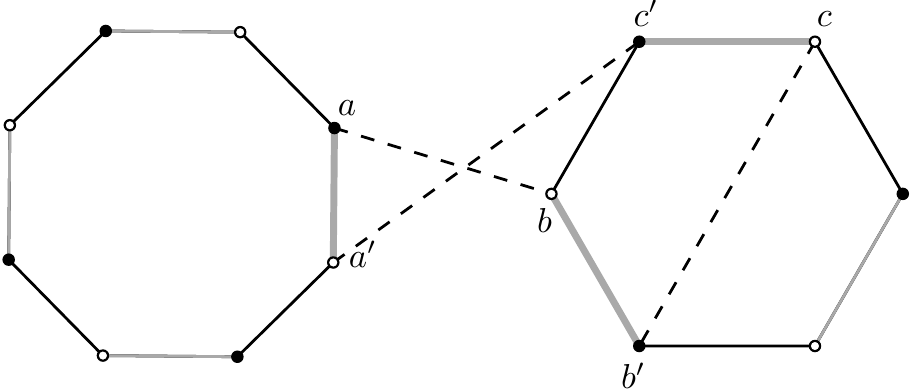}
\caption{Example of hexagon contributing to summand (\ref{s5}).} \label{fig : s5}
\end{center}
\end{figure}

\item {\bf Summand (\ref{s6})} (Figure \ref{fig : s6}, left): the $\la$-edges of $\hex$ lie in one component, say in $\la_i$, besides just one from $\de$-edges of $\hex$ is bipartite. We choose a pair of black edges in $\la_i$-component ($1\wh{1}$ and $4\wh{4}$ on Figure \ref{fig : s6}, left) and connect their black vertices or white vertices (we can do it in $2\binom{\la_i}{2}$ ways) and get $(\la_i-1)$-component in which we have to remove a bipartite square (after recoloring) with fix gray edges (the images of $2\wh{1}$ and $4\wh{5}$ after removing $14$). For obtaining the hex of suitable type in this case, we need to choose one more black edge of $\la_i$-component. This choice is equivalent to the choice of $d\in \{1,\ldots, \la_i-4\}$. Note that the same hexagon is obtained by the choice of one more pair of its $\la$-edges 
(see Figure \ref{fig : s6}) so we divide the number $2\binom{\la_i}{2}$ by 2. All such hexagons get an adding weight $1$.

\begin{figure}[htbp]
\begin{center}
\includegraphics[scale=0.7]{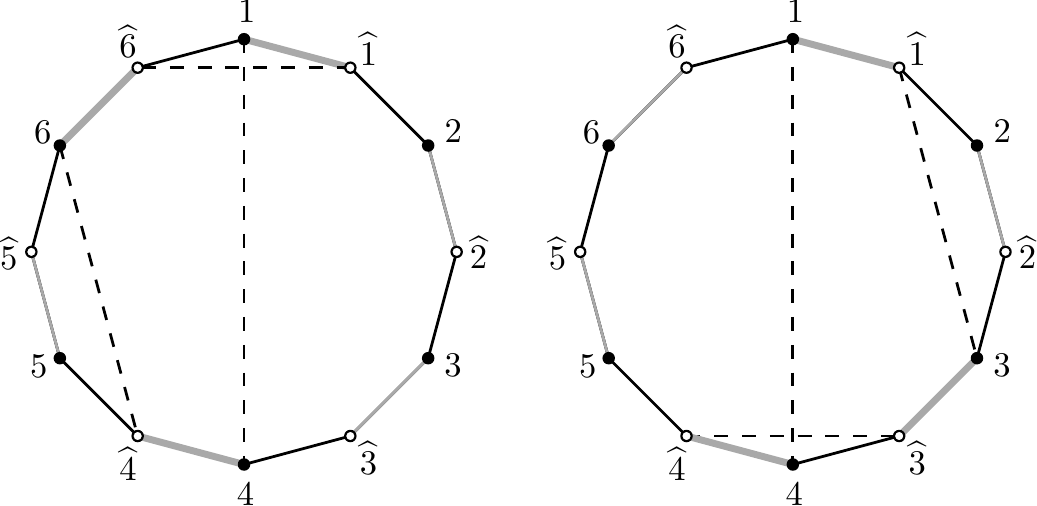}
\caption{Both hexagons for summand at $d\in \{1,2\}$ with edge $14\in \de$. The left (resp. right) one can be obtained by the choice of the pair $\{1\wh{1}, 4\wh{4}\}$ or $\{1\wh{1}, 6\wh{6}\}$  
(resp. $\{1\wh{1}, 4\wh{4}\}$ or $\{3\wh{3}, 4\wh{4}\}$).} \label{fig : s6}
\end{center}
\end{figure}
\end{itemize}
\end{proof}

We finish the proof of Theorem \ref{thm : A} by noting that the inequality $\wt(\de)\leq N-\ell(\nu)$ holds for all $\la,\nu\vdash N$ and $\de\in \wit{\GG}^\la_{N,\nu}$. Indeed, 
by assumption it holds for $k=l=m=0$. As one can see from the proof the adding weight of $\de$ is increasing no more $d-1$ when we add a part $d\in \{1,2,3\}$ to the partition $\nu$.

\subsection{Proof of Theorem \ref{thm : H}}
This section is dedicated to the proof of Theorem  \ref{thm : H}. We adapt the definition of the {\it non-orientability} measure $\vartheta$ for hypermaps introduced by Lacroix \cite{Lac09} and use some bijective constructions on labelled star hypermaps to prove that the polynomials $\Sigma^\la_{[k^m]}(\beta)$ defined as $$\Sigma^\la_{[k^m]}(\beta) = \sum_{M \in \widetilde{\mathcal{L}}^\la_{[k^m]}}\be^{\vartheta(M)}$$ verify the same recurrence relation as the coefficients $\wit{h}^\la_{n, [k^m]}(1+\beta)$. Although it shares some similarities with the proof of Theorem  \ref{thm : A}, the differences in the combinatorial objects involved (hypermaps instead of matchings) render the proof of Theorem \ref{thm : H} independent. We first define the measure of non-orientability used throughout the section. Then we cover the case $k=n$ as it is required to prove the theorem for the other values of $k$.
\subsubsection{Measure of non-orientability for labelled star hypermaps}
We adapt the definition of the measure of non-orientability introduced by Lacroix in \cite[Definition 4.1]{Lac09} to the case of labelled hypermaps. In what follows, we name {\bf a leaf} an edge connecting a black vertex of degree $1$ and the white vertex.
\begin{defn}[Measure of non-orientability for labelled star hypermaps]
To any labelled star hypermap $M$ of face distribution $\la \neq \emptyset$, we associate a labelled star hypermap $M'$ of face distribution $\la'$ obtained by 
\begin{itemize}
\item deleting the root edge,
\item defining the new root as the edge labelled with $2$ in M,
\item relabelling all the remaining edges by its label in M minus $1$. 
\end{itemize}
The procedure described above is called the {\bf root deletion process}. Following Lacroix, define recursively the function $\vartheta$ on labelled star hypermaps as: 
\begin{itemize}
\item If the root of $M$ is a leaf then $\vartheta(M) = \vartheta(M')$ 
\item Otherwise, the root of $M$ is not a leaf. We have $|\ell(\la')-\ell(\la)| \leq 1$ and:
\begin{itemize}
\item if $\ell(\la') = \ell(\la)$ the root of M is a {\bf cross-border} and $\vartheta(M) = 1+\vartheta(M')$, 
\item if $\ell(\la') = \ell(\la)-1$ the root of M is a {\bf border} and $\vartheta(M) = \vartheta(M')$,
\item if $\ell(\la') = \ell(\la)+1$ the root of M is a {\bf handle}. In this case, there is a second
hypermap $\tau(M)$ obtained from $M$ by twisting the ribbon associated with its root. The root of $\tau(M)$ is also a handle and deleting it from $\tau(M)$ also produces $M'$. Define

 $$\{\vartheta(M), \vartheta(\tau(M))\} = \{\vartheta(M'),1+\vartheta(M')\}.$$ 

At most one of $M$ and $\tau(M)$ is orientable, and any canonical choice such that if $M$ is orientable, then $\vartheta(M) = 0$ and $\vartheta(\tau(M))$ = 1 is acceptable.
\end{itemize}
\end{itemize}
If $M$ is the empty hypermap of face distribution $\la = \emptyset$, define $\vartheta(M) = 0$. 
\end{defn}
\begin{rem}
Some details about the topological meaning of the three types of edges (cross-border, border or handle) can be found in \cite[Remark 4.3]{Lac09}. Note that in the case of star hypermaps, the connectivity of the map is not altered after deletion of an edge and such maps do not contain bridges (except the degenerate case of leaves).
\end{rem}
\begin{exm} Figure \ref{fig : NOM} shows the first iterations of the root deletion process applied to the labelled star hypermap \ref{fig : NOM1}.
\end{exm}
\begin{figure}[htbp]
\begin{center}
\subfigure[A labelled star hypermap with two faces $f_1$ and $f_2$. \label{fig : NOM1}]{\includegraphics[scale = 0.2]{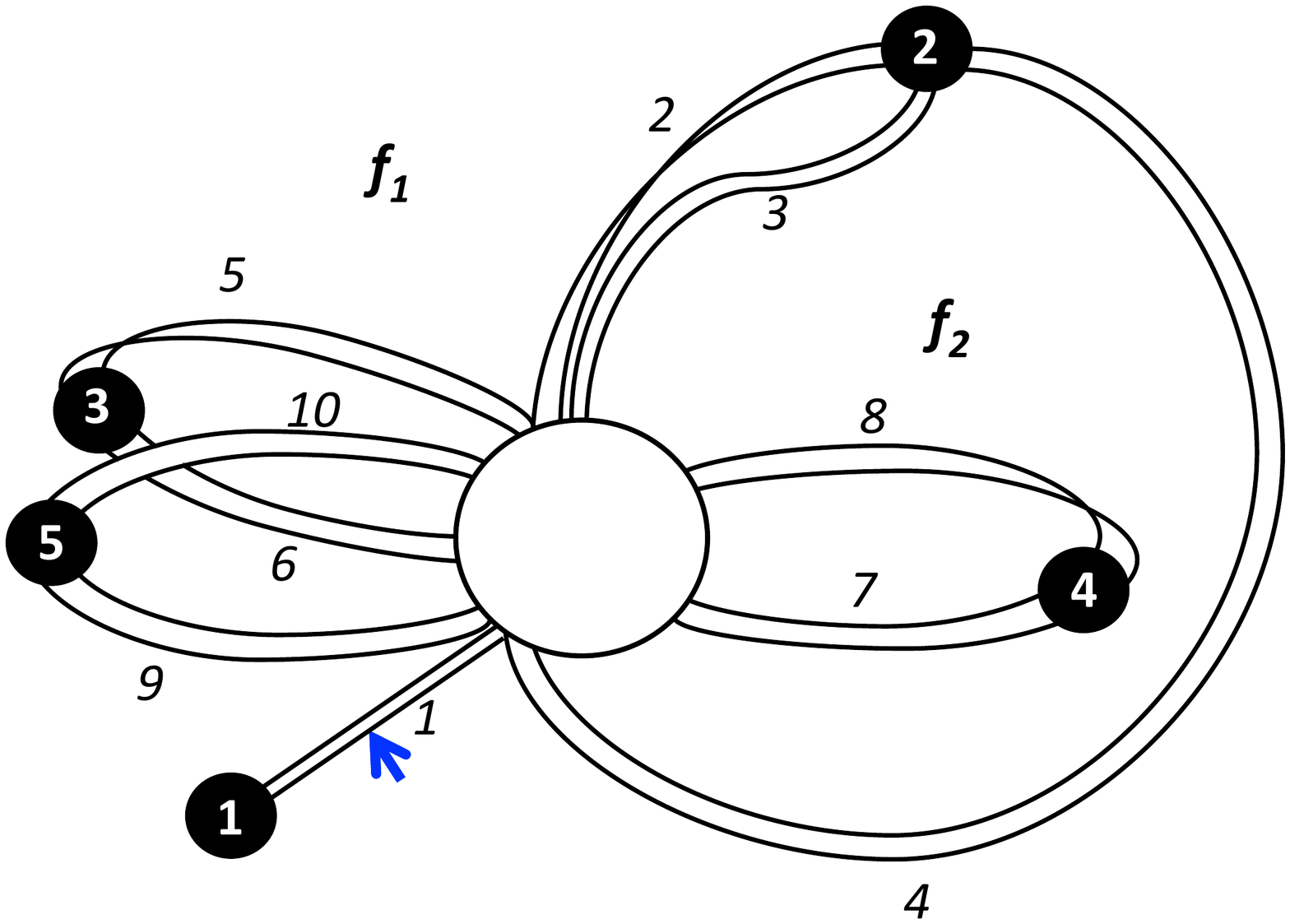}}\qquad
\subfigure[Leaf deletion. \label{fig : NOM2}]{\includegraphics[scale = 0.2]{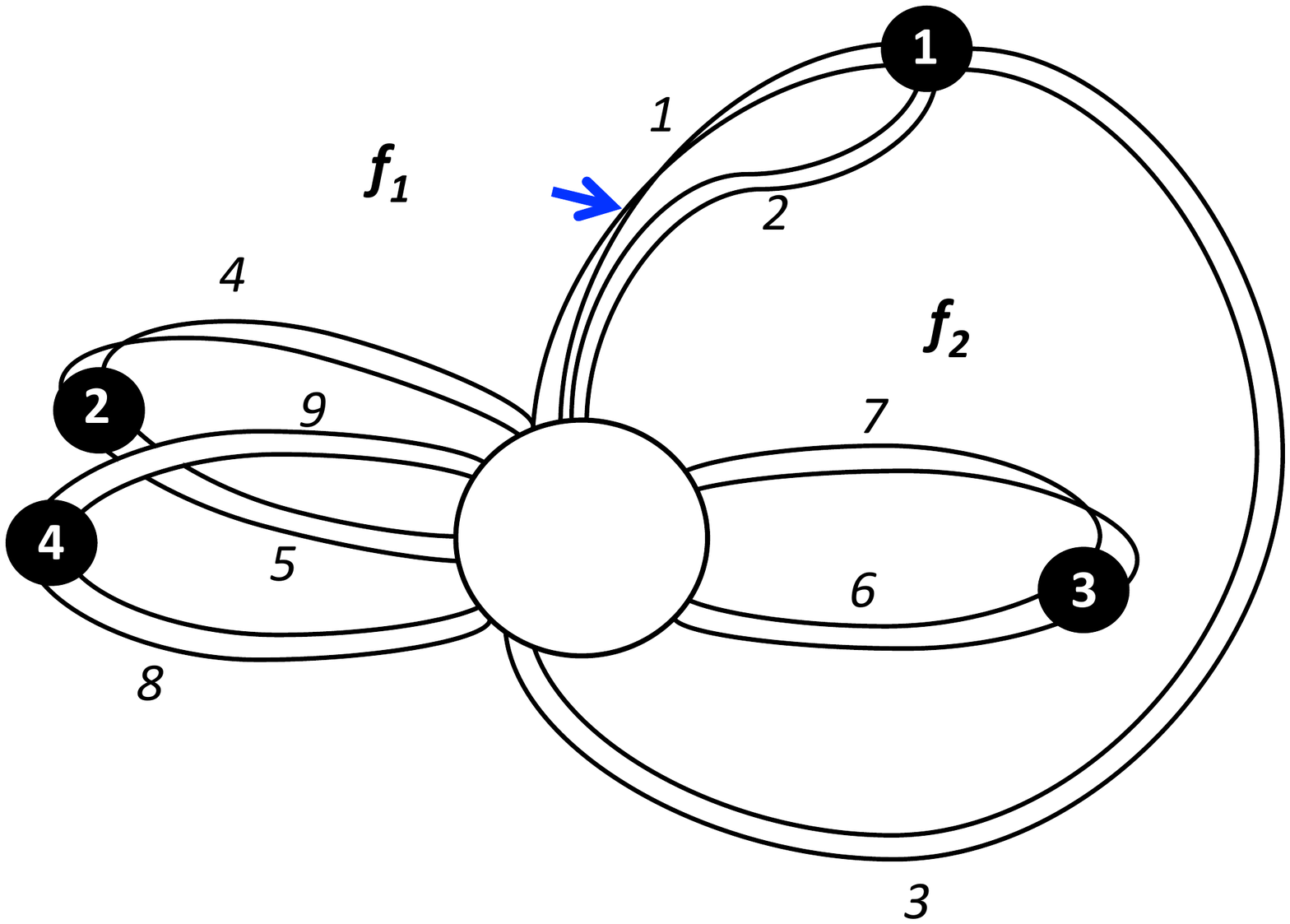}}
\subfigure[Cross border deletion.\label{fig : NOM3}]{\includegraphics[scale = 0.2]{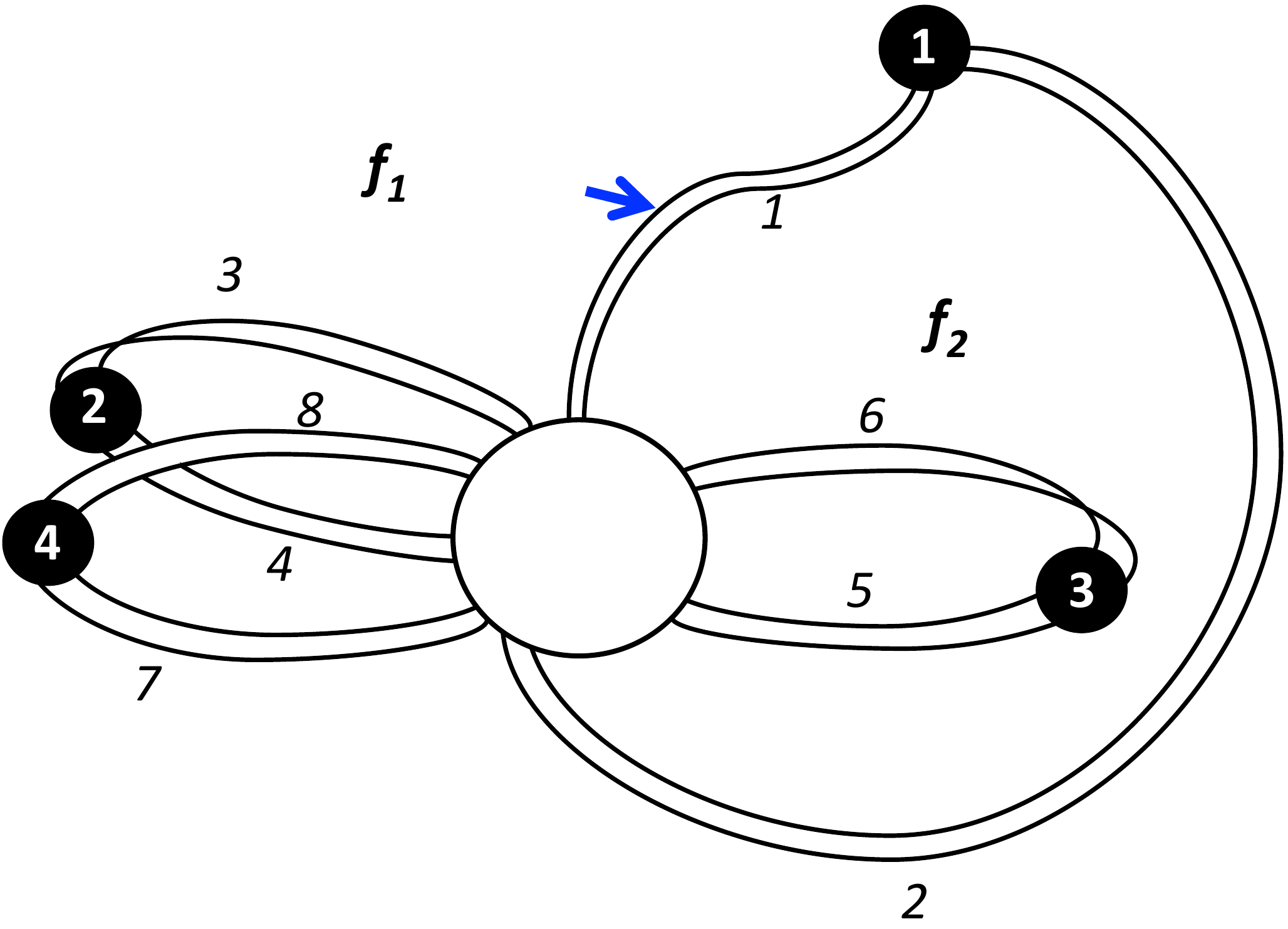}} \qquad
\subfigure[Border deletion. The two faces are merged into one single face $f_1$.\label{fig : NOM4}]{\includegraphics[scale = 0.2]{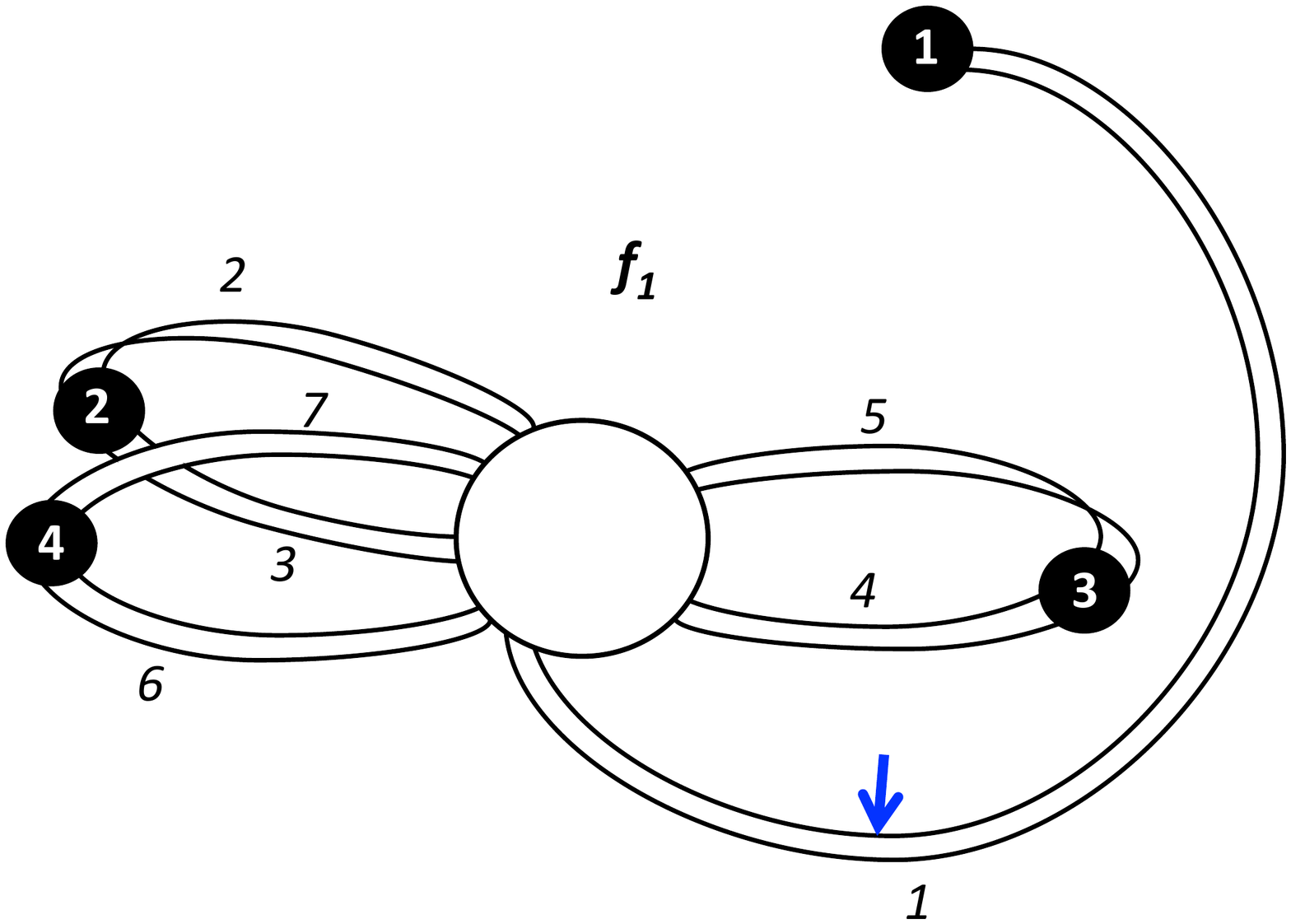}} \qquad
\subfigure[Leaf deletion.\label{fig : NOM5}]{\includegraphics[scale = 0.2]{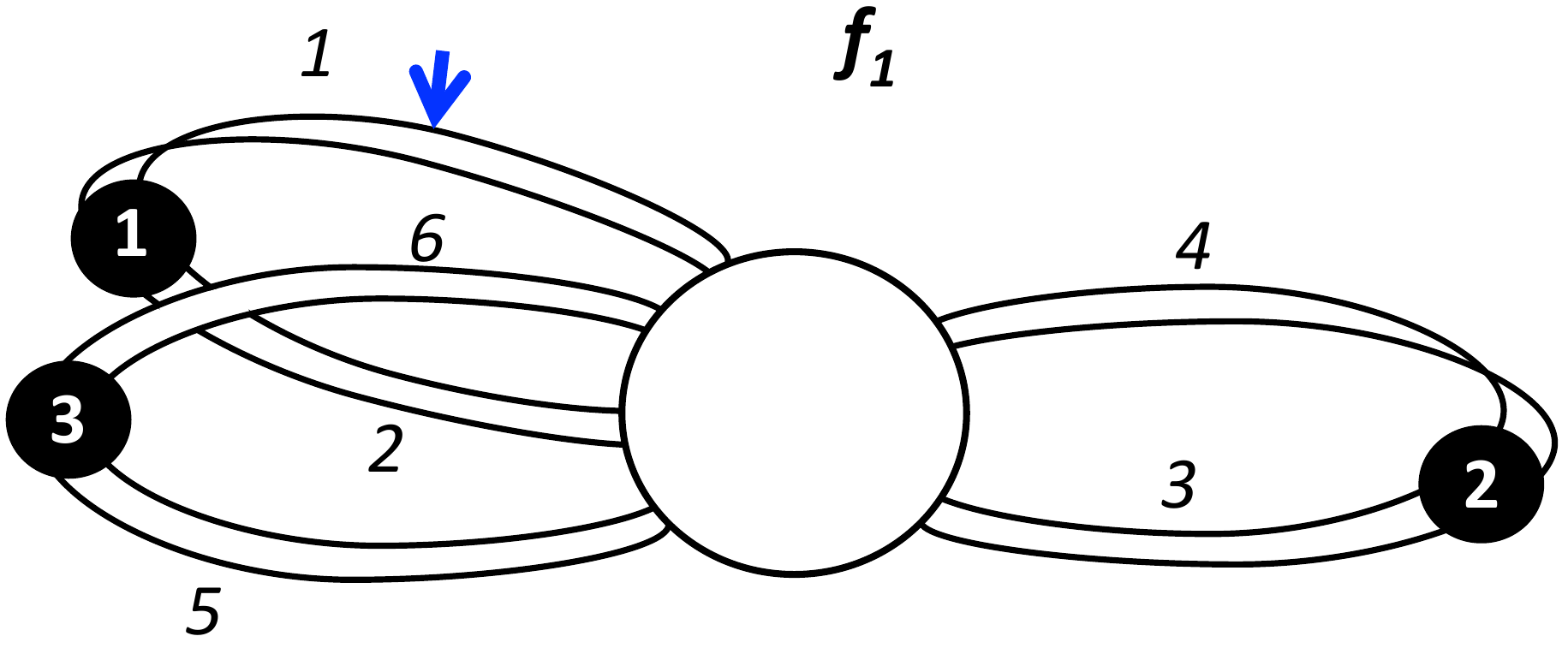}}\qquad
\subfigure[Handle deletion. A new face $f_2$ is created.\label{fig : NOM6}]{\includegraphics[scale = 0.2]{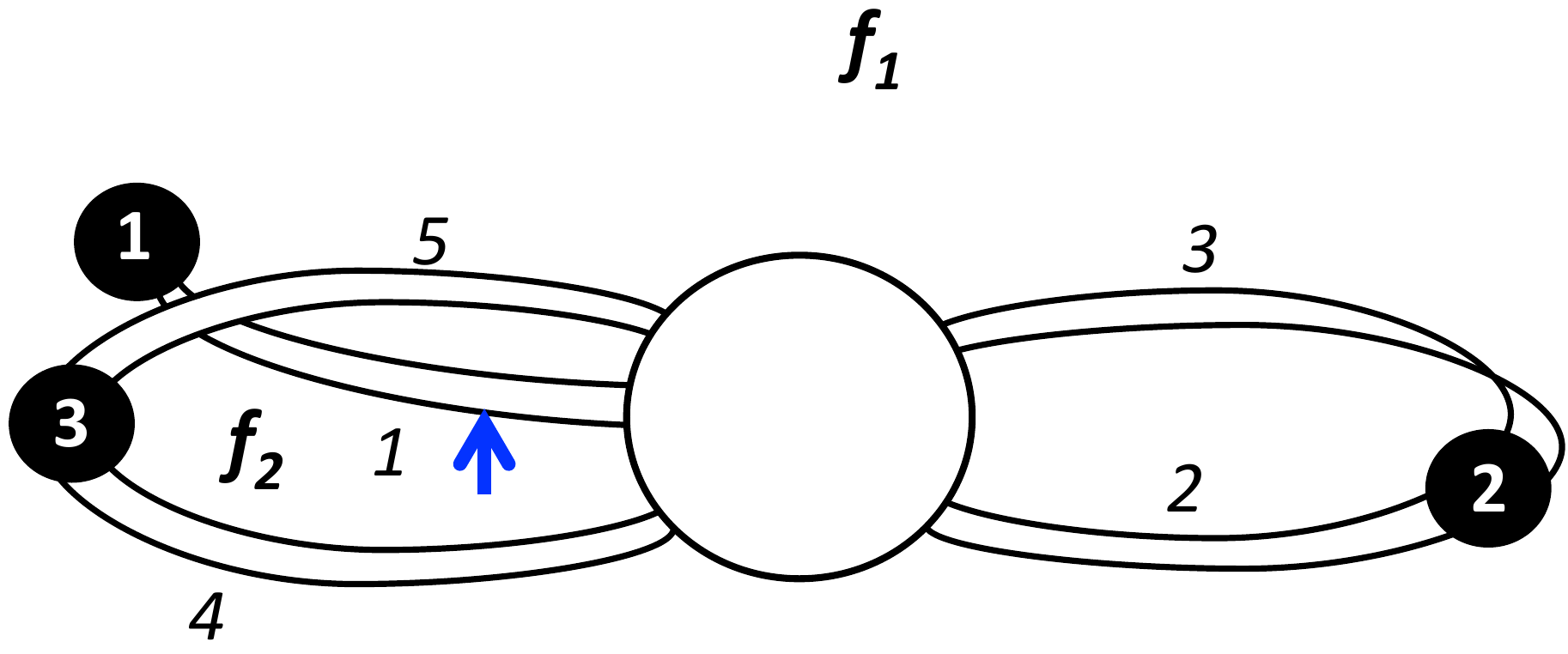}} 
\caption{Iteration of the root deletion process. The type of the deleted edge and the impact on the number of faces is mentioned below each figure.}
\label{fig : NOM}
\end{center}
\end{figure}


\begin{rem}
Note that the value of function $\vartheta$ depends on the labels of the edges. As an example the respective values of $\vartheta$ on the two labelled star hypermaps depicted on Figure \ref{fig : 2diflabMap} are $3$ and $4$. 
\begin{figure}[htbp]
\begin{center}
\includegraphics[scale = .2]{NonOrientabilityMeas1.pdf}\qquad
\includegraphics[scale = .2]{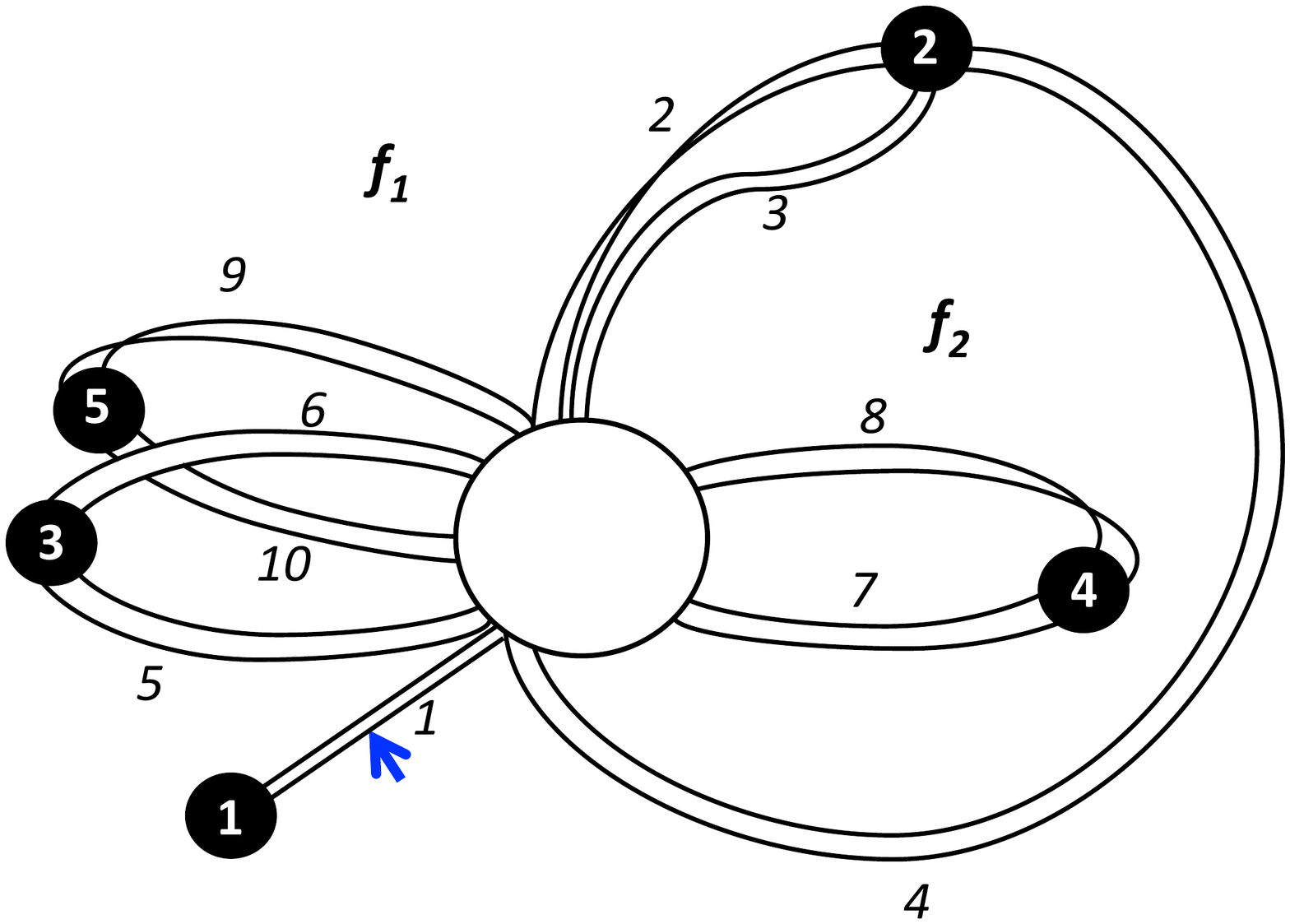}
\caption{Two labelled star hypermaps with the same underlying unlabelled map but different values of $\vartheta$.}
\label{fig : 2diflabMap}
\end{center}
\end{figure}
%
\end{rem}

\subsubsection{Case $k=n$}\label{sec : k=n}

\label{sec : bijk=n}

Now we look at the recurrence relation verified by $\Sigma^\la_{n}(\beta)$.
In the case $n=1$, $\Sigma^\la_{1}(\beta) = 1=\widetilde{h}^{1}_{1,1}(\beta+1).$ If $n>1$, star hypermaps in $\widetilde{\mathcal{L}}^\la_{n}$ do not contain any leaf and we may split these maps in three sets according to the type (cross border, border or handle) of their root edge that we delete to get the following bijective constructions. 
\begin{itemize}
\item (Cross border) The set of labelled star hypermaps with one black vertex, face distribution $\la$ and a cross border root incident to a face of degree $i$ is in bijection with the set of labelled star hypermaps with
(i) one black vertex, (ii) face distribution $\la_{\downarrow(i)}$, (iii) one marked position around the white vertex incident to a face of degree $i-1$, (iv) one marked position around the black vertex incident to the same face.
\item (Border) The set of labelled star hypermaps with one black vertex, face distribution $\la$ and a  border root incident to both a face of degree $i$ and a face of degree $j$ is in bijection with the set of labelled star hypermaps with (i) one black vertex, (ii) face distribution $\la_{\downarrow(i,j)}$, (iii) one marked position around the white vertex incident to a face of degree $i+j-1$ (once the position around the white vertex is chosen there is only one position around the black vertex such that connecting these two positions with a border cuts the face of degree $i+j-1$ into two faces of degree $i$ and $j$).
\item (Handle) The set of labelled star hypermaps with one black vertex, face distribution $\la$ and a handle root incident to a face of degree $i$ such that removing the root yields a face of degree $d$ and one of degree $i-1-d$ is in bijection with the set of labelled star hypermaps with (i) one black vertex, (ii) face distribution $\la^{\uparrow(i-1-d,d)}$, (iii) one marked position around the white vertex incident to a face of degree $i-1-d$, (iv) one marked position around the black vertex incident to a face of degree $d$ and (v) a type for the removed root: twist or untwist (as noted above, twisting the ribbon of a handle root yields another hypermap). 
\end{itemize}
\begin{exm} Figure \ref{fig : bijkn} illustrates the three bijections described above.
\begin{figure}[htbp]
\begin{center}
\subfigure[First bijection.]{\includegraphics[scale = 0.28]{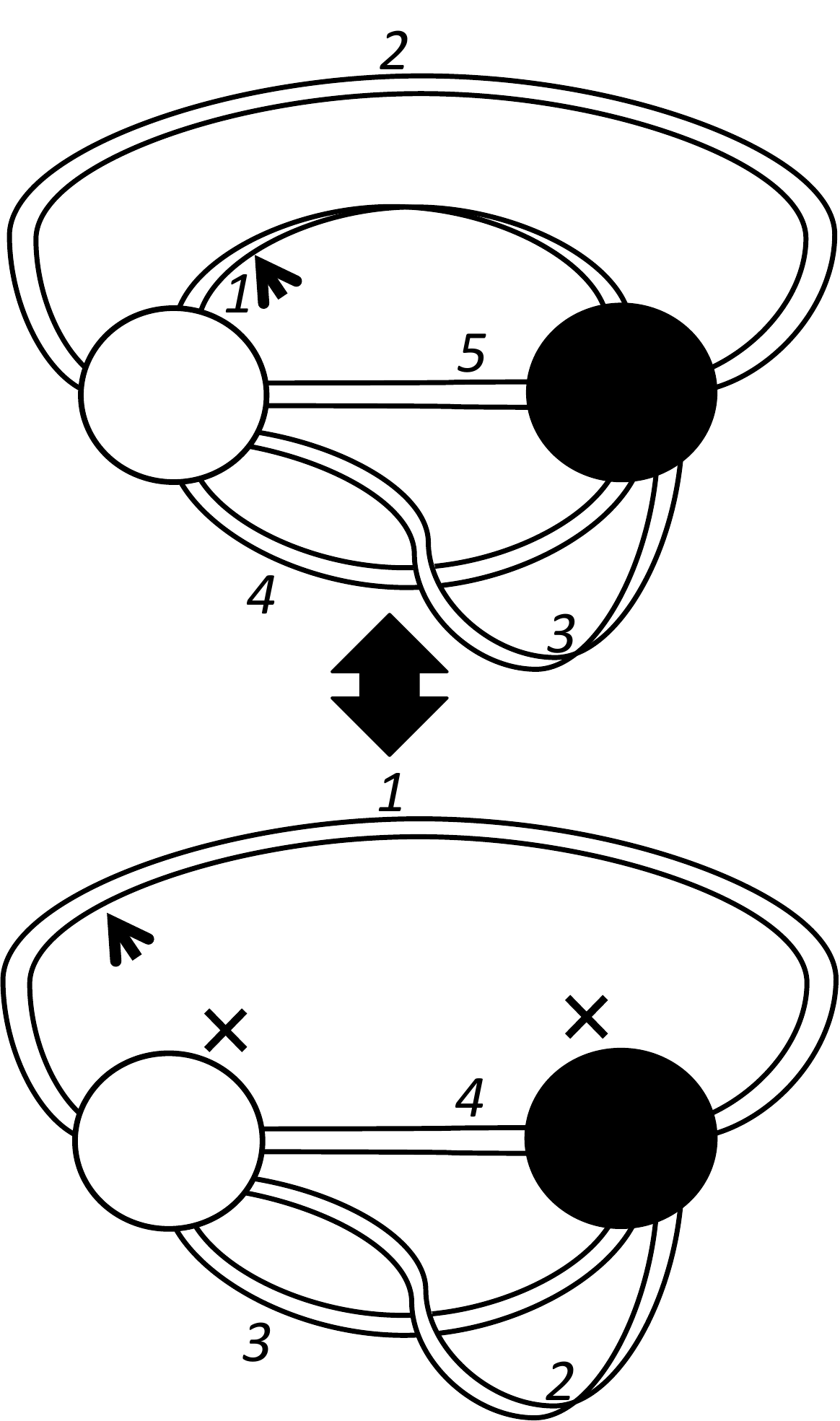}}\qquad
\subfigure[Second bijection.]{\includegraphics[scale = 0.28]{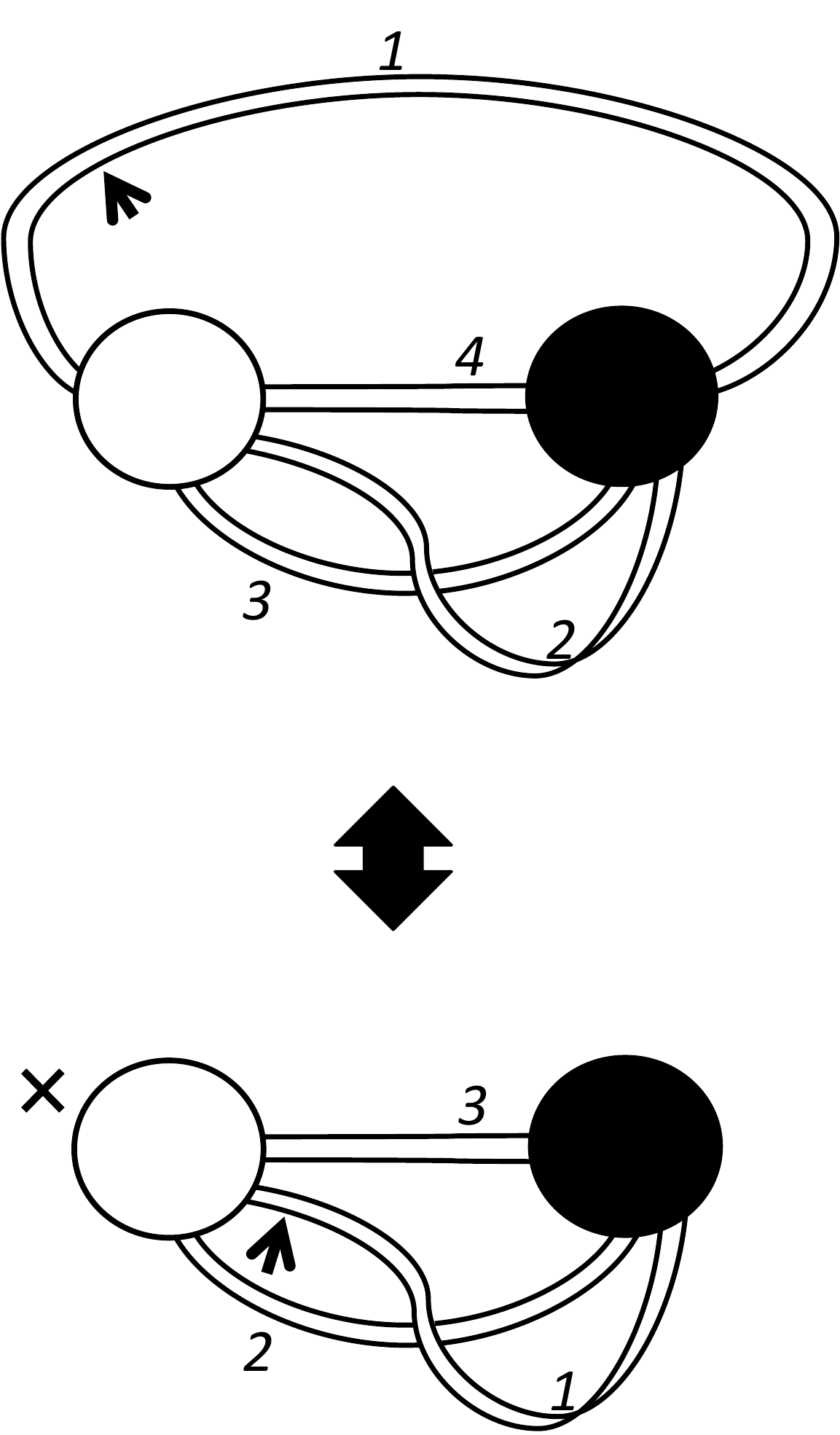}}\qquad
\subfigure[Third bijection.]{\includegraphics[scale = 0.28]{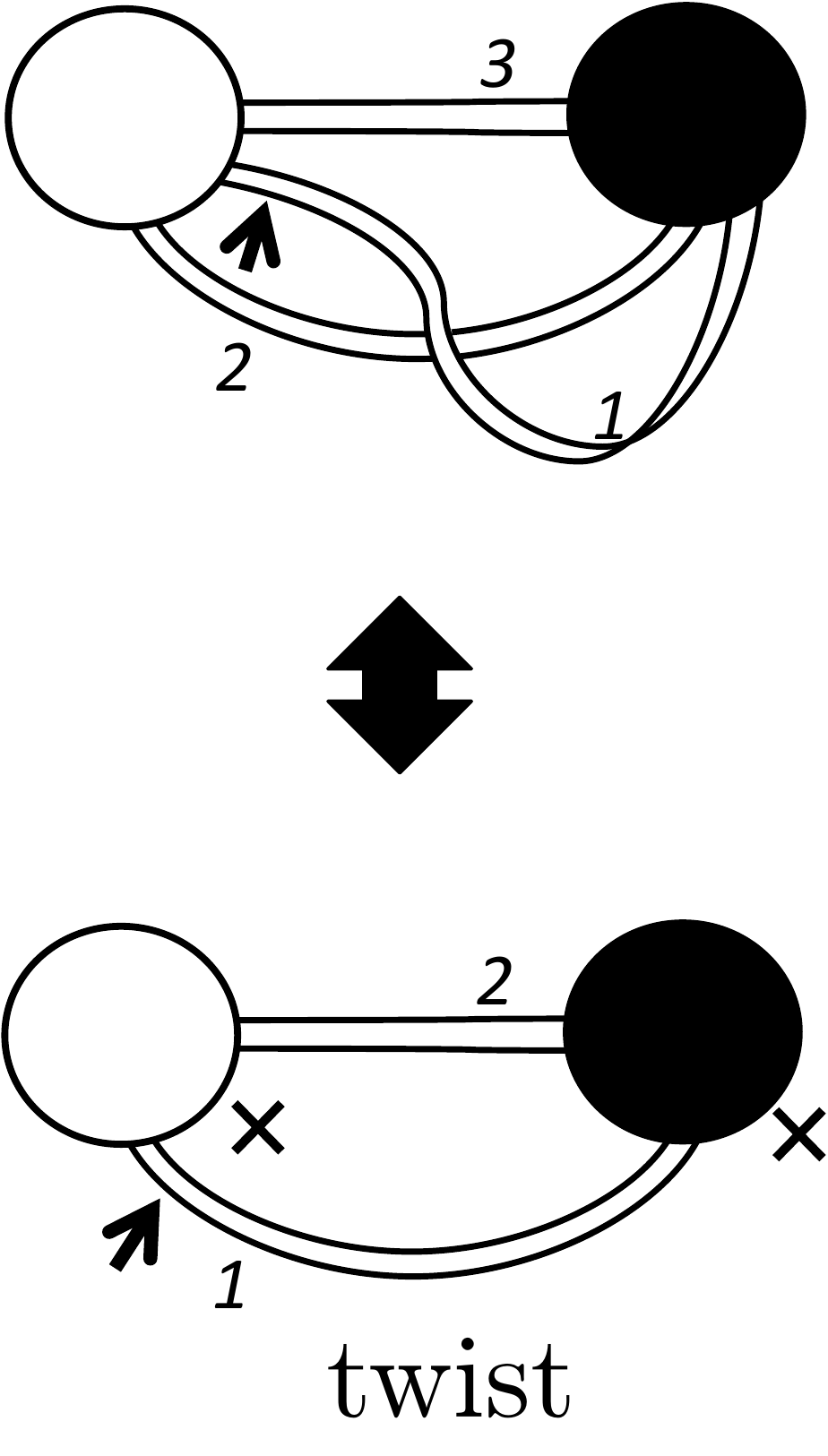}}\caption{Examples of application of the three bijections for labelled star hypermaps for $k=n$.}
\label{fig : bijkn}
\end{center}
\end{figure}
\end{exm}

As a consequence, one gets
\begin{align*}
\Sigma^\la_{n}(\beta) 
&=\sum_{i\geq 1}(i-1)^2m_{i-1}(\la_{\downarrow(i)})\hspace{-5mm}\sum_{M' \in \widetilde{\mathcal{L}}^{\la_{\downarrow(i)}}_{n-1}}\hspace{-5mm}\beta^{1+\vartheta(M')}\\
&+\sum_{i,j\geq 1}(i+j-1)m_{i+j-1}(\la_{\downarrow(i,j)})\hspace{-5mm}\sum_{M' \in \widetilde{\mathcal{L}}^{\la_{\downarrow(i,j)}}_{n-1}}\hspace{-5mm}\beta^{\vartheta(M')}\\
&+\sum_{i,d\geq 1}(i-d-1)dm_{i-d-1,d}(\la^{\uparrow(i-1-d,d)})\hspace{-8mm}\sum_{M' \in \widetilde{\mathcal{L}}^{\la^{\uparrow(i-1-d,d)}}_{n-1}}\left(\beta^{\vartheta(M')}+\beta^{1+\vartheta(M')}\right)
\end{align*}
As a conclusion, for any integer $n\geq1$, $\Sigma^\la_{n}(\beta) = \widetilde{h}_{n,n}^\la(\beta+1).$

\subsubsection{Case $k=1$}
For $k=1$, it is easy to show (\cite{GouJac96}) that $\widetilde{h}^\la_{n,[1^n]} = 0$ if $\la \neq (n)$. Using Equation (\ref{1one}) for $\la = (n)$ and $\rho  = \emptyset$, one gets :
\begin{align*}
a^n_{n,[1^n]} = a^{n-1}_{n-1,[1^{n-1}]} = a^1_{1,1} = 1
\end{align*}
Then according to Equation (\ref{eq : ha}),
\begin{align*}
h^n_{n,[1^n]} = a^n_{n,[1^n]} = 1
\end{align*}
Finally there are $(n-1)!$ labelled star hypermaps of face degree distribution $(n)$ and black vertex degree distribution $[1^n]$ that correspond to the $(n-1)!$ possible labelling of the only (unlabelled) star hypermap  with one white vertex and $n$ leaves. Furthermore for any such labelled star hypermap $M$, $\vartheta(M) = 0$ and $$\Sigma^n_{[1^n]}(\beta) = (n-1)! = \widetilde{h}^n_{n,[1^n]}(\beta +1).$$
\subsubsection{Case $k=2$}

As noticed in Remark \ref{rem : 2m}, hypermaps with black vertices degree distribution equal to $2^m$ for some integer $m$ reduce to non-bipartite maps by removing the black vertices. The resulting non-bipartite maps is a labelled {\bf monopole}, i.e. a labelled map composed of a single (white) vertex and $m$-edges with two consecutive labels and twice incident to the vertex. The type of the edges is the same as in hypermaps, i.e. an edge can be a cross-border, a border or a handle (but never a leaf). By abuse of notations, we also denote $\vartheta$ the induced measure of non-orientability for labelled monopoles and $\widetilde{\mathcal{L}}^\la_{[2^m]}$ the set of such maps with $m$ edges. Since the two edges incident to the same black vertex in the initial hypermap have consecutive labels and deleting one of the two edges makes the second one become a leaf, two iterations of the computation of $\vartheta$ in the initial hypermap is equivalent to one iteration in the resulting non-bipartite map.\\ 
According to Section \ref{sec : k=n}, in the case $m=1$, we have
\begin{align*}
\Sigma^{(1,1)}_{2}(\beta)  = \widetilde{h}_{2,2}^{(1,1)}(\be+1)~\mbox{ and }~\Sigma^{2}_{2}(\beta)  = \widetilde{h}_{2,2}^{2}(\be+1) 
\end{align*}
Suppose now $m>1$. As in section \ref{sec : bijk=n}, we split $\widetilde{\mathcal{L}}^\la_{[2^m]}$ into three subsets depending on the type of the root. 
\begin{itemize}
\item (Cross border) The set of labelled monopoles with face distribution $\la$ and a cross border root incident to a face of degree $i$ is in bijection with the set of decorated labelled monopoles with (i) face distribution $\la_{\da\da(i)}$, 
(ii) two differentiated marked positions incident to a face of degree $i-2$ ($m_{i-2}(\la_{\da\da(i)})(i-1)(i-2)$ possible choices). One marked position corresponds to the half-edge labelled $1$ in the preimage of the decorated monopole, the other to the half-edge labelled $2$.
\item (Border) The set of labelled monopoles with face distribution $\la$ and a  border root incident to both a face of size $i$ and a face of size $j$ is in bijection with the set of decorated labelled monopoles with (i) face distribution $\la_{\da\da(i,j)}$, (ii) one marked position incident to a face of degree $i+j-2$ (once this first position around the vertex is chosen there is only one more position around the vertex such that drawing a border by connecting these two positions cut the face of size $i+j-2$ into a face of degree $i$ and a face of degree $j$). We take the convention that the chosen marked position corresponds to the root half edge and the other position of the other half edge in the preimage of the decorated monopole. 
\item (Handle) The set of labelled monopoles with face distribution $\la$ and a handle root incident to a face of size $i$ such that removing the root yields a face of degree $d$ and one of degree $i-2-d$ is in bijection with the set of decorated labelled monopoles with (i) face distribution $\la^{\ua\ua(i-2-d,d)}$, (ii) one marked position around the vertex incident to a face of degree $i-2-d$, (iii) one marked position around the vertex incident to a face of degree $d$, (we assume without loss of generality that the root half edge of the deleted edge is incident to the face of degree $i-2-d$), (iv) and a type for the deleted root : twist or untwist.  
\end{itemize}

\begin{exm} Figure \ref{fig : bijkn} illustrates the three bijections described above.
\begin{figure}[htbp]
\begin{center}
\subfigure[First bijection.]{\includegraphics[scale = 0.28]{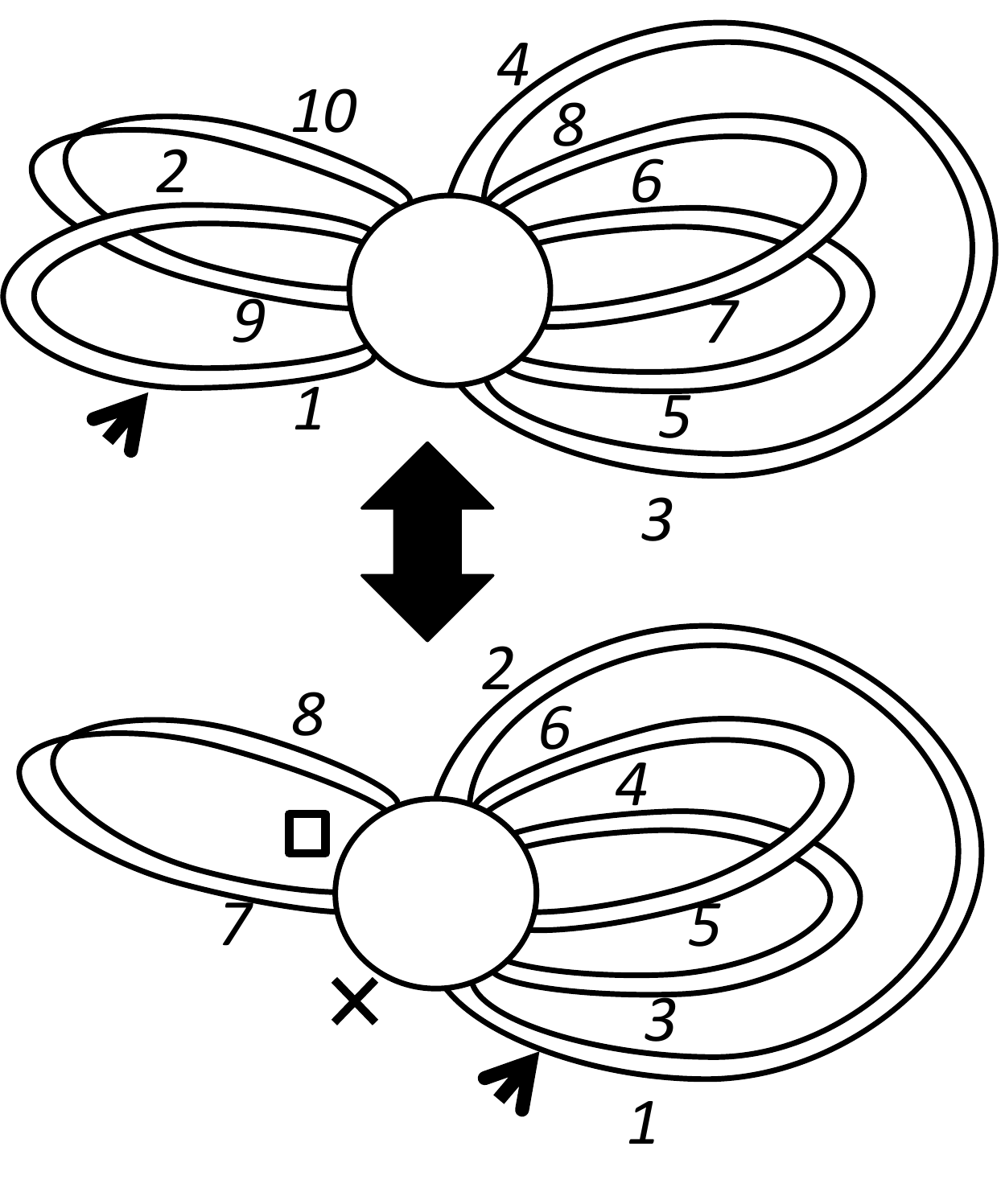}}\qquad
\subfigure[Second bijection.]{\includegraphics[scale = 0.28]{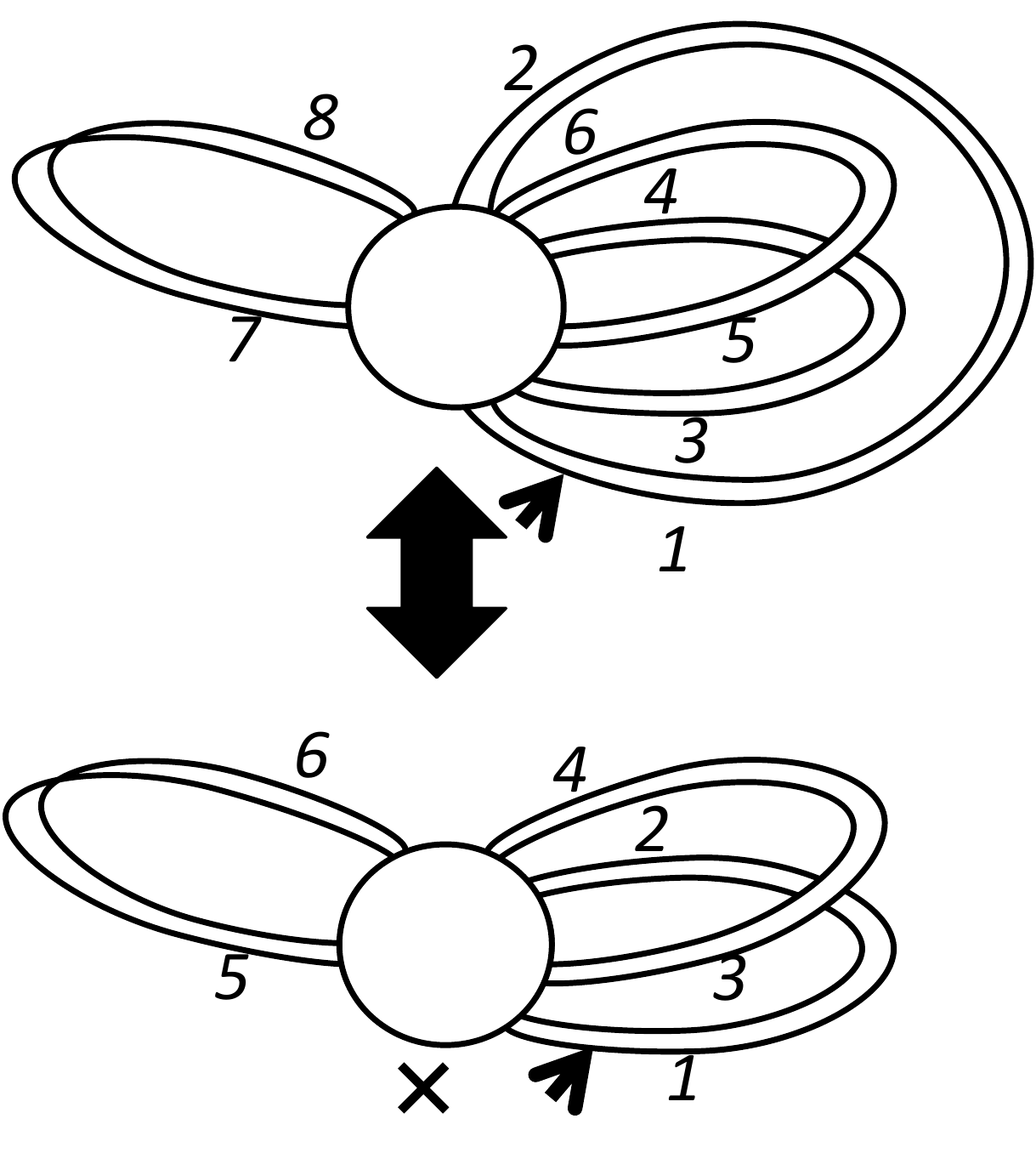}}\qquad
\subfigure[Third bijection.]{\includegraphics[scale = 0.28]{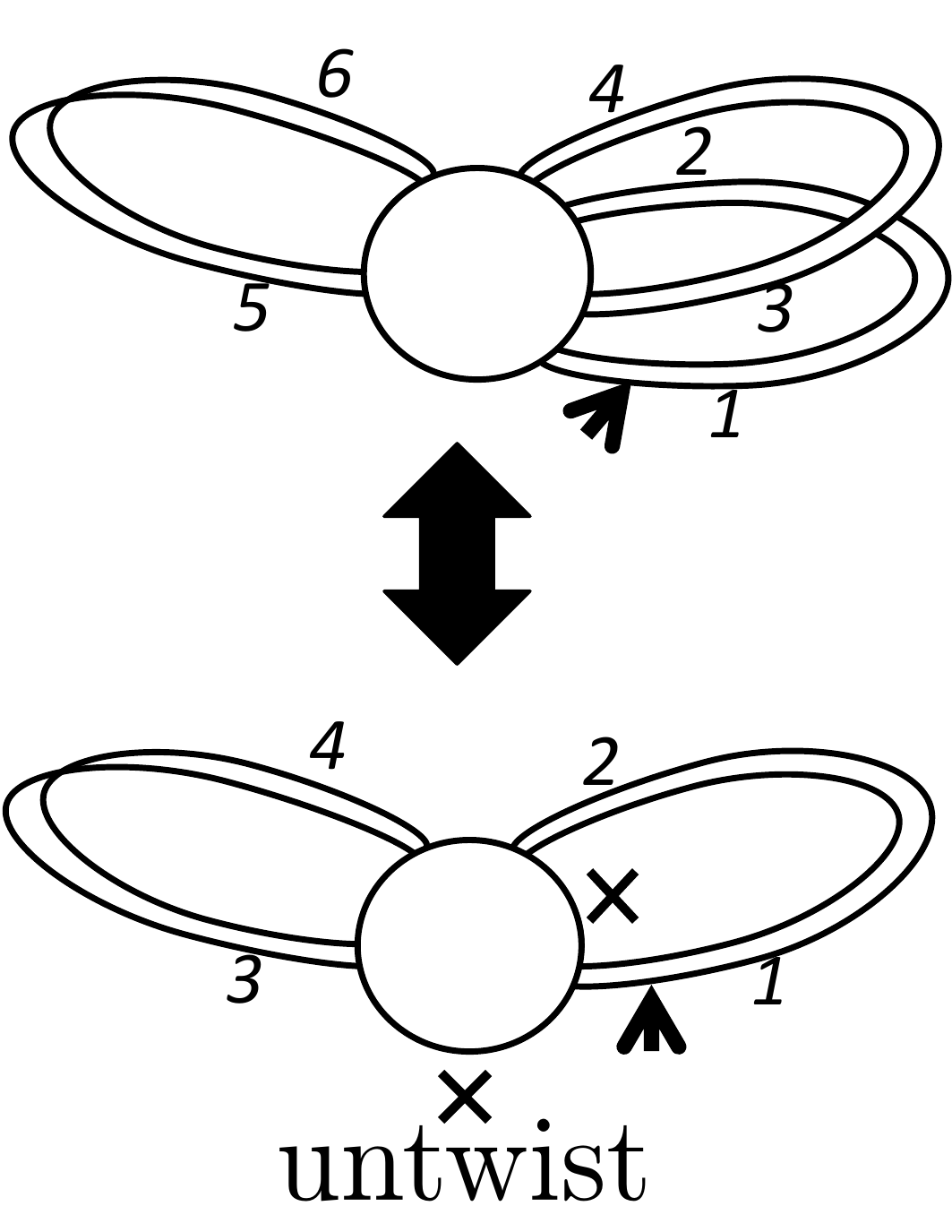}}\caption{Examples of application of the three bijections for labelled star hypermaps for $k=2$.}
\label{fig : bijkn}
\end{center}
\end{figure}
\end{exm}
As a consequence, one gets
\begin{align*}
\Sigma^\la_{[2^m]}(\beta) &=\sum_{i\geq 1}(i-1)(i-2)m_{i-2}(\la_{\da\da(i)})\sum_{M' \in \widetilde{\mathcal{L}}^{\la_{\da\da(i)}}_{[2^{m-1}]}}\beta^{1+\vartheta(M')}\\
&+\sum_{i,j\geq 1}(i+j-2)m_{i+j-2}(\la_{\da\da(i,j)})\sum_{M' \in \widetilde{\mathcal{L}}^{\la_{\da\da(i,j)}}_{[2^{m-1}]}}\beta^{\vartheta(M')}\\
&+\sum_{i,d\geq 1}(i-2-d)dm_{i-d-2,d}(\la^{\ua\ua(i-2-d,d)})\hspace{-8mm}\sum_{M' \in \widetilde{\mathcal{L}}^{\la^{\ua\ua(i-2-d,d)}}_{[2^{m-1}]}}\left(\beta^{\vartheta(M')}+\beta^{1+\vartheta(M')}\right)
\end{align*}
As a conclusion, for any integer $m\geq1$, $$\Sigma^\la_{[2^m]}(\beta) = \widetilde{h}_{2m,[2^m]}^\la(\beta+1).$$
\subsubsection{Case $k=3$}
As a final case, we show Theorem \ref{thm : H} in the case $k=3$ We focus on labelled star hypermaps that contain only black vertices of degree $3$. We call these maps {\bf labelled star triangulations}. 
First use Section \ref{sec : k=n}, to show that 
\begin{gather*}
\Sigma^{(1,1,1)}_{3}(\beta)  = \widetilde{h}_{3,3}^{(1,1,1)}(\be+1)~\mbox{ , }~\Sigma^{(2,1)}_{3}(\beta)  = \widetilde{h}_{3,3}^{(2,1)}(\be+1)\\
\mbox{ and }~\Sigma^{3}_{3}(\beta)  = \widetilde{h}_{3,3}^{3}(\be+1)
\end{gather*}
Then, we build a bijection between labelled star triangulations in $\widetilde{\mathcal{L}}^{\la}_{[3^m]}$ and some decorated labelled star triangulations with $m-1$ black vertices. To this end, we iterate three times the root deletion process. As the edges incident to the black vertex labelled $1$ (incident to the root) are labelled with $1,2$ and $3$, exactly these three edges are deleted after three iterations of the root deletion process and the resulting hypermap is a labelled star triangulation in $\widetilde{\mathcal{L}}^{\la'}_{[3^{m-1}]}$. \\
We look at all the possible configurations of the subset of the labelled triangulation composed of the white vertex, the black vertex incident to the root and the three edges incident to this black vertex. We call this subset the {\bf root submap}. Delete the root submap $\sigma_M$ of a star triangulation $M$ of face degree distribution $\la$ but keep a mark at the incidence positions around the white vertex of all of its three edges. Denote $M'$ the resulting star triangulation with three marks and $\la'$ its resulting face degree distribution. There are 6 possible cases that we split first according to the incidence of the marks to the various faces of $M'$.  
\begin{itemize}
\item[a] All the marks are incident to the same face of $M'$. In this case, it is easy to show that $\vartheta(M) = \vartheta(M') + \vartheta(\sigma_M)$. We have three sub-cases depending on the number of faces in $\sigma_M$.    
\begin{itemize}
\item[a.1] $\sigma_M$ has exactly one face of degree $3$. In this case, $\la'= \la_{\da\da\da(i)}$ for some $i$ and there is a bijection between the set of such labelled star triangulations of face degree distribution $\la$ and the set of couples composed of a labelled star map with one black vertex of face distribution $(3)$ and a labelled star triangulation of face degree distribution $\la_{\da\da\da(i)}$ with (i) one marked face of degree $i-3$, (ii) three marked positions among $i-1$ around the white vertex within the marked face and (iii) one distinguished position among the three marks. The distinguished position is used to locate the root of the star map of face degree distribution $(3)$ within the star triangulation. 

\item[a.2] $\sigma_M$ has exactly two faces of degree $(2,1)$. In this case, $\la'= \la_{\da\da\da(i,j)}$ for some $i\leq j$ and there is a bijection between the set of such labelled star triangulations of face degree distribution $\la$ and the set of couples composed of a labelled star map with one black vertex of face distribution $(2,1)$ and a labelled star triangulation of face degree distribution $\la_{\da\da\da(i,j)}$ with (i) one marked face of degree $i+j-3$, (ii) two differentiated marked positions among $i+j-2$ around the white vertex within the marked face. 

\item[a.3] $\sigma_M$ has three faces of degree $(1,1,1)$. In this case, $\la'= \la_{\da\da\da(i,j,k)}$ for some $i,j,k$. We assume that $i$ is the degree of the face incident to the root of $\sigma_M$ and the next edge of $\sigma_M$ moving counterclockwise around the white vertex and that $k$ is the degree of the face incident to the root of $\sigma_M$ and the next edge of $\sigma_M$ moving clockwise around the white vertex.There is a bijection between the set of such labelled star triangulations of face degree distribution $\la$ and the set of couples composed of a labelled star maps with one black vertex of face distribution $(1,1,1)$ and a labelled star triangulation of face degree distribution $\la_{\da\da\da(i,j,k)}$ with (i) one marked face of degree $i+j+k-3$, (ii) one marked positions among $i+j+k-3$ around the white vertex within the marked face. 
The marked position locates the root of the star map of face degree distribution $(1,1,1)$. Then it is easy to show that $i$ and $k$ give exactly the required information to locate the two other edges.
\end{itemize}

\item[b] Two of the marks are incident to one face of $M'$ and the third one to a second distinct face of $M'$. In this case, there is a handle edge in $\sigma_M$ whose contribution to $\vartheta(M)$ is not its contribution to $\vartheta(\sigma_M)$ (this edge may not be a handle at all in the submap $\sigma_M$). Denote $\sigma_M'$ the reduction of $\sigma_M$ obtained by deletion of the handle of smallest index. We have  $\vartheta(M) = \vartheta(M') + \vartheta(\sigma_M')+\eps$ where $\eps \in \{0,1\}$ depending on the contribution of the handle (recall that by twisting the handle one get a distinct map with the other value for $\eps$). We have two sub-cases depending on the number of faces in $\sigma_M'$.   

\begin{itemize}
\item[b.1] $\sigma_M'$ has exactly one face of degree $2$. In this case, $\la'= {\la^{\ua\ua\ua(i-3-d,d)}}$ for some $i$ and $d$ with $d \leq i-3-d$. There is a bijection between the set of such labelled star triangulations of face degree distribution $\la$ and the set of couples composed of a labelled star map with one black vertex of face distribution $(2)$ and a labelled star triangulation of face degree distribution ${\la^{\ua\ua\ua(i-3-d,d)}}$ with (i) one marked face of degree $i-3-d$, (ii) one marked face of degree $d$, (iii) one marked position around the white vertex within the marked face of degree $i-3-d$, (iii) one marked position around the white vertex within the marked face of degree $d$, (iv) a distinguished marked position around the white vertex within the face of degree $d$ or the face of degree $i-3-d$ (there are $i-1$ possibilities after adding the two first marks), (v) one position around the white vertex within $\sigma_M'$ (2 possibilities), (vi) an index for the handle edge ($3$ possible values) and (vii) an indication whether the handle is twisted or not. The two marks belonging to the same face of $M'$ locate the edges of $\sigma_M'$, its root being incident to the distinguished one. The other mark locates the incidence of the handle to the white vertex. The mark within $\sigma_M'$ locates the incidence of the handle to the black vertex labelled $1$. 
\item[b.2] $\sigma_M'$ has two faces of degree $(1,1)$. In this case, $\la'= \la^{\da(i,j)\ua\ua(i+j-3-d,d)}$ for some $i,j$ and $d$. There is a bijection between the set of such labelled star triangulations of face degree distribution $\la$ and the set of couples composed of a labelled star map with one black vertex of face distribution $(1,1)$ and a labelled star triangulation of face degree distribution ${\la^{\da(i,j)\ua\ua(i+j-3-d,d)}}$ with (i) one marked face of degree $i-3-d$, (ii) one marked face of degree $d$, (iii) one marked position around the white vertex within the marked face of degree $i-3-d$, (iii) one marked position around the white vertex within the marked face of degree $d$, (iv) one position around the white vertex within $\sigma_M'$ (2 possibilities), (vi) an index for the handle edge ($3$ possible values) and (vii) an indication whether the handle is twisted or not.
\end{itemize}
\item[c] The three marks are incident to three disctinct faces of $M'$. All the edges of $\sigma_M$ are handles and we have  $\vartheta(M) = \vartheta(M') + \eps_1 + \eps_2$ where $\eps_i \in \{0,1\}$ depending on the contribution of the two handles with the smallest index.\\
In this case, $\la'= \la_{\ua\ua\ua(i-3-d-f,d,f)}$ for some $i,d,f$. We assume that (i) $i-3-d-f$ is the degree of the face of $M'$ incident to the mark corresponding to the root of $\sigma_M$ (ii) $d$ is the degree of the face of $M'$ incident to the mark corresponding to the edge labelled $2$ in $\sigma_M$ and (iii) $f$ is the degree of the face of $M'$ incident to the mark corresponding to the edge labelled $3$ in $\sigma_M$.There is a bijection between the set of such labelled star triangulations of face degree distribution $\la$ and the set of labelled star triangulations of face degree distribution $\la_{\ua\ua\ua(i-3-d-f,d,f)}$ with (i) one marked face of degree $i-3-d-f$, (ii) one marked face of degree $d$, (iii) one marked face of degree $f$, (iv) one marked positions around the white vertex within the marked face of degree $i-3-d-f$, (v) one marked positions around the white vertex within the marked face of degree $d$, (vi) one marked positions around the white vertex within the marked face of degree $f$, (vii) an indication whether the two handles labelled $1$ and $2$ are twisted or not, (viii) an indication wether the order of the edges in $\sigma_M$ moving clockwise around the black vertex is $132$ or $123$. 
\end{itemize}
\begin{exm} Figure \ref{fig : Bk3} illustrates the six bijections described above.
\begin{figure}[htbp]
\begin{center}
\subfigure[Case a.1. \label{fig : Bk3_1}]{\includegraphics[scale = 0.21]{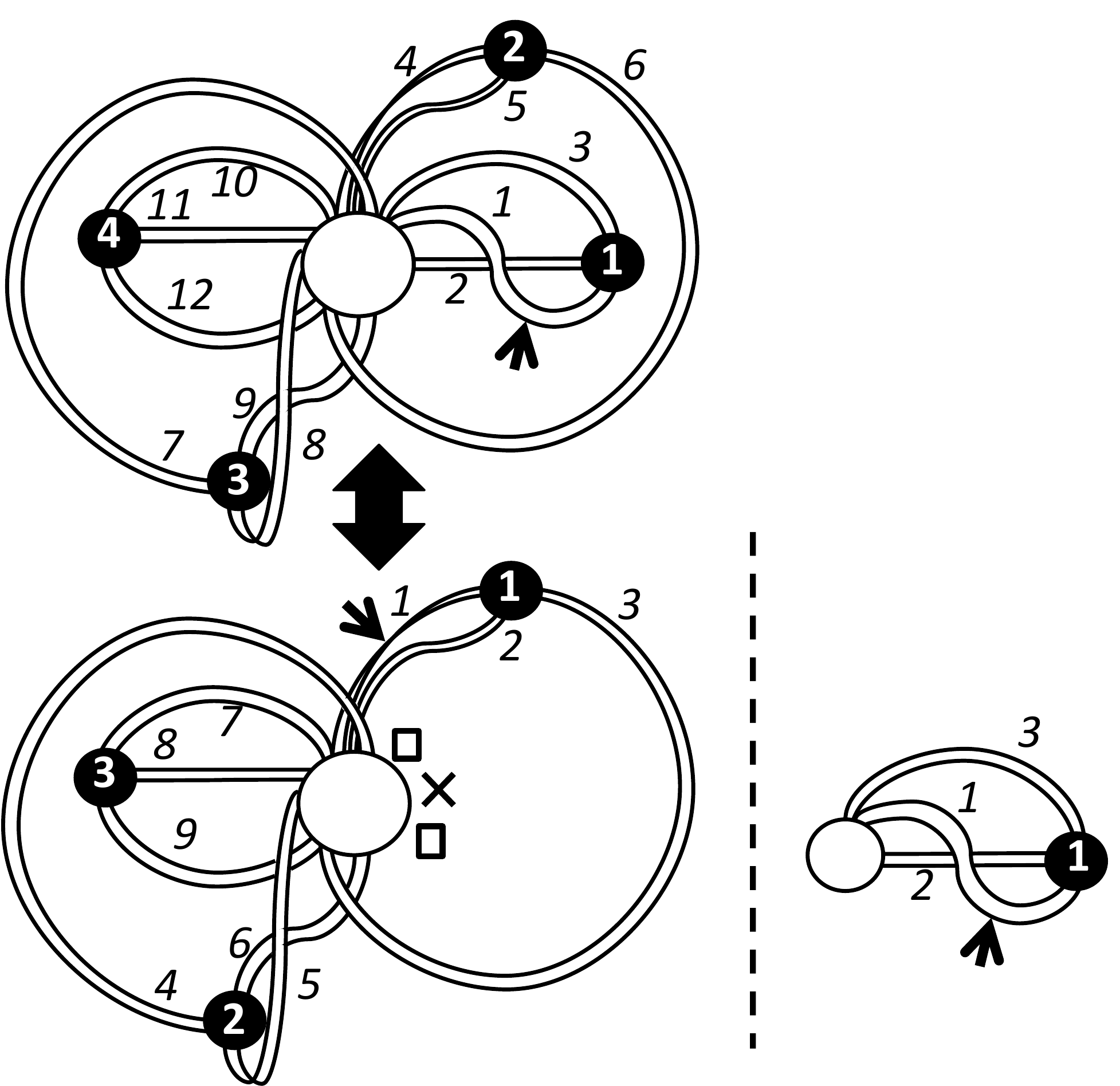}}\quad\quad
\subfigure[Case a.2. \label{fig : Bk3_2}]{\includegraphics[scale = 0.21]{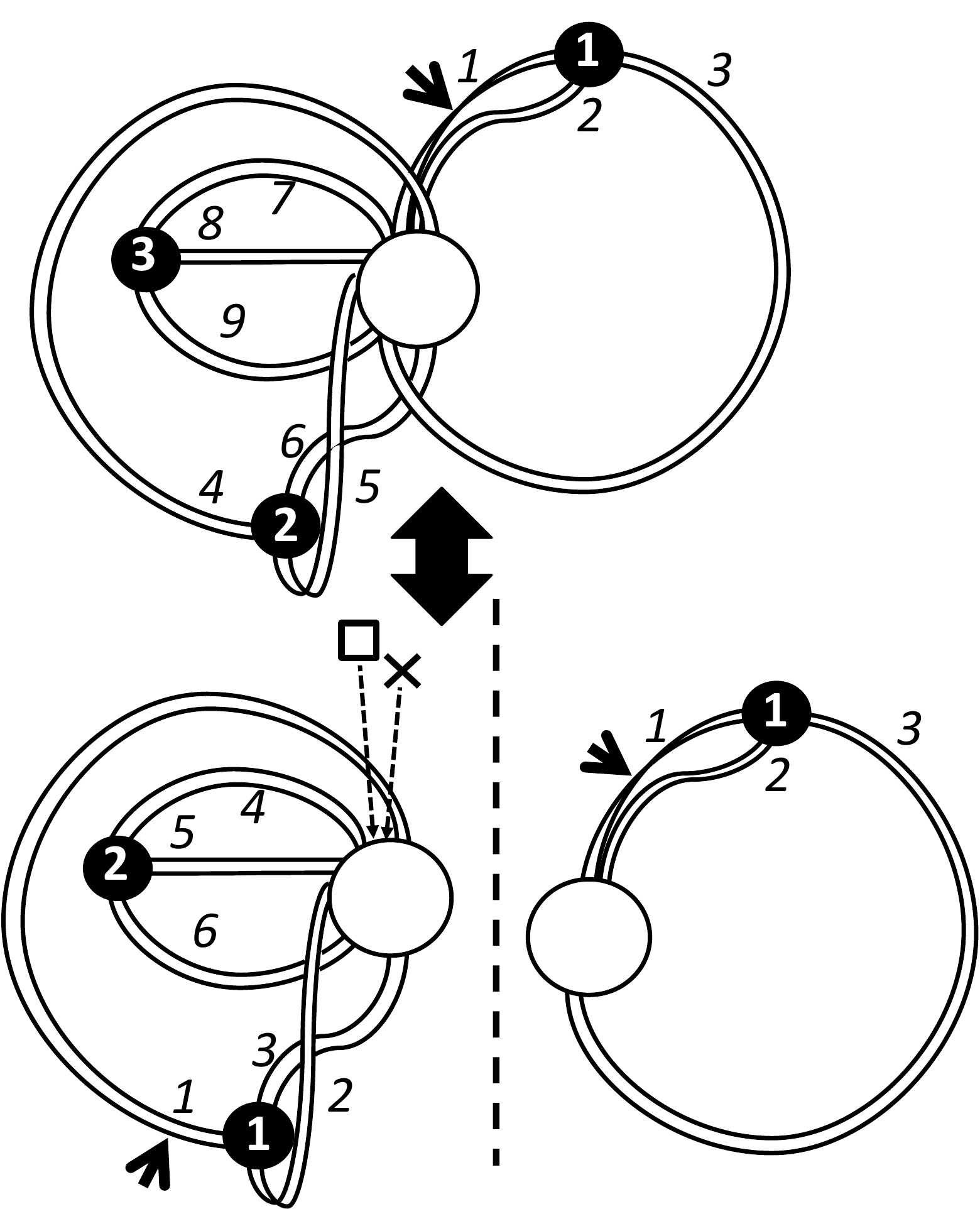}}\quad\quad
\subfigure[Case a.3.\label{fig : Bk3_6}]{\includegraphics[scale = 0.11]{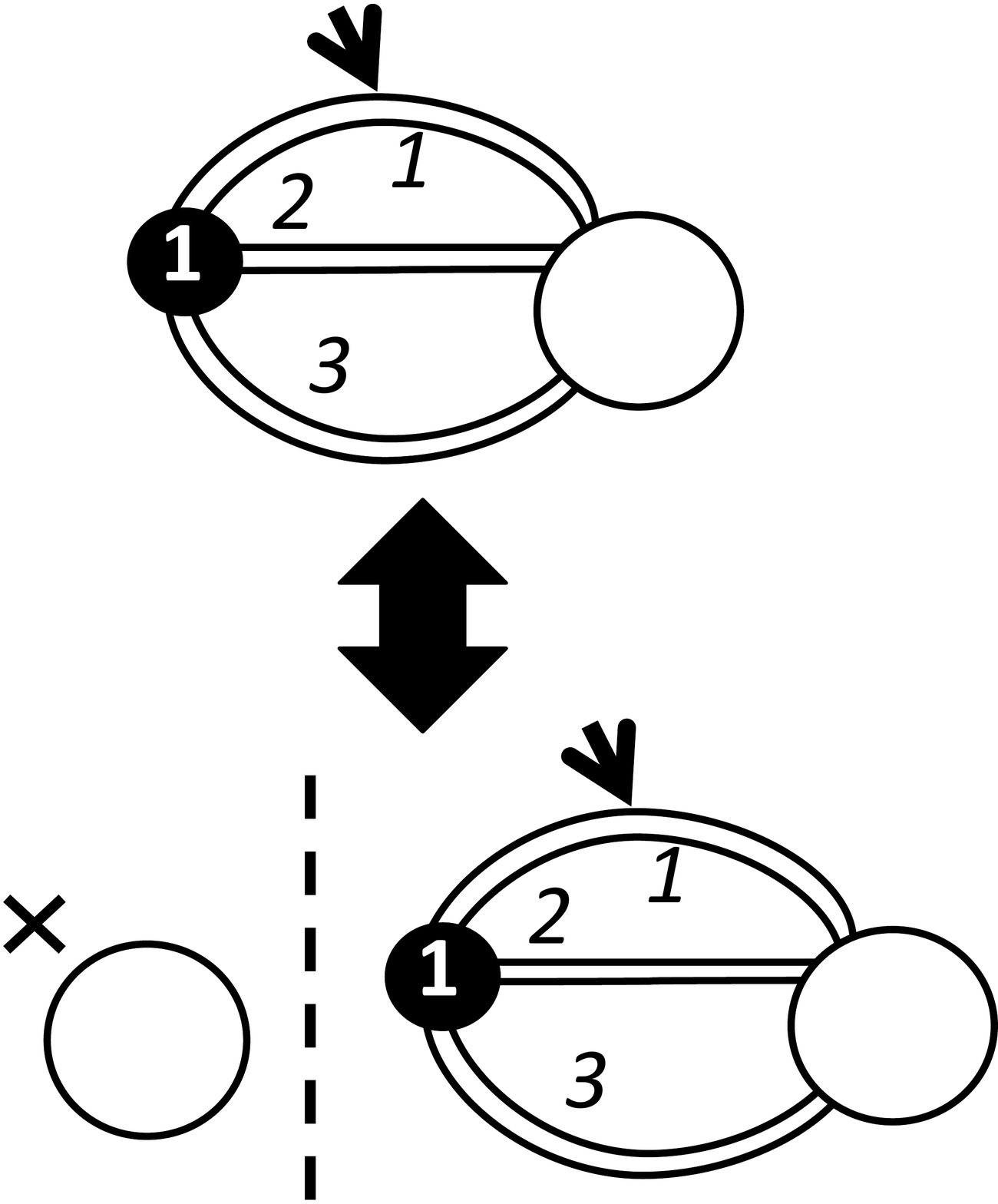}} \\
\subfigure[Case b.1.\label{fig : Bk3_3}]{\includegraphics[scale = 0.21]{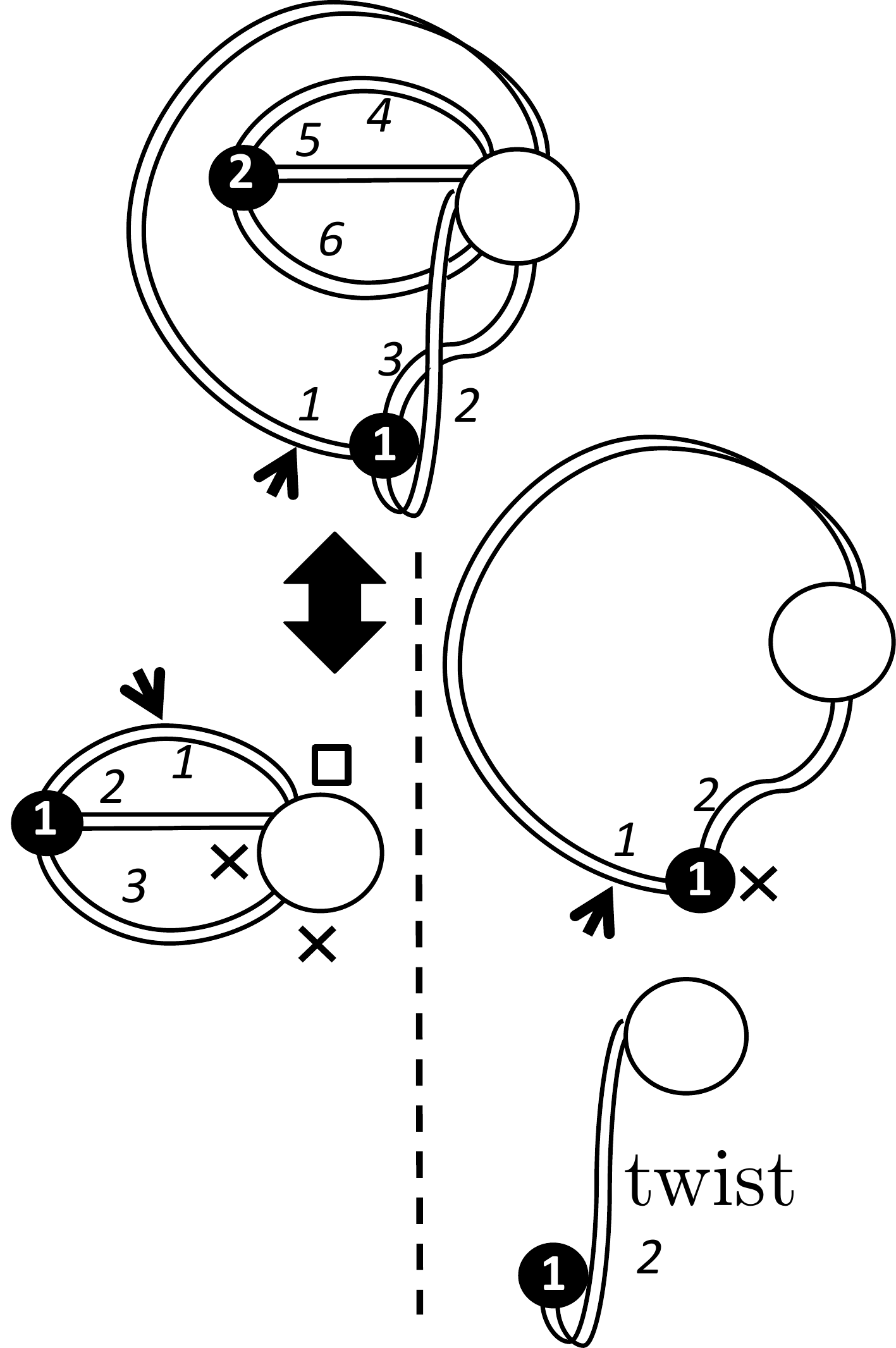}}\quad\quad\quad
\subfigure[Case b.2.\label{fig : Bk3_4}]{\includegraphics[scale = 0.21]{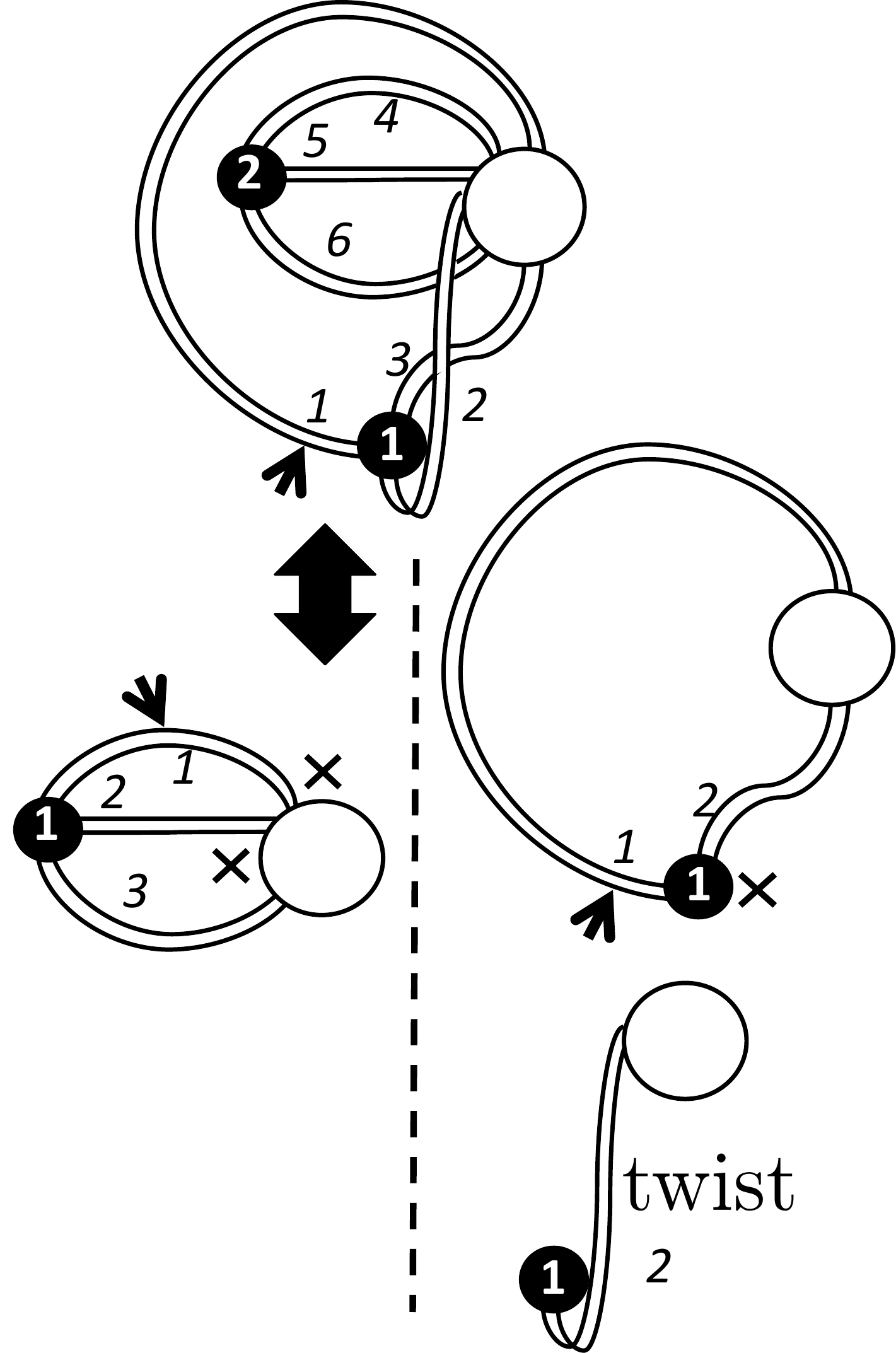}}\quad\quad\quad
\subfigure[Case c.\label{fig : Bk3_5}]{\includegraphics[scale = 0.21]{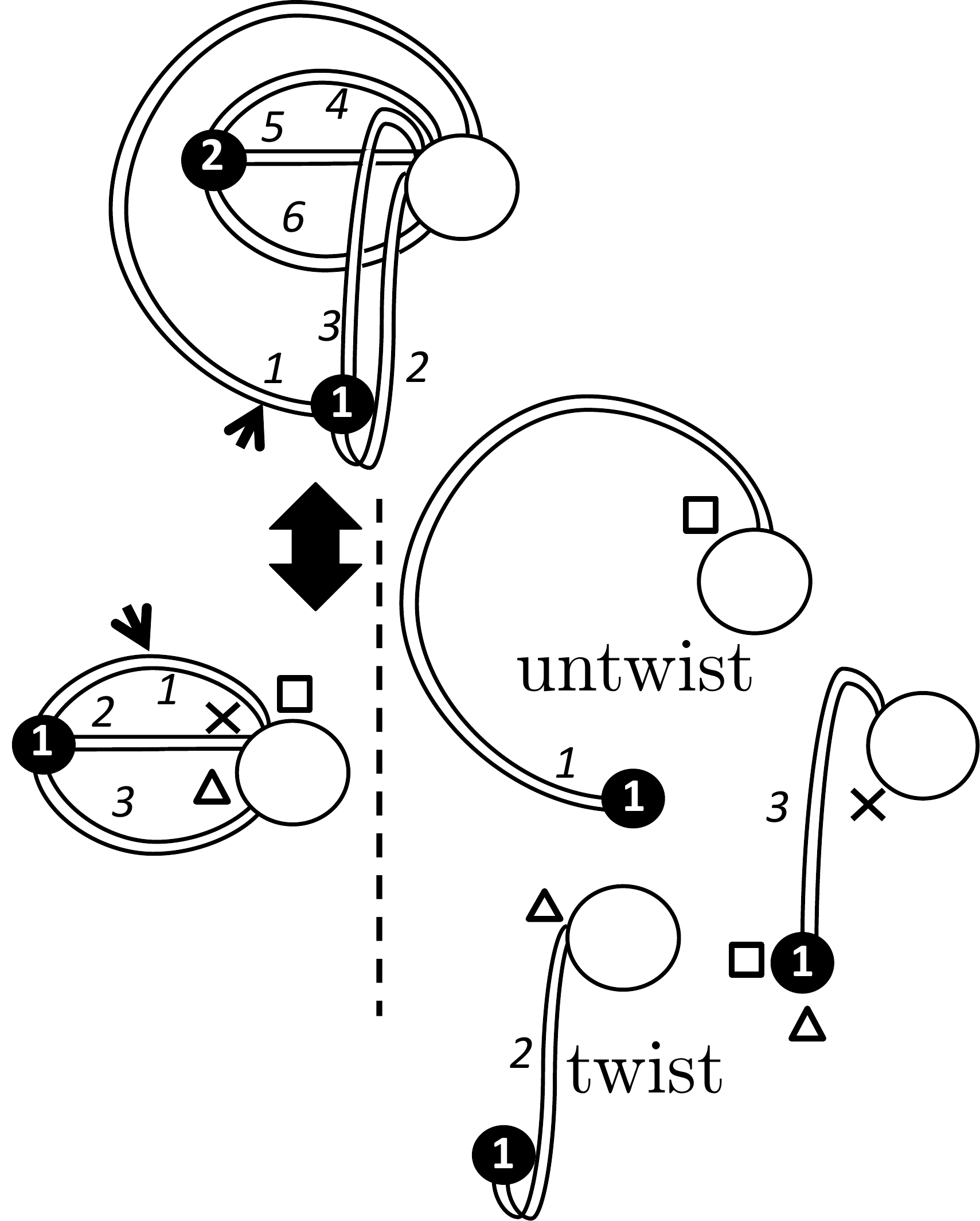}}
\caption{Illustration of the six bijections for labelled star triangulations.}
\label{fig : Bk3}
\end{center}
\end{figure}
\end{exm}
As a consequence, one gets
\begin{align*}
&\Sigma^\la_{[3^m]}(\beta) =
3\sum_{i}\binom{i-1}{3}m_{i-3}(\la_{\da\da\da (i)})
\hspace{-5mm}\sum_{\substack{\sigma_M \in \wit{\mathcal{L}}^3_3\\M'\in\wit{\mathcal{L}}^{\la_{\da\da\da (i)}}_{[3^{m-1}]}}}\hspace{-4mm}\beta^{\vartheta(\sigma_M)+\vartheta(M')}\\
\nonumber&+\frac{1}{2}\sum_{i,j}(i+j-2)(i+j-3)m_{i+j-3}({\la_{\da\da\da( i, j)}})\hspace{-5mm}\sum_{\substack{\sigma_M \in \wit{\mathcal{L}}^{(2,1)}_3\\M'\in\wit{\mathcal{L}}^{\la_{\da\da\da( i, j)}}_{[3^{m-1}]}}}\hspace{-4mm}\beta^{\vartheta(\sigma_M)+\vartheta(M')}\\
&+\sum_{i,j,k}(i+j+k-3)m_{i+j+k-3}(\la_{\da\da\da(i,j,k)})
\hspace{-5mm}\sum_{\substack{\sigma_M \in \wit{\mathcal{L}}^{(1,1,1)}_3\\M'\in\wit{\mathcal{L}}^{\la_{\da\da\da(i,j,k)}}_{[3^{m-1}]}}}\hspace{-4mm}\beta^{\vartheta(\sigma_M)+\vartheta(M')}\\
&+6\frac{1}{2}\sum_{i,d}(i-1)(i-d-3)dm_{i-3-d,d}({\la^{\ua\ua\ua(i-3-d,d)}})
\hspace{-8mm}\sum_{\substack{\sigma_M' \in \wit{\mathcal{L}}^{(2)}_2\\M'\in\wit{\mathcal{L}}^{\la^{\ua\ua\ua(i-3-d,d)}}_{[3^{m-1}]}}}\hspace{-8mm}(1+\beta)\beta^{\vartheta(\sigma_M')+\vartheta(M')}\\
&+6\frac{1}{2}\sum_{i,j,d}(i+j-3-d)dm_{i+j-d-3,d}(\la^{\da(i,j)\ua\ua(i+j-3-d,d)})
\hspace{-13mm}\sum_{\substack{\sigma_M' \in \wit{\mathcal{L}}^{(1,1)}_2\\M'\in\wit{\mathcal{L}}^{\la^{\da(i,j)\ua\ua(i+j-3-d,d)}}_{[3^{m-1}]}}}\hspace{-12mm}(1+\beta)\beta^{\vartheta(\sigma_M')+\vartheta(M')}\\
&+2\sum_{i,d,f}(i-3-d-f)dfm_{i-3-d-f,d,f}(\la^{\ua\ua\ua(i-3-d-f,d,f)})
\hspace{-10mm}\sum_{\substack{\sigma_M' \in \wit{\mathcal{L}}^{(1,1)}_2\\M'\in\wit{\mathcal{L}}^{\la^{\ua\ua\ua(i-3-d-f,d,f)}}_{[3^{m-1}]}}}\hspace{-8mm}(1+\beta)^2\beta^{\vartheta(M')}.
\end{align*}
As a conclusion, for any integer $m\geq1$, $$\Sigma^\la_{[3^m]}(\beta) = \widetilde{h}_{3m,[3^m]}^\la(\beta+1).$$

\section{Appendix: computation of operator $\Omega_k$} \label{sec : appendix}

We provide the computation of operator $\Omega_2$. The computation of $\Omega_3$ is performed in a similar way but is much more cumbersome and is not detailed here. 
Denote $\zeta_k=kp_{k+1}\frac{\partial}{\partial p_k}$ and write  $$\Omega_2=[\Delta,\Omega_1]=(\al-1)\sum_{i,k}B_{i,k}+\al\sum_{i,j,k}A_{i,j,k}+\sum_{i,j,k}\Gamma_{i,j,k}$$ where 
\begin{align*}
&B_{i,k\phantom{,j}}=\left[ (i-1)^2p_i\frac{\partial}{\partial p_{i-1}},\zeta_k \right],\\
&A_{i,j,k} = \left[(i+j-1)p_ip_j\frac{\partial}{\partial p_{i+j-1}},\zeta_k \right],\\
&\Gamma_{i,j,k}=\left[ijp_{i+j+1}\frac{\partial}{\partial p_{i}}\frac{\partial}{\partial p_{j}},\zeta_k \right].
\end{align*}
The non-zero terms are:
{\small 
\begin{align*}
B_{i,i}=\left[ (i-1)^2p_i\frac{\partial}{\partial p_{i-1}} ,ip_{i+1}\frac{\partial}{\partial p_{i}} \right]&=(i-1)^2p_i\frac{\partial}{\partial p_{i-1}}ip_{i+1}\frac{\partial}{\partial p_{i}}\\ 
&-ip_{i+1}\frac{\partial}{\partial p_{i}} \left( (i-1)^2 p_i \frac{\partial}{\partial p_{i-1}} \right)\\
&-i (i-1)^2p_ip_{i+1}\frac{\partial}{\partial p_{i}} \frac{\partial}{\partial p_{i-1}}=i(i-1)^2p_{i+1} \frac{\partial}{\partial p_{i-1}}.
\end{align*}
\begin{align*}
B_{i,i-2}=\left[ (i-1)^2p_i\frac{\partial}{\partial p_{i-1}} ,(i-2)p_{i-1}\frac{\partial}{\partial p_{i-2}} \right]
&=(i-1)^2p_i\frac{\partial}{\partial p_{i-1}}\left( (i-2)p_{i-1}\frac{\partial}{\partial p_{i-2}} \right)\\
&+ (i-2)(i-1)^2p_ip_{i-1}\frac{\partial}{\partial p_{i-2}}\frac{\partial}{\partial p_{i-1}}\\
&-(i-1)^2(i-2)p_ip_{i-1}\frac{\partial}{\partial p_{i-1}}\frac{\partial}{\partial p_{i-2}}\\
&=(i-1)^2(i-2)p_i \frac{\partial}{\partial p_{i-2}}.
\end{align*}
\begin{align*}
A_{i,j,i-1}=\left[  ijp_{i+j+1}\frac{\partial}{\partial p_{i}}\frac{\partial}{\partial p_{j}}, (i-1)p_{i}\frac{\partial}{\partial p_{i-1}} \right]&= ij(i-1)p_{i+j+1}\frac{\partial}{\partial p_{i}}\frac{\partial}{\partial p_{j}} \left( p_{i}\frac{\partial}{\partial p_{i-1}} \right)\\
&+ ij(i-1)p_{i+j+1}p_{i}\frac{\partial}{\partial p_{i}}\frac{\partial}{\partial p_{j}}\frac{\partial}{\partial p_{i-1}}\\
&- ij(i-1)p_{i+j+1}p_{i}\frac{\partial}{\partial p_{i}}\frac{\partial}{\partial p_{j}}\frac{\partial}{\partial p_{i-1}}\\
&=ij(i-1)p_{i+j+1}\frac{\partial}{\partial p_{j}}\frac{\partial}{\partial p_{i-1}}.
\end{align*}
Similarly $A_{i,j,j-1}=ij(i-1)p_{i+j+1}\frac{\partial}{\partial p_{i}}\frac{\partial}{\partial p_{j-1}}$.
\begin{flalign*}
A_{i,j,i+j+1}&=\left[  ijp_{i+j+1}\frac{\partial}{\partial p_{i}}\frac{\partial}{\partial p_{j}}, (i+j+1)p_{i+j+2}\frac{\partial}{\partial p_(i+j+1)} \right]&\\
&=ij(i+j+1)p_{i+j+1}\frac{\partial}{\partial p_{i}}\frac{\partial}{\partial p_{j}}  p_{j+i+2}\frac{\partial}{\partial p_{j+i+1}}\\
&-ij(i+j+1)p_{j+i+2}\frac{\partial}{\partial p_{j+i+1}} \left( p_{i+j+1}\frac{\partial}{\partial p_{i}}\frac{\partial}{\partial p_{j}} \right)\\ 
&-ij(i+j+1)p_{i+j+1} p_{j+i+2}\frac{\partial}{\partial p_{j+i+1}}\frac{\partial}{\partial p_{i}}\frac{\partial}{\partial p_{j}}\\
&=-ij(i+j+1)p_{j+i+2}\frac{\partial}{\partial p_{i}}\frac{\partial}{\partial p_{j}}.
\end{flalign*}
\begin{flalign*}
\Gamma_{i,j,i}=\left[(i+j-1)p_ip_j\frac{\partial}{\partial p_{i+j-1}}, ip_{i+1}\frac{\partial}{\partial p_{i}} \right]&=(i+j-1)p_ip_j\frac{\partial}{\partial p_{i+j-1}}ip_{i+1}\frac{\partial}{\partial p_{i}}&\\
&- ip_{i+1}\frac{\partial}{\partial p_{i}} \left( (i+j-1)p_ip_j\frac{\partial}{\partial p_{i+j-1}} \right)\\
&-(i+j-1)p_ip_j\frac{\partial}{\partial p_{i+j-1}}ip_{i+1}\frac{\partial}{\partial p_{i}}\\
&=-(i+j-1)ip_{i+1} p_j\frac{\partial}{\partial p_{i+j-1}}.
\end{flalign*}
Similarly $\Gamma_{i,j,j}=- (i+j-1)ip_{j+1} p_i\frac{\partial}{\partial p_{i+j-1}}$.
\begin{flalign*}
\Gamma_{i,j,i+j-2}&=\left[(i+j-1)p_ip_j\frac{\partial}{\partial p_{i+j-1}}, (i+j-2)p_{i+j-1}\frac{\partial}{\partial p_{i+j-2}} \right]&\\
&=(i+j-1)p_ip_j\frac{\partial}{\partial p_{i+j-1}} \left( (i+j-2)p_{i+j-1}\frac{\partial}{\partial p_{i+j-2}} \right)\\
&+(i+j-1)(i+j-2)p_{i+j-1}p_ip_j\frac{\partial}{\partial p_{i+j-2}}\frac{\partial}{\partial p_{i+j-1}}\\
&-(i+j-1)(i+j-2)p_{i+j-1}p_ip_j\frac{\partial}{\partial p_{i+j-2}}\frac{\partial}{\partial p_{i+j-1}}\\
&=(i+j-1)(i+j-2)p_ip_j\frac{\partial}{\partial p_{i+j-2}}.
\end{flalign*}
\\
\\
{Summing up all the non-zeros $A$ terms yields:}

\begin{flalign*}
\sum_{i,j}A_{i,j,i-1}+\sum_{i,j}A_{i,j,j-1}+\sum_{i,j}A_{i,j,i+j+1}
&=  \sum_{i,j\geq1}ij(j-1)p_{i+j+1}\frac{\partial}{\partial p_{i}}\frac{\partial}{\partial p_{j-1}}&\\
&+ \sum_{i,j\geq1}ij(i-1)p_{i+j+1}\frac{\partial}{\partial p_{j}} \frac{\partial}{\partial p_{i-1}}\\
& - \sum_{i,j\geq1}ij(i+j+1)p_{j+i+2}\frac{\partial}{\partial p_{i}}\frac{\partial}{\partial p_{j}}\\
&= \sum_{i,j\geq1}ij(i+1)p_{i+j+2}\frac{\partial}{\partial p_{i}}\frac{\partial}{\partial p_{j}}\\
&+ \sum_{i,j\geq1}ij(j+1)p_{i+j+2}\frac{\partial}{\partial p_{i}}\frac{\partial}{\partial p_{j}}\\
&- \sum_{i,j\geq1}ij(i+j+1)p_{j+i+2}\frac{\partial}{\partial p_{i}}\frac{\partial}{\partial p_{j}}\\
&=\boxed{ \sum_{i,j\geq1}ijp_{j+i+2}\frac{\partial}{\partial p_{i}}\frac{\partial}{\partial p_{j}}}.
\end{flalign*}
\\
\\
{Summing up all the non-zeros $B$ terms yields:}

\begin{flalign*}
\sum_{i}B_{i,i}+\sum_{i}B_{i,i-2}
&=\sum_{i\geq1}(i-2)(i-1)^2p_i\frac{\partial}{\partial p_{i-2}}-\sum_{i\geq1}i(i-1)^2 p_{i+1} \frac{\partial}{\partial p_{i-1}}\\
&=\sum_{i\geq1}(i-1)i^2p_{i+1}\frac{\partial}{\partial p_{i-1}}-\sum_{i\geq1}i(i-1)^2 p_{i+1} \frac{\partial}{\partial p_{i-1}}\\
&=\boxed{\sum_{i\geq1}(i-1)ip_{i+1}\frac{\partial}{\partial p_{i-1}}}.
\end{flalign*}
\\
\\
{Finally, summing up all the non-zeros $\Gamma$ terms yields:}

\begin{flalign*}
\sum_{i,j}\Gamma_{i,j,i+j-2}+\sum_{i,j}\Gamma_{i,j,i}+\sum_{i,j}\Gamma_{i,j,j}
&=\sum_{i,j\geq1}(i+j-1)(i+j-2)p_ip_j\frac{\partial}{\partial p_{i+j-2}}&\\
 &- \sum_{i,j\geq1} (i+j-1)ip_{i+1} p_j\frac{\partial}{\partial p_{i+j-1}}\\
 &-\sum_{i,j\geq1} j(i+j-1) p_{j+1} p_i\frac{\partial}{\partial p_{i+j-1}}\\
&=\sum_{i,j\geq1}(i+j-1)(i+j-2)p_ip_j\frac{\partial}{\partial p_{i+j-2}}\\
&-\sum_{i,j\geq1} (i+j-2)(i-1)p_{i} p_j\frac{\partial}{\partial p_{i+j-2}}\\
&-\sum_{i,j\geq1} (j-1)(i+j-2) p_{j} p_i\frac{\partial}{\partial p_{i+j-2}}\\
&=\sum_{i,j\geq1}(i+j-2)(i+j-1-i+1-j+1)p_ip_j\frac{\partial}{\partial p_{i+j-2}}\\
&=\boxed{\sum_{i,j\geq1}(i+j-2)p_ip_j\frac{\partial}{\partial p_{i+j-2}}}.
\end{flalign*}

As a consequence

$$\Omega_2=(\al-1)\sum_{i\geq1}(i-1)ip_{i+1}\frac{\partial}{\partial p_{i-1}}+\sum_{i,j\geq1}(i+j-2)p_ip_j\frac{\partial}{\partial p_{i+j-2}}+\al \sum_{i,j\geq1}ijp_{j+i+2}\frac{\partial}{\partial p_{i}}\frac{\partial}{\partial p_{j}}$$

\section*{Acknowledgment}

Andrei L. Kanuninnkov is supported by the Russian Science Foundation, grant 16-11-10013.

\printbibliography

\end{document}